\documentclass[12pt,reqno]{amsart}
\usepackage{latexsym,amsmath,amsfonts,amssymb,amsthm,xcolor,enumitem,comment,hyperref,chngcntr,cleveref,booktabs,tabularx,makecell,thmtools,thm-restate,placeins,bm}

\allowdisplaybreaks

\newcommand*{\red}{\textcolor{red}}

\newcommand{\msubsection}[1]{\subsection{\texorpdfstring{#1}{e}}}
\newcommand{\msection}[1]{\section{\texorpdfstring{#1}{e}}}

\crefname{subsection}{subsection}{subsections}
\Crefname{subsection}{Subsection}{Subsections}

\def\pmod #1{\ ({\rm{mod}}\ #1)}
\def\Z{\mathbb Z}

\def\N{\mathbb N}
\def\M{\mathfrak M}
\def\m{\mathfrak m}
\def\L{\mathfrak L}

\def\R{\mathbb R}
\def\Q{\mathbb Q}

\def\a{{\mathfrak a}}
\def\b{{\mathfrak b}}

\def\pmod #1{\ ({\rm{mod}}\ #1)}

\def\1{{\bf{1}}}

\newcommand{\ci}[1]{\left[\!\left[#1\right]\!\right]}

\newtheorem{theorem}{Theorem}[section]
\newtheorem{lemma}[theorem]{Lemma}

\newtheorem{claim}[theorem]{Claim}
\newtheorem{corollary}[theorem]{Corollary}

\newtheorem*{HardysCloseIntro}{\protect\Cref{HardysCloseToPolysDensityResult}}
\newtheorem*{PolynomialDensityIntro}{Theorem~\ref{eroidfklsjoifks}}
\theoremstyle{definition}

\newtheorem*{acknowledgment}{Acknowledgment}
\theoremstyle{remark}
\newtheorem{remark}[theorem]{Remark}
\begin{document}
\title{An ergodic approach to equations of the form $\bm{x+y=\lfloor\alpha(n)\rfloor}$.}
\author{Vitaly Bergelson}
\address{Department of Mathematics, Ohio State University, Columbus, OH 43210, USA}
\email{vitaly@math.ohio-state.edu}
\author{Hao Pan}
\address{School of Applied Mathematics, Nanjing University of Finance and Economics,
Nanjing 210023, People’s Republic of China}
\email{haopan79@zoho.com}
\author{Saúl Rodríguez-Martín}
\address{Department of Mathematics, Ohio State University, Columbus, OH 43210, USA}
\email{rodriguezmartin.1@osu.edu}
\thanks{}

\begin{abstract}
Motivated by questions and results contained in \cite{ESS89,KS06}, we introduce an ergodic approach to additive problems concerning solutions of equations of the form
\[
x+y=\lfloor\alpha(n)\rfloor,\quad x,y\in A,n\in\N
\]
where \(\alpha(t)\) is a sufficiently regular function, such as a polynomial or a logarithmico-exponential function of polynomial growth (e.g. \(\alpha(t)=t^{5/2}+t^2\log t\)), and $A$ is either a given set of positive upper density, or a cell of a prescribed finite partition of $\N$. We give a complete description of the Hardy field functions of polynomial growth for which these partition and density problems are always solvable.

The ergodic approach also applies to the more general problem of finding $x\in A_1,y\in A_2$, where $A_1,A_2$ are two given sets of positive density.
This in turn allows us to treat equations of the form $kx+ly=\lfloor\alpha(n)\rfloor$, where $k,l\in\Z$. 

Some of our results also hold true when $A$ is a subset of the primes with positive relative upper density.
\end{abstract}

\maketitle

\makeatletter
\renewcommand{\l@subsection}{%
  \@tocline{2}{0pt}{2em}{5em}{}%
}
\makeatother
\tableofcontents

\msection{Introduction}
\label{SecIntro}

For a subset $A$ of $\N=\{1,2,\dots\}$, define its \emph{upper density} by
$$
\overline{d}(A):=\limsup_{N\to\infty}\frac{|A\cap[1,N]|}{N}.
$$
The classical Furstenberg-S\'ark\"ozy theorem \cite{Fu77,S78a,S78b,F81b} 
asserts that if $A$ is a subset of $\N$ with $\overline{d}(A)>0$ and $p\in\mathbb{Z}[x]$ satisfies $p(0)=0$, 
then there exist distinct $x,y\in A$ and $n\in\N$ such that 
$$x-y=p(n).$$

In \cite[Theorem 3 and Problem 2]{ESS89}, Erd\H os, S\'ark\"ozy and S\'os considered a variant of the Furstenberg-S\'ark\"ozy theorem which deals with the equation $x+y=p(n)$. They conjectured that for every finite coloring of the positive integers and every polynomial $p(n)\in\mathbb{Z}[x]$ which takes even values, there exist distinct $x$ and $y$ of the same color and $n\in\mathbb{N}$ such that $x+y=p(n)$.
This conjecture was confirmed by  Khalfalah and Szemer\'edi:
\begin{theorem}[{cf. \cite[Theorem 1.4]{KS06}}]\label{KSt}
Let $p(x)=a_dx^d+a_{d-1}x^{d-1}+\cdots+a_0$ be a nonconstant polynomial with $a_i\in\Z$ and $a_d>0$. 
Suppose that $p(n)$ is even for some $n\in\N$. Then for every finite coloring of $\N$, there exist monochromatic distinct $x$ and $y$ such that $x+y=p(n)$ for some $n\in\N$.
\end{theorem}

The equation $x+y=p(n)$ appearing in \Cref{KSt} is an example of an \emph{additive problem} where, given a sufficiently regular function $\alpha(x)$ (such as a real polynomial or, say, $x^2\log(x)$) and a subset $A\subseteq\mathbb{N}$, one looks for solutions of the equation 
\begin{equation}
\label{x+yisalphaofn}
x+y=\lfloor\alpha(n)\rfloor,
\end{equation} 
where $x,y\in A$, $n\in\N$.
In this paper, we develop an ergodic approach to problems of this kind. In particular, we will show that for a rather rich family of functions $\alpha(n)$, \Cref{x+yisalphaofn} can be solved for $x,y$ belonging to any given in advance set $A$ with $\overline{d}(A)>0$. As is often the case with Diophantine problems of a Ramsey-theoretic nature, one must account for the so-called \emph{local obstructions}.
For example, the divisibility condition in \Cref{KSt} is necessary, and the “coloristic” \Cref{KSt} does not admit a density version (for example, consider $p(n)=n^2$ and $A=\{3n+1;n\in\N\}$).
This being said, it is worth mentioning that our methods are versatile enough to produce a density result, \Cref{eroidfklsjoifks}, which has the coloristic \Cref{KSt} as its corollary.

We would also like to observe that the issue with the local obstruction disappears if, instead of integer polynomials, one considers real polynomials 
\begin{equation*}
p(x)=a_dx^d+a_{d-1}x^{d-1}+\cdots+a_0\in\R[x]
\end{equation*}
such that $a_d>0$ and $\frac{a_i}{a_d}\not\in\Q$ for some $1\leq i\leq d$. In that case, for all $A\subseteq\N$ such that $\overline{d}(A)>0$ there exist $x,y\in A,n\in\N$ such that $x+y=\lfloor p(n)\rfloor$ (this is a special case of \Cref{maindc} below).

We now introduce the class of functions $\alpha(x)$ considered throughout the paper; further background and references are provided in \Cref{SecHardy}. Let $\mathcal G$ be the ring of germs at $+\infty$ of continuous real-valued functions, with addition and multiplication defined pointwise.
By abuse of language, we will not distinguish between functions and their germs in $\mathcal G$.
A \emph{Hardy field} is a subfield of $\mathcal G$ that is closed under differentiation, and we denote the union of all Hardy fields by $\mathbf U$. A classical example is the Hardy field $\mathbf L$ of logarithmico-exponential functions introduced by Hardy in \cite{Har1,Har2}. It consists of the germs at $+\infty$ of real-valued functions obtained from real constants, $\exp x$, and $\log x$ by finitely many applications of addition, multiplication, division, and composition.

Examples of functions in Hardy fields include
$$
x^{\sqrt{2}},\ \pi x^2+e,\ x\log x+\sqrt{x},\ \log(e^x+x).
$$
It is a classical fact that every $\alpha\in\mathbf U$ has a limit in the extended real line $\mathbb R\cup\{-\infty,+\infty\}$ as $x\to+\infty$. Consequently, no nonconstant periodic function, such as $\sin x$ or $\cos x$, belongs to $\mathbf U$. Moreover, if $u_1$ and $u_2$ belong to the same Hardy field and $u_2\neq0$, then both $u_1-u_2$ and $u_1/u_2$ have limits in the extended real line. Finally, $\mathbf U$ is closed under differentiation and taking antiderivatives.

In this paper, we work with Hardy field functions \(\alpha(x)\in\mathbf U\) satisfying
\[
\lim_{x\to+\infty}\frac{\alpha(x)}{x^l}=0
\]for some \(l>0\). We will refer to such functions simply as \emph{sub-polynomials}.

Sub-polynomials appear naturally in various papers which deal with uniform distribution, additive combinatorics, and ergodic Ramsey theory (see, for example, \cite{B94, FW09, CKW10,BMR1,BMR2}).

Throughout this paper, we will be using the following notation:
\begin{align*}
\lfloor x\rfloor &:= \max\{k\in\Z:k\leq x\}
    &&\text{for the floor function},\\
\lceil x\rceil &:= \min\{k\in\Z:k\geq x\}
    &&\text{for the ceiling function},\\
[\![x]\!] &:= \left\lfloor x+\frac12\right\rfloor
    &&\text{for the nearest-integer function},\\
\{x\} &:= x-\lfloor x\rfloor
    &&\text{for the fractional part of \(x\)}.
\end{align*}

We start the presentation of our results with the following theorem, proved in \Cref{ProofOftexp,ProofOfmainp}, which naturally complements \Cref{KSt}.

\begin{theorem}\label{mainp}
Let \(\alpha(x)\) be a sub-polynomial such that
$\lim_{x\to+\infty}\alpha(x)=+\infty$. Suppose also that for all $p(x)\in\Q[x]$,
\begin{equation}\label{fginfty}
\lim_{x\to+\infty}|\alpha(x)-p(x)|=+\infty.
\end{equation}
Then for every finite coloring of $\mathbb{N}$, there are distinct $x,y$ of the same color and $n\in\mathbb{N}$ such that $x+y=\lfloor\alpha(n)\rfloor$.
\end{theorem}

We note that the density version of \Cref{mainp} does not hold in general (see \Cref{HardysCloseToPolysDensityResult}).

Our proof of \Cref{mainp} proceeds through two finitistic density results, \Cref{QuantitativeErgToComb} and \Cref{ColoringThmForClosestInteger}. 
In fact, as explained in \Cref{SecColorralphaSmall}, a minor adaptation of the same argument also yields the density result \Cref{eroidfklsjoifks}, stated below, which implies \Cref{KSt}.
\Cref{mainp} also holds true if one replaces $\lfloor\cdot\rfloor$ by either $\lceil\cdot\rceil$ or $[\![\cdot]\!]$, see \Cref{ColoringThmForClosestInteger}.

A necessary condition for the conclusion of \Cref{mainp} to hold is that $\lfloor\alpha(n)\rfloor$ be even for infinitely many $n\in\N$. 
Indeed, suppose that, for some \(M\in\N\), the sequence \(\lfloor\alpha(n)\rfloor\) takes no even values greater than \(M\). Color the integers greater than \(M\) according to their parity, and assign a distinct color to each integer in \([1,M]\). Then the equation
\begin{equation}
\label{r9weoifsldkiofedlsk}
x+y=\lfloor\alpha(n)\rfloor,\qquad x\neq y
\end{equation}
has no solution with $x,y$ having the same color and $n\in\mathbb{N}$. Note that this necessary condition is also satisfied under the hypotheses of \Cref{KSt}, since, if $p\in\Z[x]$ and $p(n)$ is even for some $n$, then $p(n+2k)$ is even for all $k\in\N$.

Conversely, combining \Cref{mainp,KSt} shows that this necessary condition is also sufficient for sub-polynomials; we include the details of this deduction at the end of \Cref{SecColorralphaSmall}, after the proofs of both results.
\begin{theorem}\label{mainp1}
Let $\alpha(x)$ be a sub-polynomial such that $\lim_{x\to+\infty}\alpha(x)=+\infty$. Then the following are equivalent:
\begin{enumerate}
    \item\label{mainp11} $\lfloor\alpha(n)\rfloor$ is even for infinitely many $n\in\N$.
    \item\label{mainp12}
    For every finite coloring of $\N$, there are distinct $x,y$ of the same color and \(n\in\N\) such that
\[
x+y=\lfloor\alpha(n)\rfloor.
\]
\end{enumerate}
\end{theorem}

We note that if $\lim_{x\to+\infty}|\alpha(x)-p(x)|=+\infty$ for all $p(x)\in\Q[x]$, then \(\alpha\) necessarily satisfies condition \ref{mainp11} of \Cref{mainp1}. Indeed, by \cite[Theorem 1.4]{B94}, the sequence $$
\left(\left\{\frac{\alpha(n)}{2}\right\}\right)_{n\in\N}$$ is dense in \([0,1)\). Therefore, there are infinitely many \(n\in\N\) for which $
\left\{\frac{\alpha(n)}{2}\right\}<\frac12$, and for every such \(n\), \(\lfloor\alpha(n)\rfloor\) is even.

We now turn to density analogues of \Cref{mainp1}.
We have already discussed above an example of local obstruction to the density version of \Cref{KSt}, involving the function $\alpha(n)=n^2$. The following theorem may be viewed as a far-reaching generalization of that example. 
Moreover, as we shall see in \Cref{maindc}, among sub-polynomials satisfying \(\alpha(x)/x\to+\infty\), \Cref{HardysCloseToPolysDensityResult} describes precisely those for which the corresponding density conclusion fails. 
\begin{theorem}
\label{HardysCloseToPolysDensityResult}
Let $\alpha(x)$ be an eventually positive sub-polynomial of the form 
\begin{equation}
\label{r9wedoslcdsds}
\alpha(x)=c p(x)+r(x),
\end{equation}
where $c>0$, $p\in\mathbb{Z}[x]$ has degree $\geq2$, and $\lim_{x\to+\infty}\frac{|r(x)|}{\log(x)}<+\infty$. Then there exists $A\subseteq\mathbb{N}$ such that $\overline{d}(A)>0$ and there are no $x,y\in A,n\in\mathbb{N}$ such that $x+y=\lfloor\alpha(n)\rfloor$.
\end{theorem}

\begin{remark}
The variant of \Cref{HardysCloseToPolysDensityResult} for polynomials $p(x)$ of degree at most $1$ has a more cumbersome formulation, and is treated in \Cref{SecObstructionsDensityTheorem} (see \Cref{Alphaclosetolinear}).
\end{remark}

The following theorem, which is a corollary of an ergodic result to be proved in \Cref{SecExpPolyRat}, demonstrates that one can obtain a density amplification of \Cref{KSt} by requiring $A$ to have positive density within a suitable arithmetic progression. 

\begin{theorem}\label{eroidfklsjoifks}
Let $\delta>0$ and let $p\in\mathbb{Z}[x]$ satisfy $\lim_{x\to+\infty}p(x)=+\infty$. Suppose that $p(n)$ is even for some $n\in\N$. Then there exist $b,d\in\N$ such that, for every $A\subseteq\mathbb{N}$ satisfying
\begin{equation*}
\overline{d}_{b+d\N}(A):=\limsup_N\frac{\left|A\cap\{b+d,b+2d,\dots,b+Nd\}\right|}{N}>\delta,
\end{equation*}
there exist distinct $x,y\in A$ such that $x+y=p(n)$ for some $n\in\mathbb{N}$.
\end{theorem}

Note that \Cref{KSt} follows immediately from \Cref{eroidfklsjoifks}. Indeed, it suffices to observe that given a $k$-coloring of $\mathbb N$ and $b,d\in\N$, some color class $A$ must satisfy $\overline{d}_{b+d\N}(A)\geq\frac{1}{k}$.

An analogue of \Cref{eroidfklsjoifks} also holds for every sub-polynomial satisfying the equivalent conditions in \Cref{mainp1}, although its precise formulation is more involved and is omitted. 

When $\alpha$ remains sufficiently far from every scalar multiple of a rational polynomial, we have the following density version of \Cref{mainp}.

\begin{theorem}\label{maindc}
Let $\alpha(x)$ be a sub-polynomial such that $
\lim_{x\to+\infty}\alpha(x)=+\infty$.
Suppose that, for every $p\in\Q[x]$ and every $c\in\R$, 
\begin{equation}\label{fglog}
\lim_{x\to+\infty}
\frac{|\alpha(x)-c p(x)|}{\log x}
=+\infty.
\end{equation}
Then, for every \(A\subseteq\N\) with \(\overline{d}(A)>0\), there exist distinct \(x,y\in A\) and some \(n\in\N\) such that
\[
x+y=\lfloor\alpha(n)\rfloor.
\]
\end{theorem}

Combining \Cref{HardysCloseToPolysDensityResult} and \Cref{maindc}, we obtain:

\begin{theorem}
For a sub-polynomial $\alpha(x)$  such that $
\lim_{x\to+\infty}\frac{\alpha(x)}{x}=+\infty$, the following are equivalent:
\begin{itemize}
    \item For every $p\in\Q[x]$ and every $c\in\R$, 
\begin{equation*}
\lim_{x\to+\infty}
\frac{|\alpha(x)-c p(x)|}{\log x}
=+\infty.
\end{equation*}
\item For every \(A\subseteq\N\) with \(\overline{d}(A)>0\), there exist distinct \(x,y\in A\) and some \(n\in\N\) such that
\[
x+y=\lfloor\alpha(n)\rfloor.
\]
\end{itemize}
\end{theorem}

\begin{remark}
\label{4refsdpossdlfoxc}
By a classical result of Boshernitzan \cite[Theorem 1.3]{B94}, if $\alpha$ is a sub-polynomial, then the sequence $
\bigl(\{\alpha(n)\}\bigr)_{n\in\N}$
is uniformly distributed modulo \(1\) (u.d.\ mod \(1\)) if and only if 
\[
\lim_{x\to+\infty}
\frac{|\alpha(x)-p(x)|}{\log x}
=+\infty
\]
for every $p\in\Q[x]$. Consequently, $\alpha$ satisfies \eqref{fglog} if and only if, for every $\kappa>0$, the sequence $
\bigl(\kappa\alpha(n)\bigr)_{n\in\N}$
is u.d. mod \(1\).
\end{remark}

\Cref{maindc} forms a natural counterpart of a result of Frantzikinakis and Wierdl, who in \cite[Theorem C]{FW09} established a variant of the Furstenberg--Sárközy theorem for sub-polynomials. More precisely, they proved the following.
\begin{theorem}
\label{FW09ThmC}
If a sub-polynomial $\alpha(x)$ satisfies
\begin{equation}\label{fgfw}
\lim_{x\to+\infty}|\alpha(x)-c p(x)|=+\infty
\end{equation}
for every $p\in\Q[x]$ and every $c\in\R$, then, for every set $A\subseteq\N$ with $\overline{d}(A)>0$, there exist distinct $x,y\in A$ and some $n\in\N$ such that
\begin{equation}
x-y=\lfloor\alpha(n)\rfloor.
\end{equation}
\end{theorem}

Thus, for sub-polynomials satisfying \(\alpha(x)/x\to+\infty\), condition \eqref{fgfw} is sufficient (but not necessary) for the equation \(x-y=\lfloor\alpha(n)\rfloor\) to have a solution in every set of positive upper density. By contrast, condition \eqref{fglog} is necessary and sufficient for the same conclusion to hold for \(x+y=\lfloor\alpha(n)\rfloor\).

We actually establish an extended version of \Cref{maindc}, which we state next. For \(A_1,A_2\subseteq\N\), define their \emph{joint upper density} by
\[
\overline{d}(A_1,A_2)
:=
\limsup_{N\to\infty}
\min_{1\leq i\leq 2}
\frac{|A_i\cap[1,N]|}{N}.
\]
Equivalently, \(\overline{d}(A_1,A_2)>0\) if and only if there exists a sequence of natural numbers \((N_j)_{j\in\N}\) tending to \(\infty\) such that
\[
\liminf_{j\to\infty}\frac{|A_1\cap[1,N_j]|}{N_j}>0
\qquad\text{and}\qquad
\liminf_{j\to\infty}\frac{|A_2\cap[1,N_j]|}{N_j}>0.
\]
Clearly, $
\overline{d}(A_1,A_2)>0$ implies that both \(\overline{d}(A_1)>0\) and \(\overline{d}(A_2)>0\). But the converse need not hold: for example, the sets 
\begin{equation*}
\textstyle A_1=\N\cap\bigcup_N[(4N)!,(4N+1)!],\qquad A_2=\N\cap\bigcup_N[(4N+2)!,(4N+3)!]
\end{equation*}
satisfy $\overline{d}(A_1)=\overline{d}(A_2)=1$, but $\overline{d}(A_1,A_2)=0$.
Still, if $A_1$ has positive lower density, i.e.
$$
\underline{d}(A_1):=\liminf_{N\to\infty}\frac{|A_1\cap[1,N]|}{N}>0,
$$
then for all $A_2$ with positive upper density we have $\overline{d}(A_1,A_2)>0$. 

\begin{theorem}\label{maind}
Let \(\alpha(x)\) be a sub-polynomial satisfying the hypotheses of \Cref{maindc}.
Then, for any $A_1,A_2\subseteq\N$ satisfying $
\overline{d}(A_1,A_2)>0$, 
there exist $n\in\N$ and distinct $x\in A_1$ and $y\in A_2$ such that
\[
x+y=\lfloor\alpha(n)\rfloor.
\]
\end{theorem}

\begin{remark}
In \Cref{maind} the assumption $
        \overline{d}(A_1,A_2)>0$
        cannot in general be replaced by the two separate assumptions $
        \overline{d}(A_1)>0$ and $\overline{d}(A_2)>0$ (see \Cref{RemarkMaind1}).
\end{remark}
\begin{remark}
One can show that, under the assumptions of \Cref{maind}, the set
        \[
        \bigl\{n\in\N:\text{there exist distinct }x\in A_1
        \text{ and }y\in A_2\text{ such that }
        x+y=\lfloor\alpha(n)\rfloor\bigr\}
        \]
        has positive upper density; see \Cref{GoodSumsHavePositiveDensity}. Moreover, given any $a<b$, one may replace the conclusion $
        x+y=\lfloor\alpha(n)\rfloor$
        by
        \[
        \alpha(n)-(x+y)\in(a,b);
        \]
        see \Cref{MorePreciseVersionOfmaind}. In particular, \Cref{maind} remains valid when $\lfloor\cdot\rfloor$ is replaced by either $\lceil\cdot\rceil$ or $[\![\cdot]\!]$.
\end{remark}

Note that, if $A\subseteq\N$ has positive upper density, then 
$\overline d(kA,\ell A)>0$
for every $k,\ell\in\N$, where $kA:=\{kx:x\in A\}$. Therefore, \Cref{maind} has the following immediate consequence.

\begin{theorem}
\label{LinearEquationCorollary}
Let $\alpha(x)$ be a sub-polynomial satisfying the hypotheses of \Cref{maindc}, let $k,l\in\N$, and let $A\subseteq\N$ have positive upper density. Then there exist $x,y\in A$ and $n\in\N$ such that
\[
kx+ly=\lfloor\alpha(n)\rfloor.
\]
\end{theorem}

It is natural to ask whether analogous density results hold for equations of the form $
kx-\ell y=\lfloor\alpha(n)\rfloor$.
The methods used to prove \Cref{maind} also give the following result, proved in \Cref{SecErgThm}.

\begin{restatable}{theorem}{LinearDifferenceEquationCorollary}
\label{LinearDifferenceEquationCorollary}
Let $\alpha(x)$ be a sub-polynomial satisfying the hypotheses of \Cref{maindc}, let $k,\ell\in\N$ satisfy $k\geq \ell$, and let $A\subseteq\N$ have positive upper density. Then there exist $x,y\in A$ and $n\in\N$ such that
\[
kx-\ell y=\lfloor\alpha(n)\rfloor.
\]
\end{restatable}

When $k=\ell$, \Cref{LinearDifferenceEquationCorollary} is a special case of \Cref{FW09ThmC}. One can adapt the proof of \Cref{LinearDifferenceEquationCorollary} to check that the conclusion of \Cref{LinearDifferenceEquationCorollary} holds for $k<\ell$ if $A$ has positive lower density. But the inequality $k\geq \ell$ is needed in general: if $k<\ell$ and $\alpha(x)\succ x$, one can find a set $A\subseteq\N$ with $\overline d(A)>0$ for which no $x,y\in A$ and $n\in\N$ satisfy $kx-\ell y=\lfloor\alpha(n)\rfloor$ (see \Cref{LinearDifferenceCounterexample}).

We now turn to the ergodic-theoretic result underlying \Cref{maind}. 

Let $(X,\mathcal{B},\mu,T)$ be a measure-preserving system (m.p.s.\footnote{A measure-preserving system (m.p.s.) is a quadruple $(X,\mathcal{B},\mu,T)$, where $(X,\mathcal{B},\mu)$ is a probability space and $T:X\to X$ is an invertible measure-preserving map, meaning that $T$ is measurable, has measurable inverse and $\mu(A)=\mu(T(A))$ for all $A\in\mathcal{B}$.}). The invertible measure-preserving map $T:X\to X$ induces a unitary operator $U_T:L^2(X,\mu)\to L^2(X,\mu)$ via the formula $
U_Tf:=f\circ T$. By the classical von Neumann ergodic theorem, for any $f\in L^2(X,\mu)$, we have
\begin{equation}\label{ergodicconverg}
\lim_{N\to\infty}\frac{1}{N}\sum_{n=1}^{N}U_T^nf
=P_T(f),
\qquad\text{ in }L^2(X,\mu),
\end{equation}
where $P_T:L^2(X,\mu)\to L^2(X,\mu)$ is the orthogonal projection onto the subspace of $U_T$-invariant elements of $L^2(X,\mu)$.

Let us observe that the convergence in \Cref{ergodicconverg} is not uniform over all m.p.s. and all functions in the unit ball of \(L^2\). In fact, for every $N_0\in\N$ and every standard\footnote{A m.p.s. is standard if $(X,\mathcal{B})$ is isomorphic to a compact metric space with its Borel $\sigma$-algebra.} aperiodic\footnote{A m.p.s. is aperiodic if $\mu(\{x\in X;T^nx= x\})=0$ for every $n\in\N$.} m.p.s. $(X,\mathcal{B},\mu,T)$, there is a set $A\in \mathcal{B}$ such that
\begin{equation}
\label{te9rordoifkfdjofd}
\left\|
\frac{1}{N}\sum_{n=1}^{N}1_{T^{-n}A}
-P_T(1_A)
\right\|_2
\geq 0.1
\end{equation}
for every $1\leq N\leq N_0$. 
Indeed, this follows from the Rohlin-tower construction in \Cref{RohlinTowersNonUniformity}.

On the other hand, as the following \Cref{tergodic} demonstrates, there is a strong uniformity if one considers the approximation of weighted ergodic averages along sub-polynomials satisfying \Cref{fglog} by the conventional ergodic averages.

\begin{theorem}\label{tergodic}
Let $\alpha(x)$ be a positive sub-polynomial such that, for every $\kappa>0$, the sequence $
\bigl(\kappa\alpha(n)\bigr)_{n\in\N}$
is u.d. mod $1$. Then there exists a sequence $w\colon\N\to(0,+\infty)$ with $w(N)\searrow0$ such that, for every $N\in\N$, every m.p.s. $(X,\mathcal{B},\mu,T)$, and every $f\in L^2(X)$, we have
\begin{equation}
\label{refdsolfdoslk}
\left\|
\frac{1}{\alpha(N)}\sum_{n=1}^N\alpha'(n)U_T^{\lfloor\alpha(n)\rfloor}f
-\frac{1}{\alpha(N)}\sum_{1\leq m\leq\alpha(N)}U_T^mf
\right\|_2
\leq w(N)\|f\|_2.
\end{equation}
\end{theorem}

To prove \Cref{maind}, we will use the following corollary of \Cref{tergodic}, which is obtained by applying \Cref{refdsolfdoslk} to $1_B$ (note that $\|1_B\|_{2}=\sqrt{\mu(B)}$), taking the inner product with $1_A$, and applying the Cauchy--Schwarz inequality.

\begin{corollary}
\label{4rewiodkl}
Let $\alpha(x)$ be a positive sub-polynomial such that, for every $\kappa>0$, the sequence $\bigl(\kappa\alpha(n)\bigr)_{n\in\N}$ is u.d. mod $1$. Then there is a sequence $w\colon\N\to(0,+\infty)$ with $w(N)\searrow0$ such that, for every $N\in\N$, every m.p.s. $(X,\mathcal{B},\mu,T)$, and all $A,B\in\mathcal{B}$, we have
\begin{multline}
\left|
\frac{1}{\alpha(N)}\sum_{n=1}^N\alpha'(n)
\mu\bigl(A\cap T^{-\lfloor\alpha(n)\rfloor}B\bigr)
-\frac{1}{\alpha(N)}\sum_{1\leq m\leq\alpha(N)}
\mu\bigl(A\cap T^{-m}B\bigr)
\right|
\\
\leq w(N)\sqrt{\mu(A)\mu(B)}.
\end{multline}
\end{corollary}

As was already mentioned above, we also have an ergodic result, \Cref{ErgodicThmIntegerPolys}, which is similar in spirit to \Cref{tergodic} and allows us to prove \Cref{KSt,mainp}. Since its formulation is somewhat more involved, we defer it to \Cref{ProofOfmainp}.

In the proof of \Cref{tergodic}, we derive explicit upper bounds for \(w(N)\), particularly under the assumption that
\[
\lim_{x\to+\infty}\frac{|\alpha(x)-p(x)|}{\log x}=+\infty
\qquad\text{for every }p\in\mathbb R[x].
\]
These bounds yield quantitative results such as the following, which we prove in \Cref{SecErgThm}.

\begin{restatable}{theorem}{AlphaFarFromPolyExplicit}
\label{AlphaFarFromPolyExplicit1.4}
Let $c\geq1$ be a non-integer, and set $
\delta:=\frac{2}{c\cdot2^{c+8}}$. 
Suppose a Hardy field function $\alpha(x)$ of degree $c$ satisfies $\lim_{x\to+\infty}\frac{\alpha(x)}{x^l}\in\{0,+\infty\}$ for all $l\in\mathbb{N}$. Then, for every sufficiently large $M\in\N$ and every pair of sets $A,B\subseteq\{1,\dots,M\}$ satisfying
\begin{equation*}
|A|\cdot|B|\geq\frac{M^2}{M^{\delta}}=M^{2-\delta},
\end{equation*}
there exist distinct $x\in A,y\in B$ and $n\in\N$ such that
\begin{equation*}
x+y=\lfloor\alpha(n)\rfloor.
\end{equation*}
\end{restatable}

See \cite[Theorem 3]{RS97} for a version of \Cref{AlphaFarFromPolyExplicit1.4} in the case $\alpha(x)=x^c$. In \Cref{SecErgThm} we deduce the following result from a variant of \Cref{AlphaFarFromPolyExplicit1.4}:

\begin{theorem}
\label{SumsOfPrimes}
Let $\alpha(x)$ be a positive sub-polynomial such that
\begin{equation}
\label{PrimeSubpolyCondition}
\lim_{x\to+\infty}\frac{|\alpha(x)-p(x)|}{(\log x)^m}=+\infty\textup{ for all }m>0,p\in\R[x].
\end{equation}
and let $\mathcal{P}\subseteq\N$ be the set of prime numbers. Suppose that $A,B\subseteq\mathcal{P}$ satisfy
\begin{equation}
\overline{d}_{\mathcal{P}}(A,B)
:=
\limsup_{N\to\infty}
\frac{\min(
|A\cap[1,N]|,
|B\cap[1,N]|)}{|\mathcal{P}\cap[1,N]|}
>0.
\end{equation}
Then there exist arbitrarily large $x\in A,y\in B$ and some $n\in\N$ such that
\[
x+y=\lfloor \alpha(n)\rfloor.
\]
\end{theorem}

\begin{remark}
Let $k,l\in\N$ and assume that a positive sub-polynomial $\alpha(x)$ satisfies \Cref{PrimeSubpolyCondition}, and $A\subseteq\mathcal{P}$ satisfies $\overline{d}_{\mathcal{P}}(A)>0$. Then we can adapt the proof of \Cref{SumsOfPrimes} to prove that, similarly to \Cref{LinearEquationCorollary,LinearDifferenceEquationCorollary}, there are $x,y\in A,n\in\N$ such that
\begin{equation*}
kx+ly=\lfloor\alpha(n)\rfloor,
\end{equation*}
and if $k\geq l$, then there are $x,y\in A,n\in\N$ such that
\begin{equation*}
kx-ly=\lfloor\alpha(n)\rfloor.
\end{equation*}
\end{remark}

\paragraph{\textbf{Structure of the paper.}}
In \Cref{SecHardy}, we review the necessary background on Hardy fields and describe the decomposition and classification of sub-polynomials used throughout the paper. In \Cref{SecErgThm}, we obtain a quantitative version of \Cref{maind} from \Cref{tergodic}, which we use to derive \Cref{AlphaFarFromPolyExplicit1.4,LinearDifferenceEquationCorollary}. 
In \Cref{ProofOftexp}, we prove \Cref{texp}, an exponential sum estimate that we then use to prove \Cref{tergodic} via a standard spectral argument.
In \Cref{ProofOfmainp}, we establish the integer-polynomial ergodic approximation result \Cref{ErgodicThmIntegerPolys}. We then use it to prove \Cref{mainp} for the sub-polynomials not covered by \Cref{maind}. This section also contains a proof of \Cref{eroidfklsjoifks}, and hence of \Cref{KSt}.
In \Cref{SecRemarksmaind}, we present counterexamples showing that several hypotheses in \Cref{maind,LinearDifferenceEquationCorollary} cannot, in general, be weakened. Finally, in \Cref{SecObstructionsDensityTheorem}, we prove \Cref{HardysCloseToPolysDensityResult,Alphaclosetolinear}, thereby completing the characterization of the sub-polynomial Hardy functions $\alpha(x)$ that satisfy the conclusion of \Cref{maindc}.

\msection{Asymptotic notation and preliminaries on Hardy fields}
\label{SecHardy}

Throughout the article, we use standard asymptotic notation. Let \(D\subseteq\mathbb{R}^n\), and let \(f\colon D\to\mathbb{C}\) and \(g\colon D\to(0,+\infty)\). We write
\begin{equation*}
f(x_1,\dots,x_n)\ll g(x_1,\dots,x_n)\quad\textup{or}\quad f(x_1,\dots,x_n)=O( g(x_1,\dots,x_n)),
\end{equation*}
if there is a universal constant $C>0$ (called `implied constant') such that 
\begin{equation}
\label{rwedsiolkds}
|f(x_1,\dots,x_n)|\leq Cg(x_1,\dots,x_n)
\end{equation}
 for all $(x_1,\dots,x_n)\in D$. 
When the implied constant $C$ in \Cref{rwedsiolkds} depends on additional objects or parameters, such as a Hardy field function \(\alpha\), we indicate this dependence by writing
\begin{equation*}
f(x_1,\dots,x_n)\ll_\alpha g(x_1,\dots,x_n)\quad\textup{or}\quad f(x_1,\dots,x_n)=O_\alpha( g(x_1,\dots,x_n)).
\end{equation*}
Given $a\in\mathbb{R}$ and one variable functions $f,g:(a,+\infty)\to(0,+\infty)$,  we write 
\begin{equation*}
f(x)\prec g(x)\quad\textup{or}\quad f(x)=o(g(x)),
\end{equation*}
if
$\frac{|f(x)|}{g(x)}\to0$ as $x\to+\infty$.

We now introduce basic facts about Hardy fields that will be needed in this paper; more details and discussion may be found in \cite{B04,B81,R83}.

Let \(f\) and \(g\) be real-valued continuous functions defined, respectively, on intervals \((a_f,+\infty)\) and \((a_g,+\infty)\). We say that \(f\) and \(g\) are \emph{eventually equal}, and write \(f\sim g\), if \(f(x)=g(x)\) for all sufficiently large \(x\). The resulting equivalence classes are called \emph{germs at infinity}, and we will usually identify a function with its germ. Under pointwise addition and multiplication, the set \(\mathcal{G}\) of germs at infinity forms a ring.

A Hardy field $\mathcal{H}$ is a subfield of $\mathcal{G}$ that is closed under differentiation: for every $f\in\mathcal{H}$, the derivative $f'(x)$ is defined for all sufficiently large $x$ and $f'\in\mathcal{H}$.

Every nonzero $f$ in a Hardy field $\mathcal{H}$ is eventually either positive or negative, since otherwise it could not have a multiplicative inverse. Applying this observation to $f'$ shows that every $f\in\mathcal{H}$ is eventually monotone and therefore has a limit at infinity,
\begin{equation*}
\lim_{t\to+\infty}f(t)\in\mathbb{R}\cup\{-\infty,+\infty\}.
\end{equation*}

The intersection $\mathbf{E}$ of all maximal Hardy fields (i.e. Hardy fields not strictly contained in another Hardy field) is closed under differentiation and integration and contains every polynomial in $\mathbb{R}[x]$, as well as the functions $\log(x)$ and $e^x$, see \cite[Section 1]{B81}. A function $\alpha(x)$ in a Hardy field is called a sub-polynomial if there exists $l>0$ such that $\lim_{x\to+\infty}\frac{\alpha(x)}{x^l}=0$; this limit always exists or equals $\pm\infty$, since $x^l\in\mathbf{E}$. For a sub-polynomial $\alpha\neq0$, we define its degree by
\begin{equation}
\label{DefDegree}
\deg(\alpha):=\inf\left\{l\in\mathbb{R};\lim_{x\to+\infty}\frac{|\alpha(x)|}{x^l}=0\right\}\in\mathbb{R}\cup\{-\infty\}.
\end{equation}
Equivalently, 
\begin{equation}
\label{Degree2ndDef}
\deg(\alpha)=\lim_{x\to+\infty}\frac{\log(\alpha(x))}{\log(x)}.
\end{equation}
Clearly, if $\deg(f)<\deg(g)$, then $\lim_{x\to+\infty}\frac{f(x)}{g(x)}=0$. 

\begin{lemma}
\label{LemmaLogHardyBehaviour}
Let $\alpha(x)>0$ be a sub-polynomial with $\deg(\alpha)\in\mathbb{R}$. Then, for every $\lambda>0$,
\begin{equation}
\label{LogHardyBehaviour}
\lim_{x\to+\infty}\frac{\alpha(\lambda x)}{\alpha(x)}=\lambda^{\deg(\alpha)}.
\end{equation}
\end{lemma}
\begin{proof}
By \cite[Theorem 5.3]{B81}, $\log(\alpha(x))$ is a Hardy field function. By \Cref{Degree2ndDef} we have
\begin{equation}
\label{LogsOfHardys}
\log(\alpha(x))=\deg(\alpha)\cdot\log(x)+r(x),
\end{equation}
for some Hardy field function $r(x)\prec\log(x)$. In particular, $r'(x)\prec\frac{1}{x}$, so $r(\lambda x)-r(x)\to0$ for every $\lambda>0$. Applying \Cref{LogsOfHardys} at $x$ and at $\lambda x$ yields
\begin{equation*}
\log(\alpha(\lambda x))-\log(\alpha(x))=\deg(\alpha)\log(\lambda)+o(1),
\end{equation*}
concluding the proof.
\end{proof}

\begin{lemma}
\label{LemDegreeDerivative}
Let $f$ be a positive sub-polynomial such that $\lim_{x\to+\infty}f(x)\in\{0,+\infty\}$. If $d:=\deg(f)\in\mathbb{R}$, then $\deg(f')=d-1$.
\end{lemma}
\begin{proof}
For every $l\neq0$, L'Hôpital's rule gives
\begin{equation*}
\lim_{x\to+\infty}
\frac{f(x)}{x^l}
=
\lim_{x\to+\infty}
\frac{f'(x)}{lx^{l-1}}.\qedhere
\end{equation*}
\end{proof}

The hypothesis $\lim_{x\to+\infty}f(x)\in\{0,+\infty\}$ is necessary; for example, consider $f(x)=1+x^{-2}$.

\begin{lemma}
\label{WeightedDensityToDensity}
Let $\alpha(x)\succ1$ be a positive sub-polynomial, and let $S\subseteq\N$ satisfy $\overline{d}(S)=0$. Then
\[
\lim_{N\to\infty}\frac{1}{\alpha(N)}
\sum_{\substack{1\leq n<N\\n\in S}}\alpha'(n)=0.
\]
\end{lemma}

\begin{proof}
By \Cref{LemDegreeDerivative}, $\deg(\alpha')\geq-1$. Since $\alpha'$ is eventually positive and monotone, \Cref{LogHardyBehaviour}, applied to $\alpha'$, gives a constant $K\geq1$ such that, for all sufficiently large $k\in\mathbb{N}$ and every $x\in[2^k,2^{k+1}]$,
\begin{equation*}
\frac{1}{K}\alpha'(2^k)\leq\alpha'(x)\leq K\alpha'(2^k).
\end{equation*}
Therefore, for every sufficiently large $k$ we have
\begin{align*}
\sum_{\substack{2^k\leq n<2^{k+1}\\n\in S}}\alpha'(n)
&\leq K\left|S\cap[2^k,2^{k+1})\right|\alpha'(2^k)\\
&\leq K^2\frac{\left|S\cap[2^k,2^{k+1})\right|}{2^k}
\sum_{2^k\leq n<2^{k+1}}\alpha'(n)\\
&\leq2K^2\frac{\left|S\cap[2^k,2^{k+1})\right|}{2^k}
\bigl(\alpha(2^{k+1})-\alpha(2^k)\bigr).
\end{align*}
Since $\overline{d}(S)=0$, it follows that
\begin{equation*}
\lim_{k\to\infty}
\frac{\displaystyle\sum_{\substack{2^k\leq n<2^{k+1}\\n\in S}}\alpha'(n)}
{\alpha(2^{k+1})-\alpha(2^k)}=0.
\end{equation*}
Summing the numerator and denominator for all $k\in\N$ and taking the limit when $k\to\infty$ gives
\begin{equation*}
\lim_{k\to\infty}
\frac{\displaystyle\sum_{\substack{1\leq n<2^k\\n\in S}}\alpha'(n)}
{\alpha(2^k)}=0,
\end{equation*}
and the result follows.
\end{proof}

We call a sub-polynomial $r(x)$ \textit{special} if
$$\lim_{x\to+\infty}\frac{|r(x)|}{x^l}\in\{0,+\infty\}\textup{ for all }l=1,2,\dots.$$

For example, $x\log x+x$ and $1+\frac1x$ are special, whereas $x+\frac{x}{\log x}$ is not.

The following decomposition of Hardy functions is closely related, though not identical, to that in \cite[Lemma 3.1]{BKQW07}. We include the proof for completeness.

\begin{lemma}
\label{PolyAndSpecialParts}
Every sub-polynomial $\alpha(x)\in\mathcal{H}$ is either special or admits a unique representation, for some integers $1\leq l\leq d$, of the form
\begin{align*}
\alpha(x)&=p_\alpha(x)+r_\alpha(x)\\
&=\a_dx^d+\a_{d-1}x^{d-1}+\cdots+\a_lx^l+r_\alpha(x),
\end{align*}
where $p_\alpha(x)\in\mathbb{R}[x]$, $\a_d,\a_l\neq0$, and $r_\alpha(x)\prec x^l$ is special. We call $p_\alpha$ and $r_\alpha$ the \textit{polynomial part} and the \textit{special part} of $\alpha$, respectively. If $\alpha$ is special, we set $p_\alpha=0$ and $r_\alpha=\alpha$.
\end{lemma}

\begin{proof}
We first prove existence. If $\alpha$ is special, there is nothing to prove. Otherwise, for some positive integer $d$, we have $\a_d:=\lim_{x\to+\infty}\frac{\alpha(x)}{x^d}\in\mathbb{R}\setminus\{0\}$. Setting $r_d(x):=\alpha(x)-\a_dx^d$, we obtain
\begin{equation*}
\alpha(x)=\a_dx^d+r_d(x),
\end{equation*}
where $r_d(x)\prec x^d$. If $r_d$ is special, the construction is complete. Otherwise, there exist an integer $1\leq d_1<d$ and a nonzero coefficient $\a_{d_1}$ such that, after setting $r_{d_1}(x):=r_d(x)-\a_{d_1}x^{d_1}$, we have
\begin{equation*}
\alpha(x)=\a_dx^d+\a_{d_1}x^{d_1}+r_{d_1}(x),
\end{equation*}
where $r_{d_1}(x)\prec x^{d_1}$. Repeating this procedure produces a strictly decreasing sequence of positive integer exponents, so it terminates after at most $d$ steps, yielding the required polynomial--special decomposition of $\alpha$.

Uniqueness of the representation is immediate when $\alpha$ is special. If $\alpha$ is not special, let $\alpha(x)=p_\alpha(x)+r_\alpha(x)$ be a decomposition as in \Cref{PolyAndSpecialParts}, and let $k$ be the smallest natural number such that $r_\alpha(x)\prec x^k$. Given any $p\in\mathbb{R}[x]$ with $p\neq p_\alpha$, set $r(x):=\alpha(x)-p(x)$. There are two cases:
\begin{itemize}
    \item If $r_\alpha\prec p_\alpha-p$, then $r(x)=r_\alpha(x)+(p_\alpha-p)(x)$ is not special. Hence $\alpha(x)=p(x)+r(x)$ is not a decomposition of the form described in \Cref{PolyAndSpecialParts}.
    \item If $0\neq p_\alpha-p\prec r_\alpha$, then $\lim_{x\to+\infty}\frac{r(x)}{r_\alpha(x)}=1$. In the identity $p=p_\alpha+(p-p_\alpha)$, the polynomial $p_\alpha$ contains only terms of degree at least $k$, whereas $p-p_\alpha$ has degree less than $k$. Thus, the lowest-degree term of $p$ grows slower than $r(x)$, so again $\alpha(x)=p(x)+r(x)$ is not a decomposition of the required form.\qedhere
\end{itemize}
\end{proof}

For example, if $\alpha(x)=2x^2+x+x\log x$, then $p_\alpha(x)=2x^2$ and $r_\alpha(x)=x\log x+x$.
Every sub-polynomial $\alpha(x)$ belongs to one of the following three families, according to its polynomial and special parts:

\begin{table}[!htbp]
\centering
\begin{tabularx}{\linewidth}{
>{\bfseries}p{0.9cm}
>{\raggedright\arraybackslash}p{4.2cm}
>{\raggedright\arraybackslash}X}
Type & & Examples\\
\midrule
\raisebox{0pt}[0pt][0pt]{\hypertarget{CaseI}{}}(I)&
$r_\alpha(x)\succ\log x$
&
$x^2+\sqrt{x},\ x^{\sqrt2},\ (\log x)^2$
\\[0.6em]
\raisebox{0pt}[0pt][0pt]{\hypertarget{CaseII}{}}(II)&
\begin{tabular}[t]{@{}l@{}}
$r_\alpha(x)\ll_\alpha\log x$\\
$p_\alpha(x)\notin\bigcup_{c\in\mathbb R}c\cdot\Q[x]$
\end{tabular}
&
$x^2+\sqrt2\,x, \sqrt3\,x^3+ex+\exp\left(\sqrt{\log\log x}\right)$
\\[2em]
\raisebox{0pt}[0pt][0pt]{\hypertarget{CaseIII}{}}(III)&
\begin{tabular}[t]{@{}l@{}}
$r_\alpha(x)\ll_\alpha\log x$\\
$p_\alpha(x)\in\bigcup_{c\in\mathbb R}c\cdot\Q[x]$
\end{tabular}
&
$\pi x^3+\pi x,\ 2x^2+\sqrt{\log x},\ \log\log x,\ 23$
\end{tabularx}
\end{table}
\FloatBarrier

By \cite[Theorem 1.3]{B94}, $\alpha(x)$ of type ~\hyperlink{CaseI}{(I)} or type ~\hyperlink{CaseII}{(II)} if and only if $\{\kappa\cdot \alpha(x)\}$ is u.d. modulo $1$ for every $\kappa\in\R\setminus\{0\}$; equivalently, $\alpha(x)$ satisfies the hypotheses of \Cref{maind}.

\msection{An Ergodic Approach}\label{SecErgThm}
In this section, we show how the ergodic approximation result \Cref{tergodic} leads to the additive-combinatorial results stated in the introduction. We begin with a quantitative result, \Cref{QuantitativeErgToComb}, which turns \Cref{tergodic} into solutions of
\[
x+y=\lfloor\alpha(n)\rfloor.
\]
We then use it to deduce \Cref{maind}, its quantitative version \Cref{AlphaFarFromPolyExplicit1.4}, and the refinements \Cref{GoodSumsHavePositiveDensity,MorePreciseVersionOfmaind}. Finally, we explain how the same ergodic input gives the difference version \Cref{LinearDifferenceEquationCorollary}.

\begin{theorem}
\label{QuantitativeErgToComb}
Let $\alpha(x)\succ1$ be a positive sub-polynomial. Suppose that, for some $N\in\mathbb{N}$ and $\varepsilon>0$, the following estimate holds for every m.p.s. $(\mathcal X,\mathcal{B},\mu,T)$ and every $A,B\in\mathcal{B}$:
\begin{multline}
\label{r9eopfsdlds}
\left|
\frac{1}{\alpha(N)}\sum_{n=1}^N\alpha'(n)
\mu\bigl(A\cap T^{-\lfloor\alpha(n)\rfloor}B\bigr)
-\frac{1}{\alpha(N)}\sum_{1\leq m\leq\alpha(N)}
\mu\bigl(A\cap T^{-m}B\bigr)
\right|
\\
\leq \varepsilon\sqrt{\mu(A)\mu(B)}.
\end{multline}
Then, for any subsets $A_1,A_2\subseteq\left[1,\frac{\alpha(N)}{2}\right)\cap\mathbb{N}$ satisfying
\begin{equation}
\label{5trdgf9oopfd}
|A_1|\cdot|A_2|>\varepsilon^{2}\cdot\alpha(N)^2,
\end{equation}
there exist $x\in A_1$, $y\in A_2$, and an integer $n\in\{1,\dots,N\}$ such that $x+y=\lfloor\alpha(n)\rfloor$.
\end{theorem}

\begin{proof}
Let $P>\lceil\alpha(N)\rceil$ be a natural number.
We consider the m.p.s. $(\mathcal{X},\mathcal{B},\mu,T)$, where $\mathcal{X}=\mathbb{Z}/P\mathbb{Z}$, $\mathcal{B}=2^{\mathcal{X}}$, $\mu$ is the normalized counting measure, and $Tx=x-1$ for all $x$. For $n\in\mathbb{Z}$ and $A\subseteq\mathbb{Z}$, let $\overline{n}$ and $\overline{A}$ denote their respective images in $\mathbb{Z}/P\mathbb{Z}$. Note that for any fixed $1\leq m\leq\alpha(N)$, the sets $A_1$ and $m-A_2$ are contained in an interval of length smaller than $P$, which implies that
$$|A_1\cap(m-A_2)|=\left|\overline{A_1}\cap\overline{m-A_2}\right|.$$
Therefore,
\begin{align*}
\frac{1}{\alpha(N)}\sum_{1\leq m\leq\alpha(N)}
\mu\bigl(\overline{A_1}\cap T^{-m}\overline{-A_2}\bigr)
&=
\frac{1}{\alpha(N)\cdot P}
\sum_{1\leq m\leq\alpha(N)}\left|A_1\cap (m-A_2)\right|\\
&=\frac{|A_1|\cdot|A_2|}{\alpha(N)\cdot P}.
\end{align*}
The last equality follows by double counting. Indeed, the map
$$
A_1\times A_2\to\sqcup_{1\leq m\leq\alpha(N)}(A_1\cap (m-A_2))\times\{m\},\qquad(a_1,a_2)\mapsto(a_1,a_1+a_2)
$$
is a bijection.
On the other hand, \Cref{r9eopfsdlds} gives
\begin{align*}
\bigg|
\frac{1}{\alpha(N)}\sum_{1\leq m\leq\alpha(N)}
\mu\bigl(\overline{A_1}\cap T^{-m}\overline{-A_2}\bigr)
&-\frac{1}{\alpha(N)}\sum_{n=1}^N\alpha'(n)
\mu\bigl(\overline{A_1}\cap T^{-\lfloor\alpha(n)\rfloor}\overline{-A_2}\bigr)
\bigg|\\
&\leq\varepsilon
\sqrt{\mu(\overline{A_1})\mu(\overline{-A_2})}\\
&=\varepsilon\frac{\sqrt{|A_1|\cdot|A_2|}}{P}.
\end{align*}
Therefore, combining the two previous equations we obtain
\begin{align}
\label{rweofisdioirfsdkl}
\frac{1}{\alpha(N)\cdot P}\sum_{n=1}^N\alpha'(n)|A_1\cap(\lfloor\alpha(n)\rfloor-A_2)|
&\geq\frac{|A_1|\cdot|A_2|}{\alpha(N)\cdot P}-\varepsilon\frac{\sqrt{|A_1|\cdot|A_2|}}{P}>0,
\end{align}
where the last inequality follows from \Cref{5trdgf9oopfd}. Thus, $A_1\cap(\lfloor\alpha(n)\rfloor-A_2)\neq\varnothing$ for some $1\leq n\leq N$. Hence there exist $x\in A_1$ and $y\in A_2$ such that $x+y=\lfloor\alpha(n)\rfloor$.
\end{proof}

\begin{corollary}
\label{QuantitativeErgToComb1set}
Let $\alpha(x)\succ1$ be a positive sub-polynomial. Suppose that, for some $N\in\mathbb{N}$ and $\varepsilon>0$, \Cref{r9eopfsdlds} holds for every m.p.s. $(\mathcal X,\mathcal{B},\mu,T)$ and every $A,B\in\mathcal{B}$.
Then, for any subsets $A_1,A_2\subseteq\left[1,\frac{\alpha(N)}{2}\right)\cap\mathbb{N}$ with $|A_i|\geq3$ and
\begin{equation}
|A_1|\cdot|A_2|>9\varepsilon^2 \alpha(N)^2,
\end{equation}
there exist distinct $x\in A_1,y\in A_2$ and an integer $n\in\{1,\dots,N\}$ such that $x+y=\lfloor\alpha(n)\rfloor$.
\end{corollary}

\begin{proof}
Choose disjoint subsets $A_1^*\subseteq A_1,A_2^*\subseteq A_2$ such that $|A_i^*|\geq\frac{|A_i|}{3}$. 
Applying \Cref{QuantitativeErgToComb} to $A_1^*$ and $A_2^*$ gives distinct $x\in A_1$ and $y\in A_2$ and an integer $n\in\{1,\dots,N\}$ such that $x+y=\lfloor\alpha(n)\rfloor$.
\end{proof}

\begin{remark}
\label{Distinctsummands}
A slight refinement of the argument in the proof of \Cref{QuantitativeErgToComb} shows that given any $\lambda>1$, in \Cref{QuantitativeErgToComb1set} we can replace the expression $9\varepsilon^2 \alpha(N)^2$ by $\lambda\varepsilon^2 \alpha(N)^2$ for sufficiently large $N$.
\end{remark}

Finally, we prove \Cref{maind} from \Cref{tergodic}.
\begin{proof}[Deducing \Cref{maind} from \Cref{tergodic}]
Let $\delta<\overline{d}(A_1,A_2)$ be positive. For arbitrarily big $N\in\mathbb{N}$, we have $\left|A_1\cap \left[1,\frac{\alpha(N)}2\right)\right|,\left|A_2\cap \left[1,\frac{\alpha(N)}2\right)\right|\geq\delta|\alpha(N)|$, so applying \Cref{QuantitativeErgToComb1set} with $\varepsilon=\frac{\delta}{3}$ and using \Cref{4rewiodkl}, we are done.
\end{proof}

We can now prove \Cref{AlphaFarFromPolyExplicit1.4}.

\AlphaFarFromPolyExplicit*

\begin{remark}
The Furstenberg--Sárközy theorem also admits a quantitative version. Indeed, Sárközy \cite{S78b} proved that there exists a constant $c>0$ such that, for all sufficiently large $N$, every set $A\subseteq[1,N]$ satisfying
\[
|A|\geq \frac{N}{(\log N)^c}
\]
contains distinct elements $x,y$ for which $x-y$ is a perfect square. 
\end{remark}

\begin{proof}[Proof of \Cref{AlphaFarFromPolyExplicit1.4}]
We will need a quantitative form of \Cref{4rewiodkl}, which follows from an estimate proved independently later in the paper, at the end of \Cref{SecErgralphaBig}. More precisely, \Cref{4ref8dsi8e9roikl} (with $\lambda_0(x)=x^{1/2}$ and $\sigma_0=2^{-(c+5)}$), together with \Cref{rwedsfoikl,4rewiodkl}, implies that, for every sufficiently large $N$, every m.p.s. $(\mathcal X,\mathcal B,\mu,T)$, and all $B_1,B_2\in\mathcal B$, one has
\begin{multline*}
\left|
\frac{1}{\alpha(N)}\sum_{n=1}^N\alpha'(n)
\mu\bigl(B_1\cap T^{-\lfloor\alpha(n)\rfloor}B_2\bigr)
-\frac{1}{\alpha(N)}\sum_{1\leq m\leq \alpha(N)}
\mu\bigl(B_1\cap T^{-m}B_2\bigr)
\right|
\\
\leq N^{-\frac{1}{2^{c+7}}}\sqrt{\mu(B_1)\mu(B_2)}.
\end{multline*}
Now let $M\in\mathbb{N}$ be sufficiently large, choose $N$ such that $\alpha(N-1)/2\leq M<\alpha(N)/2$, and set $\varepsilon=N^{-\frac{1}{2^{c+7}}}$. If $M$ is sufficiently large, then $\alpha(N)\leq3M$, so
\begin{equation*}
9\varepsilon^2\cdot\alpha(N)^2
<100M^2N^{-\frac{1}{2^{c+6}}}
<\frac{M^2}{M^\delta},
\end{equation*}
where the last inequality follows from comparing degrees. Therefore, by \Cref{QuantitativeErgToComb1set}, every pair of sets $A,B\subseteq\{1,\dots,M\}$ satisfying $|A|\cdot|B|\geq M^2/M^\delta$ contains distinct elements $x\in A$ and $y\in B$ such that $x+y=\lfloor\alpha(n)\rfloor$ for some $n\in\N$.
\end{proof}

\Cref{SumsOfPrimes} follows immediately from the prime number theorem and the following variant of \Cref{AlphaFarFromPolyExplicit1.4}:

\begin{theorem}
Let $k>0$ and let $\alpha(x)$ be a positive Hardy field function satisfying \Cref{PrimeSubpolyCondition}.
Then, for every sufficiently large $M\in\N$ and every pair of sets $A,B\subseteq\{1,\dots,M\}$ satisfying
\begin{equation*}
|A|\cdot|B|\geq\frac{M^2}{\log(M)^k},
\end{equation*}
there exist distinct $x\in A,y\in B$ and $n\in\N$ such that
\begin{equation*}
x+y=\lfloor\alpha(n)\rfloor.
\end{equation*}
\end{theorem}
\begin{proof}
The proof is essentially the same as that of \Cref{AlphaFarFromPolyExplicit1.4}, once we observe that \Cref{PrimeSubpolyCondition} implies that the function $\lambda_0(x)$ from \Cref{4ref8dsi8e9roikl} satisfies $\lambda_0(x)\succ\log(x)^l$ for every $l\in\N$. This bound on $\lambda_0(x)$ follows from the fact that, for $m,n\geq1$, the $n^{\textup{th}}$ derivative of $\log(x)^m$ has the same asymptotic growth, up to a nonzero multiplicative constant, as $\frac{\log(x)^{m-1}}{x^n}$.
\end{proof}

\begin{theorem}
\label{GoodSumsHavePositiveDensity}
Suppose that $\alpha(x)$ is a positive sub-polynomial such that, for every $\kappa>0$, the sequence $\bigl(\kappa\alpha(n)\bigr)_{n\in\N}$ is u.d. modulo $1$, and let $A_1,A_2\subseteq\mathbb{N}$ satisfy $\overline{d}(A_1,A_2)>0$. Then the set
\begin{equation*}
S:=\left\{n\in\mathbb{N}:x+y=\lfloor\alpha(n)\rfloor\textup{ for some distinct }x\in A_1\textup{ and }y\in A_2\right\}
\end{equation*}
has positive upper density.
\end{theorem}

\begin{proof}
The first part of the argument is similar to the proof of \Cref{QuantitativeErgToComb}. By replacing $A_1$ and $A_2$ with suitable subsets if necessary, we may assume that $A_1\cap A_2=\varnothing$. The positive joint upper density of $A_1$ and $A_2$ allows us to choose $\delta>0$ and arbitrarily large $N$ such that, on setting
\[
A_{i,N}:=A_i\cap\left[1,\frac{\alpha(N)}2\right)
\qquad\text{and}\qquad
P:=\lceil\alpha(N)\rceil+1,
\]
we have $|A_{i,N}|\geq\delta P$ for $i=1,2$. For each such sufficiently large $N$, define the cyclic m.p.s. as in the proof of \Cref{QuantitativeErgToComb}, with $A_{i,N}$ in place of $A_i$. By \Cref{4rewiodkl}, we may apply \Cref{rweofisdioirfsdkl} with $\varepsilon=\delta/10$. Writing
\[
q_N:=\frac{\sqrt{|A_{1,N}|\cdot|A_{2,N}|}}{P}\geq\delta,
\]
we obtain
\begin{align*}
\frac{1}{\alpha(N)}\sum_{n=1}^N\alpha'(n)\frac{|A_{1,N}\cap(\lfloor\alpha(n)\rfloor-A_{2,N})|}{P}
&\geq\frac{P}{\alpha(N)}q_N^2-\varepsilon q_N\\
&\geq q_N(q_N-\varepsilon)>\frac{\delta^2}{2}.
\end{align*}
Moreover, $0\leq |A_{1,N}\cap(\lfloor\alpha(n)\rfloor-A_{2,N})|/P\leq1$ for all $1\leq n\leq N$, and this quotient vanishes whenever $n\notin S$. It follows that, for infinitely many $N$,
\begin{equation}
\label{redsoilkldkclk}
\frac{1}{\alpha(N)}\sum_{\substack{1\leq n<N\\n\in S}}\alpha'(n)\geq\frac{\delta^2}{3}.
\end{equation}

The conclusion now follows immediately from \Cref{WeightedDensityToDensity} and \Cref{redsoilkldkclk}.
\end{proof}

\begin{theorem}
\label{MorePreciseVersionOfmaind}
Let $\alpha(x)$ satisfy the hypotheses of \Cref{maind}. Then, for every $A_1,A_2\subseteq\N$ with $\overline{d}(A_1,A_2)>0$ and every pair of real numbers $a<b$, there exist distinct $x\in A_1$, $y\in A_2$, and $n\in\mathbb{N}$ such that $a<\alpha(n)-(x+y)<b$.
\end{theorem}

\begin{proof}
Since $\alpha(x)+c$ satisfies the hypotheses of \Cref{maind} for every $c\in\mathbb{R}$, it suffices to prove that, for every $k\in\mathbb{N}$, there exist distinct $x\in A_1$, $y\in A_2$, and $n\in\mathbb{N}$ such that
\[
0\leq\alpha(n)-(x+y)<\frac{1}{k}.
\]
Indeed, $k\alpha(x)$ satisfies the hypotheses of \Cref{maind}, and $\overline{d}(kA_1,kA_2)>0$, where $kA_i:=\{ka:a\in A_i\}$. Hence there exist distinct $x\in A_1$, $y\in A_2$, and $n\in\mathbb{N}$ such that
\[
kx+ky=\lfloor k\alpha(n)\rfloor.
\]
Therefore, $0\leq k\alpha(n)-k(x+y)<1$. Dividing by $k$ gives the required inequality.
\end{proof}

\LinearDifferenceEquationCorollary*

\begin{proof}[Proof sketch]
This proof closely follows that of \Cref{QuantitativeErgToComb}. For $N\in\N$, set
\[
A_N:=A\cap\left[1,\frac{\alpha(N)}{k}\right].
\]
In order to prove that there are infinitely many pairs $x,y\in A$ and $n\in\N$ such that $\ell y+\lfloor\alpha(n)\rfloor=kx$, it suffices to prove that
\begin{equation}
\label{4refdsrewds}
\limsup_N\frac{1}{\alpha(N)}\sum_{n=1}^N\alpha'(n)\frac{|(\ell A_N+\lfloor\alpha(n)\rfloor)\cap kA_N|}{\alpha(N)}>0.
\end{equation}
By \Cref{PositiveAverageLinearDifferences} below, $\overline{d}(A)>0$ implies
\begin{equation*}
\limsup_N\frac{1}{\alpha(N)}\sum_{m=1}^{\lfloor\alpha(N)\rfloor}\frac{|(\ell A_N+m)\cap kA_N|}{\alpha(N)}>0.
\end{equation*}
Moreover, by \Cref{4rewiodkl}, applied as in the proof of \Cref{QuantitativeErgToComb}, we have
\begin{equation*}
\limsup_N
\left|
\frac{1}{\alpha(N)}\sum_{m=1}^{\lfloor\alpha(N)\rfloor}\frac{|(\ell A_N+m)\cap kA_N|}{\alpha(N)}
-
\frac{1}{\alpha(N)}\sum_{n=1}^N\alpha'(n)\frac{|(\ell A_N+\lfloor\alpha(n)\rfloor)\cap kA_N|}{\alpha(N)}
\right|
=0.
\end{equation*}
So \Cref{4refdsrewds} follows.
\end{proof}

\begin{lemma}
\label{PositiveAverageLinearDifferences}
Let $A\subseteq\N$ have positive upper density, and let $k,\ell\in\N$ satisfy $k\geq \ell$. If
\[
A_N:=A\cap\left[1,\frac{N}{k}\right],
\]
then
\[
\limsup_N\frac{1}{N}\sum_{m=1}^{N}\frac{|(\ell A_N+m)\cap kA_N|}{N}>0.
\]
\end{lemma}

This lemma is false in general when $k<\ell$. For instance, one may take $A$ to be a union of intervals of the form $[(1-\varepsilon)N!,N!]$, with $\varepsilon>0$ small enough.

\begin{proof}
The definition of upper density gives $\varepsilon>0$ and infinitely many $N$ such that $|A_N|\geq\varepsilon N$. For such $N$, the sum
\[
\sum_{m=1}^{N}|(\ell A_N+m)\cap kA_N|
\]
counts the pairs $(x,y)\in A_N^2$ such that $kx-\ell y\in[1,N]$. Every pair with $x>y$ has this property, because $k\geq \ell$ and $x\leq N/k$. Hence, for infinitely many $N$,
\[
\frac{1}{N}\sum_{m=1}^{N}\frac{|(\ell A_N+m)\cap kA_N|}{N}
\geq \frac{|A_N|(|A_N|-1)}{2N^2}
\geq \frac{\varepsilon^2}{4},
\]
and the desired conclusion follows.
\end{proof}

\msection{Proof of \Cref{tergodic}}
\label{ProofOftexp}

Our proof of \Cref{tergodic} proceeds through the following equivalent statement. Set $e(x)=e^{2\pi ix}$ for $x\in\R$.
 
\begin{theorem}\label{texp}
Let $\alpha(x)$ be a positive sub-polynomial such that, for every $\kappa>0$, the sequence $\bigl(\kappa\alpha(n)\bigr)_{n\in\N}$ is u.d.\ mod $1$. Then there exists a sequence $w\colon\N\to(0,+\infty)$ with $w(N)\searrow0$ such that
\begin{equation}\label{4refdsoireodfsi}
\left|
\frac{1}{\alpha(N)}\sum_{1\leq n<N}\alpha'(n)
e\bigl(\lfloor\alpha(n)\rfloor\theta\bigr)
-\frac{1}{\alpha(N)}\sum_{1\leq m<\alpha(N)}e(m\theta)
\right|
<w(N)
\end{equation}
for every $\theta\in\R$ and every $N\in\N$.
\end{theorem}

We now explain why \Cref{tergodic} and \Cref{texp} are equivalent. In one direction,
\Cref{texp} follows from \Cref{tergodic} by applying the latter, for each \(\theta\in\mathbb R\), to the m.p.s. \(\mathbb T=\mathbb R/\mathbb Z\) with Lebesgue measure, the translation \(T:x\mapsto x+\theta\), and the function \(f(z)=e(z)\). 
To prove the opposite direction, we will show now that \Cref{texp} actually implies the following \Cref{tunitary}, a stronger form of \Cref{tergodic} for general unitary operators (applying \Cref{tunitary} to the Hilbert space $L^2(X,\mu)$ and the operator $U=U_T$ immediately yields \Cref{tergodic}). 

\begin{theorem}\label{tunitary}
Let $\alpha(x)$ be a positive sub-polynomial such that, for every $\kappa>0$, the sequence $\bigl(\kappa\alpha(n)\bigr)_{n\in\N}$ is u.d. modulo $1$. Then there is a positive sequence $w(N)\searrow0$ such that, for every $N\in\mathbb{N}$, every Hilbert space $\mathcal{H}$, every unitary operator $U:\mathcal{H}\to \mathcal{H}$, and every $f\in \mathcal{H}$,
$$
\bigg\|\frac{1}{\alpha(N)}\sum_{1\leq n<N}\alpha'(n)\cdot U^{\lfloor\alpha(n)\rfloor}f-\frac1{\alpha(N)}\sum_{1\leq m<\alpha(N)}U^{m}f\bigg\|\leq w(N)\|f\|.
$$
\end{theorem}

\begin{proof}[Proof of \Cref{tunitary} from \Cref{texp}]
Let $\mathcal{H}$ be a Hilbert space, let $U:\mathcal{H}\to \mathcal{H}$ be a unitary operator, and let $f\in \mathcal{H}$. Since the function $n\mapsto\langle U^nf,f\rangle$ is positive definite, Herglotz's theorem (cf. \cite[Chapter 1.8]{He83}) gives
$$
\langle U^nf,f\rangle=\int_0^1e(n\theta)d\nu(\theta)
$$
for some finite positive measure $\nu$ on $[0,1)$. In particular, $\|f\|^2=\int_0^1d\nu$. It follows that, for every finitely supported sequence of complex coefficients $(a_n)_{n\in\mathbb{Z}}$,
\begin{align*}
\left\|\sum_na_nU^nf\right\|^2&=\left\langle\sum_na_nU^nf,\sum_ma_mU^mf\right\rangle\\
&=\sum_{n,m}a_n\overline{a_m}\left\langle U^nf,U^mf\right\rangle\\
&=\sum_{n,m}a_n\overline{a_m}\left\langle U^{n-m}f,f\right\rangle\\
&=\sum_{n,m}a_n\overline{a_m}\int_0^1e((n-m)\theta)d\nu(\theta)\\
&=\int_0^1\sum_{n,m}a_ne(n\theta)\cdot\overline{a_me(m\theta)}d\nu(\theta)\\
&=\int_0^1\left|\sum_na_ne(n\theta)\right|^2d\nu(\theta).
\end{align*}

Thus, we have
\begin{align}
&\bigg\|\frac{1}{\alpha(N)}\sum_{1\leq n<N}\alpha'(n)\cdot U^{\lfloor\alpha(n)\rfloor}f-\frac1{\alpha(N)}\sum_{1\leq m<\alpha(N)}U^mf\bigg\|^2\notag\\
=&\int_0^1\bigg|\frac1{\alpha(N)}\sum_{1\leq n<N}\alpha'(n)e\left(\lfloor\alpha(n)\rfloor\theta\right)-
\frac1{\alpha(N)}\sum_{1\leq m<\alpha(N)}e(m\theta)\bigg|^2d\nu(\theta)\label{UsingHerglotz}\\
\leq&w(N)^2\int_0^1d\nu(\theta)=w(N)^2\cdot\|f\|^2.\notag\qedhere
\end{align}
\end{proof}

\begin{remark}
\label{rwedsfoikl}
The reduction of \Cref{tergodic,4rewiodkl} to  \Cref{texp} shows that the same sequence $w(N)$ works for all of them.
\end{remark}

\begin{remark}
In \cite[Lemma 7.3]{BKQW07}, the authors study exponential averages closely related to those appearing in \Cref{4refdsoireodfsi}. They use estimates for these averages to prove pointwise ergodic theorems along certain Hardy-field sequences (see \cite[Theorems 3.4, 3.5, 3.8]{BKQW07}).
\end{remark}

\msubsection{Proof of \Cref{texp} when $\alpha(x)\prec x$}
\label{SectexpSlowerThanx}
We now prove that for all sub-polynomials $1\prec\alpha(x)\prec x$, there exists a positive sequence $w(N)\searrow0$ such that
\begin{equation}
\label{4refdsoireodfsi2}
\left|\frac1{\alpha(N)}\sum_{1\leq n<N}\alpha'(n)e\left(\lfloor\alpha(n)\rfloor\theta\right)-
\frac1{\alpha(N)}\sum_{1\leq n< \alpha(N)}e(n\theta)\right|<w(N)
\end{equation}
for every $\theta\in\R,N\in\mathbb{N}$.

We will use that $\alpha(x+1)-\alpha(x)\leq\alpha'(x)\leq\alpha(x)-\alpha(x-1)$ for big enough $x$. After multiplying the left-hand side of \Cref{4refdsoireodfsi2} by $\alpha(N)$, and setting $M=\lfloor\alpha(N)\rfloor$, we can rewrite it as
\begin{equation}
\label{refsdpo9pfdsol}
\sum_{1\leq m<M}e(m\theta)\left(1-\sum_{\substack{n\in\mathbb{N}\\m\leq\alpha(n)<m+1}}\alpha'(n)\right)+O_\alpha\left(1\right),
\end{equation}
where the error $O_\alpha(1)$ comes from two sources. The first is the contribution of the terms $\alpha'(n)e(\lfloor\alpha(n)\rfloor\theta)$ for which $\alpha(n)\leq1$; the second is the contribution of the terms $\alpha'(n)e(M\theta)$ for which $M\leq\alpha(n)<\alpha(N)$, whose norm is at most
$$
\sum_{M\leq\alpha(n)<\alpha(N)}|\alpha'(n)|
\leq\sum_{M\leq\alpha(n)<\alpha(N)}|\alpha(n)-\alpha(n-1)|\leq2.
$$
On the other hand, for all sufficiently large $m\in\mathbb{N}$, and letting $N_m=\left\lceil\alpha^{-1}(m)\right\rceil$, we have $1-\sum_{m\leq\alpha(n)<m+1}\alpha'(n)\ll\alpha'(N_m-1)$, because by the mean value theorem,
\begin{align*}
1-2\alpha'(N_{m}-1)&\leq\alpha\left(N_{m+1}\right)-\alpha\left(N_m-1\right)\\&=\sum_{m\leq\alpha(n)<m+1}\alpha(n)-\alpha(n-1)\\
&\leq\sum_{m\leq\alpha(n)<m+1}\alpha'(n)\\
&\leq \sum_{m\leq\alpha(n)<m+1}\alpha(n+1)-\alpha(n)\\
&=\alpha(N_{m+1}+1)-\alpha(N_{m})\\
&\leq 1+2\alpha'(N_{m}-1)
\end{align*}
Using $\lim_{m\to\infty}\alpha'(N_m-1)=0$ together with (\ref{refsdpo9pfdsol}), we conclude the proof:
\begin{equation*}
\sum_{1\leq n<N}\alpha'(n)e\left(\lfloor\alpha(n)\rfloor\theta\right)-
\sum_{1\leq n< \alpha(N)}e(n\theta)\ll_\alpha\sum_{m=2}^{\lfloor\alpha(N)\rfloor}\alpha'(N_m-1)\prec \alpha(N).
\end{equation*}

\begin{remark}
If $1\prec\alpha(x)\prec x$ is a sub-polynomial, then $\alpha'(x)<\frac{\alpha(x)}{x}$ for all sufficiently large $x$. Indeed, $\frac{\alpha'(x)}{\alpha(x)}$ is a Hardy function, and if $\frac{\alpha'(x)}{\alpha(x)}\geq\frac{1}{x}$ for all sufficiently large $x$, then integrating would contradict $\ln(x)-\ln(\alpha(x))\to+\infty$. Thus, in this case we can make the estimate above more precise:
\begin{equation*}
\sum_{m=2}^{\lfloor\alpha(N)\rfloor}\alpha'(N_m-1)
\leq
\sum_{m=2}^{\lfloor\alpha(N)\rfloor}\frac{\alpha(N_m-1)}{N_m-1}
\leq 
\sum_{m=2}^{\lfloor\alpha(N)\rfloor}\frac{m}{N_m-1}.
\end{equation*}
Note that, if $\deg(\alpha)<\frac{1}{2}$, the sum $\sum_{m=2}^{\infty}\frac{m}{N_m-1}$ is finite, so we can let the function $w(N)$ from \Cref{4refdsoireodfsi} be $\frac{C_\alpha}{\alpha(N)}$ for some constant $C_\alpha$.
\end{remark}

\msubsection{Some exponential sum estimates}
\label{SecPrelims}

In this section we collect some preliminary results and lemmas which will be needed in \Cref{SecErgralphaBig,SecErgralphaSmall,SecColorralphaSmall}.

We will use summation by parts throughout this article, in the following form.

\begin{lemma}[Summation by parts]\label{SumParts}
Let $(a_n),(B_n)$, $n\in\mathbb{N}$, be sequences of complex numbers, and denote $b_n=B_{n+1}-B_n$. If $X<Y$ are natural numbers,
\begin{equation*}
\sum_{X\leq n<Y}a_nb_n=a_{Y}B_{Y}-a_XB_X-
\sum_{X\leq n< Y}B_{n+1}(a_{n+1}-a_{n}).
\end{equation*}
\end{lemma}

The following application of summation by parts will be particularly useful.

\begin{lemma}\label{ChangingCoefs}
Let $N\in\mathbb{N}$. Suppose $\beta:(0,+\infty)\to\mathbb{R}$ is monotone and $(z_n)_{n\in\mathbb{N}}$ is a sequence of complex numbers. Also let $1\leq L\leq N$ and let $\lambda\in(0,1]$. Suppose that for all $1\leq Y\leq L$,
\begin{equation*}
\left|\sum_{N-Y\leq n<N}z_n\right|\leq\lambda L.
\end{equation*}
Then for all $1\leq Y\leq L$ we have
\begin{equation*}
\left|\sum_{N-Y\leq n<N}\beta(n)z_n\right|\leq2\lambda L\max\big(|\beta(N)|,|\beta(N-Y)|\big).
\end{equation*}
\end{lemma}

\begin{proof} 
For each $N-L\leq m\leq N$ let $\displaystyle B_m=-\sum_{m\leq n<N}z_n$, so that $z_m=B_{m+1}-B_m$ and $|B_m|\leq\lambda L$ for all $m$. For $1\leq Y\leq L$, using summation by parts we obtain
\begin{align*}
\left|\sum_{N-Y\leq n<N}\beta(n)z_n\right|&=\left|-\beta(N-Y)B_{N-Y}-\sum_{N-Y\leq m< N}B_{m+1}(\beta(m+1)-\beta(m))\right|\\
&\leq\lambda L\left(|\beta(N-Y)|+\sum_{N-Y\leq m<N}|\beta(m+1)-\beta(m)|\right)\\
&\leq2\lambda L\max(|\beta(N)|,|\beta(N-Y)|).\qedhere
\end{align*}
\end{proof}

Let $(a_1,\dots,a_n)$ denote the greatest common divisor of $a_1,\dots,a_n\in\mathbb{N}$.

\begin{lemma}[Rational polynomials]
\label{Va97Thm7.1}
Fix $d\in\mathbb{N}$ and $\varepsilon>0$. Let $p(n)=a_1n+\cdots+a_dn^d\in\mathbb{Z}[x]$ and $Q\in\mathbb{N}$ satisfy $(Q,a_1,\dots,a_d)=1$. Then
\begin{equation*}
\frac1Q\sum_{r=1}^Q e\left(\frac{p(r)}Q\right)\ll_{d,\varepsilon} Q^{-\frac1d+\varepsilon}.
\end{equation*}
\end{lemma}

\begin{proof}
See \cite[Theorem 7.1]{Va97}. The implied constant only depends on $d$ and $\varepsilon$, as indicated in the `Notation' section of \cite{Va97}.
\end{proof}

For $x\in\R$, let $\|x\|= |x-[\![x]\!]|$ be the distance from $x$ to the nearest integer.

\begin{lemma}\label{DerivEstimates}
Let $X,Y\in\mathbb{R},Y>0$, and let $f:[X,X+Y]\to\mathbb{R}$ have continuous derivative.
\begin{enumerate}[label=(\roman*)]
    \item\label{1DerivEstimates} Suppose that $X,Y$ are natural numbers, $f'$ is monotone on $[X,X+Y]$ and for some $\lambda\in\left(0,\frac{1}{2}\right)$ and for all $x\in[X,X+Y]$,
$$
0<\lambda\leq\|f'(x)\|
$$
Then
\begin{equation}\label{KL}
\left|\sum_{X\leq n\leq X+Y}e\left(f(n)\right)\right|\leq\frac{2}{\lambda}.
\end{equation}
\item \label{dDerivEstimates} Suppose that $d\geq2$, $X,Y\in\mathbb N$, that $f$ has continuous $d$-th derivative on $[X,X+Y]$, and for some constants $\lambda>0,h\geq1$ we have
$$
\lambda\leq\left|f^{(d)}(x)\right|\leq h\lambda
$$
on $[X,X+Y]$. Then
\begin{equation}\label{VC}
\left|\sum_{X\leq n\leq X+Y}e\left(f(n)\right)\right|<21
hY\left(\lambda^{\frac1{2^d-2}}+Y^{-\frac2{2^d}}+(Y^d\lambda)^{-\frac2{2^d}}\right).
\end{equation}
\end{enumerate}
\end{lemma}

\begin{proof}
\Cref{1DerivEstimates} is the Kuzmin-Landau inequality (see the proof of Theorem 2.1 in \cite{GK91}). \Cref{dDerivEstimates} is a version of \cite[Satz 4]{C29} which, in translation from German, reads as follows:

Let $a<b$ be integers, let $k\geq 2$ be an integer, let $\varkappa=2^k$, and let
$f$ be a real-valued function on the interval $[a,b]$ which is $k$ times
differentiable, possibly only in $(a,b)$. Suppose that, for some $r>0$, we have 
$|f^{(k)}(t)|\geq r$ for all $t$. If
\[
R=\frac{1}{b-a}\left|f^{(k-1)}(b)-f^{(k-1)}(a)\right|,
\]
then
\begin{equation}
\label{3r9efdsoclcds}
\left|\sum_{x=a}^{b} e^{2\pi i f(x)}\right|
<
21(b-a)\left\{
\left(\frac{r}{R^2}\right)^{-\frac{1}{\varkappa-2}}
+\left(r(b-a)^k\right)^{-\frac{2}{\varkappa}}
+\left(\frac{r(b-a)}{R}\right)^{-\frac{2}{\varkappa}}
\right\}.
\end{equation}
We now apply this estimate to the function $f$ in \ref{dDerivEstimates}, setting $k=d,\varkappa=2^d,a=X,b=X+Y,r=\lambda$ and 
$$R=\frac{1}{b-a}|f^{(d-1)}(b)-f^{(d-1)}(a)|=\left|\frac{1}{b-a}\int_a^bf^{(k)}(x)dx\right|,$$ 
so that $\lambda\leq R\leq h\lambda$. Thus, applying \Cref{3r9efdsoclcds} we obtain
\begin{align*}
\left|\sum_{X\leq n\leq X+Y}e(f(n))\right|&<21(b-a)\left(
\left(\frac{r}{R^{2}}\right)^{-1/(\varkappa - 2)}
+ \left(r(b - a)^{k}\right)^{-2/\varkappa}
+ \left(\frac{r(b - a)}{R}\right)^{-2/\varkappa}
\right)\\
&\leq21Y\left(
\left(\frac{\lambda}{(h\lambda)^{2}}\right)^{-1/(2^d - 2)}
+ \left(\lambda Y^{d}\right)^{-2/2^d}
+ \left(\frac{\lambda Y}{h\lambda}\right)^{-2/\varkappa}
\right)\\
&<21
hY\left(\lambda^{\frac1{2^d-2}}+Y^{-\frac2{2^d}}+(Y^d\lambda)^{-\frac2{2^d}}\right).\qedhere
\end{align*}
\end{proof}

\begin{lemma}[{Passing from $\lfloor\alpha(n)\rfloor$ to $\alpha(n)$}]\label{fit} Let $X,Y,H\in\mathbb N$ with \(X<Y\) and \(H\geq3\), let \(\theta\in(0,1)\), and let \((\alpha(n))_{n\in\mathbb N}\) be a sequence of real numbers. Then there is a constant $c(\theta)\in\mathbb{C}$ such that $|c(\theta)|\leq\|\theta\|$ and
\begin{align*}
\sum_{X\leq n< Y}e\left(\lfloor\alpha(n)\rfloor\theta\right)=
&c(\theta)\sum_{|h|\leq H}\frac{1}{h+\theta}\sum_{X\leq n< Y}e\left((h+\theta)\alpha(n)\right)\notag\\
&+O\left(\log^2 
H\sup_{\substack{1\leq h\leq H^2\\
h\in\mathbb{N}}}\bigg|\sum_{X\leq n< Y}e(h \alpha(n))\bigg|+
\frac{(Y-X)\log^2H}{H}\right),
\end{align*}
where the implicit constant in the \(O\)-term can be taken to be \(10^3\).
\end{lemma}

\begin{proof} 
This is a version of Inequality (3.15) in \cite{CKW10}, which is given with a short sketch of the proof. We include the details of the proof for convenience of the reader. We first need two claims.

\begin{claim}
\label{Claim4refdo}
For all \(x\in\mathbb{R}\), \(\theta\in(0,1)\), and \(H\in\mathbb N\) with \(H\geq3\), there is a constant \(c(\theta)\in\mathbb C\), with \(|c(\theta)|\leq\|\theta\|\), such that
\begin{equation}
\label{refd9soerfds9oi}
e(-\theta\{x\})=c(\theta)\sum_{|h|\leq H}\frac{e(hx)}{h+\theta}+O\left(\frac{\log(H)}{1+H\|x\|}\right),
\end{equation}
where the implied constant can be taken to be \(10\).
\end{claim}
\begin{proof}
We may assume $x\not\in\mathbb{Z}$. We use the Fourier decomposition of the function $x\mapsto e(-\theta\{x\})$, which implies that for $x\not\in\mathbb{Z}$,
\begin{equation}
\label{r9efiodsdfiosklc}
e(-\theta\{x\})=\lim_{H\to\infty}\sum_{|h|\leq H}\frac{1-e(-(h+\theta))}{2\pi i(h+\theta)}e(hx)=c(\theta)\lim_{H\to\infty}\sum_{|h|\leq H}\frac{e(hx)}{h+\theta},
\end{equation}
where $c(\theta)=\frac{1-e(-\theta)}{2\pi i}$, so that $|c(\theta)|\leq\|\theta\|$ (as the function $\theta\mapsto e(\theta)$, $\theta\in\mathbb{R}$, is $2\pi$-Lipschitz).
If $\|x\|\leq H^{-1}$, then \Cref{refd9soerfds9oi} follows from the triangle inequality:
\begin{align*}
\left|e(-\theta\{x\})-c(\theta)\sum_{|h|\leq H}\frac{e(hx)}{h+\theta}\right|&\leq
1+\left|c(\theta)\sum_{-1\leq h\leq 1}\frac{e(hx)}{h+\theta}\right|+\left|c(\theta)\sum_{2\leq|h|\leq H}\frac{e(hx)}{h+\theta}\right|\\
&\leq 1+3\frac{|c(\theta)|}{\|\theta\|}+\frac{1}{2}\sum_{2\leq |h|\leq H}\left|\frac{1}{h+\theta}\right|\\
&\leq 4+\log(H)\leq\frac{10\log(H)}{1+H\|x\|}.
\end{align*}
If $\|x\|\geq H^{-1}$, we check that
\begin{equation}
\label{ref9do9reiofd}
\lim_{L\to\infty}\left|\sum_{H<h\leq L}\frac{e(hx)}{h+\theta}\right|\leq5\frac{\log(H)}{1+H\|x\|}.
\end{equation}
The negative terms of the Fourier series, with $h<-H$, are handled similarly.
We use summation by parts, \Cref{SumParts}, with $a_n=\frac{1}{n+\theta}$ and $B_n=\sum_{j=1}^{n-1}e(jx)$, so that $|B_n|\leq\frac{1}{\|x\|}$ for all $n$, and take limits when $L\to\infty$, to obtain
\begin{align*}
\left|\sum_{h>H}\frac{e(hx)}{h+\theta}\right|&\leq\frac{|B_{H+1}|}{H+1+\theta}+\sum_{h>H}|B_{h+1}|\left|\frac{1}{h+1+\theta}-\frac{1}{h+\theta}\right|\\
&\leq\frac{1}{H\|x\|}+\frac{1}{H\|x\|}\leq\frac{5\log(H)}{1+H\|x\|}.\qedhere
\end{align*}
\end{proof}

\begin{claim}
\label{Claim4refdo2}
For $H\geq3$, the Fourier expansion of $x\mapsto\frac{1}{1+H\|x\|}$ has the form
\begin{equation}
\frac{1}{1+H\|x\|}=\sum_{h\in\mathbb{Z}}b_he(hx),\textup{ with }|b_h|\leq2\frac{H\log(H)}{H^2+h^2}.
\end{equation}
\end{claim}

\begin{proof}
Let us see how to obtain the bound on $b_h$. For \(h\neq0\), integration by parts gives
$$b_h=\int_{-1/2}^{1/2}\frac{e(-hx)}{1+H|x|}dx=\frac{H}{2\pi i h}\left(\int_{-1/2}^0\frac{e(-hx)}{(1-Hx)^2}dx-\int_{0}^{1/2}\frac{e(-hx)}{(1+Hx)^2}dx\right).$$
We focus on the integral from $0$ to $\frac{1}{2}$; again by integration by parts,
$$
\left|\int_{0}^{1/2}\frac{e(-hx)}{(1+Hx)^2}dx\right|\leq\frac{2}{|h|}+\left|\int_0^{1/2}\frac{-2He(-hx)}{h(1+Hx)^3}dx\right|\leq\frac{2}{|h|}+\frac{1}{|h|}=\frac{3}{|h|},$$
where the last inequality uses the following bound:
$$
\int_0^{1/2}\frac{H}{(1+Hx)^3}dx=\int_1^{1+H/2}\frac{1}{u^3}du\leq1.$$
We conclude that $|b_h|\leq\frac{H}{h^2}$, which implies the bound we want when $|h|\geq H$. If $|h|\leq H$, then we conclude as follows:
\begin{equation*}
|b_h|
\leq
\int_{-1/2}^{1/2}
\frac{1}{1+H|x|}dx=
\frac{2}{H}\log(1+H/2)
\leq\frac{2\ln(H)}{H}\leq2\frac{H\log(H)}{H^2+h^2}.\qedhere
\end{equation*}
\end{proof}
We now return to the proof of \Cref{fit}. Using \Cref{Claim4refdo},
\begin{align*}
\sum_{X\leq n< Y}e\left(\lfloor\alpha(n)\rfloor\theta\right)&=\sum_{X\leq n< Y}e\left(\alpha(n)\theta\right)e(-\{\alpha(n)\}\theta)
\end{align*}
\begin{align*}
&=
\sum_{X\leq n< Y}e\left(\alpha(n)\theta\right)\left(c(\theta)\sum_{|h|\leq H}\frac{e(h\alpha(n))}{h+\theta}\right)+O\left(\sum_{X\leq n<Y}\frac{\log(H)}{1+H\|\alpha(n)\|}\right)\\
&=
c(\theta)\sum_{|h|\leq H}
\frac{1}{h+\theta}\sum_{X\leq n<Y}e((h+\theta)\alpha(n))+O\left(\log(H)\sum_{X\leq n<Y}\frac{1}{1+H\|\alpha(n)\|}\right).
\end{align*}
It remains to use \Cref{Claim4refdo2} to bound the error term. By \Cref{Claim4refdo2},
\begin{align*}
&\sum_{X\leq n<Y}\frac{1}{1+H\|\alpha(n)\|}
=
\sum_{X\leq n<Y}\sum_{h\in\mathbb{Z}}b_he(h\alpha(n))\leq\sum_{h\in\mathbb{Z}}|b_h|\cdot\left|\sum_{X\leq n<Y}e(h\alpha(n))\right|\\
&\ll
\sum_{h\in\mathbb{Z}}\frac{H\log(H)}{H^2+h^2}\left|\sum_{X\leq n<Y}e(h\alpha(n))\right|\\
&\leq\frac{\log(H)}{H}(Y-X)+\sum_{1\leq|h|\leq H^2}\frac{H\log(H)}{H^2+h^2}\left|\sum_{X\leq n<Y}e(h\alpha(n))\right|\\
&\quad+\sum_{|h|>H^2}\frac{H\log(H)}{h^2}(Y-X)\\
&\leq 3\frac{\log(H)}{H}(Y-X)+H\log(H)\left(\sum_{1\leq|h|\leq H^2}\frac{1}{H^2+h^2}\right)\sup_{1\leq h\leq H^2}\left|\sum_{X\leq n<Y}e(h\alpha(n))\right|.
\end{align*}
The remaining elementary estimate is
\begin{equation*}
\sum_{1\leq|h|\leq H^2}\frac{1}{H^2+h^2}\leq\sum_{|h|\leq H}\frac{1}{H^2}+\sum_{H\leq|h|\leq H^2}\frac{1}{h^2}\leq\frac{3}{H}+\frac{2}{H}\leq\frac{5}{H}.
\end{equation*}
Substituting this into the previous expression gives the claimed error term.\qedhere
\end{proof}

In the next lemma, sums of the form $\sum_{x\leq n<y}$, where $x,y\in\R$, are always over the natural numbers $n$ in the interval $[x,y)$.

\begin{lemma}\label{major}
For any sub-polynomial $\alpha(x)\gg_\alpha x$ there is $N_0=N_0(\alpha)$ such that, for all $N\geq N_0$, $1\leq L\leq \frac{N}{2}$ and $\theta\in\mathbb{R}$,
\begin{align}
\label{reodsil9oerwifdslk}
\sum_{N-L\leq n<N}\alpha'(n)e\left(\alpha(n)\theta\right)&-\sum_{\alpha(N-L)\leq m<\alpha(N)}e(m\theta)\\
&\ll
L|\theta|(\alpha'(N)^2+1)+L|\alpha''(N)|+\alpha'(N)+1.\notag
\end{align}
In particular, if $\alpha(x)=\mathfrak{a}_1x+r_\alpha(x)$ with $\mathfrak{a}_1>0$ and $r_\alpha(x)\prec x$, then
$$
\sum_{N-L\leq n<N}\alpha'(n)e\left(\alpha(n)\theta\right)-\sum_{\alpha(N-L)\leq m<\alpha(N)}e(m\theta)\ll 
L|\theta|(\mathfrak{a}_1^2+1)+\mathfrak{a}_1+1.
$$
\end{lemma}

\begin{proof} 
We may assume $L$ is an integer; otherwise change $L$ to $\lfloor L\rfloor$ and add an error term $O(\alpha'(N)+1)$.
Write $\alpha^\Delta(x)=\alpha(x+1)-\alpha(x)$, so that by the mean value theorem and the hypotheses on $\alpha$, we have for big enough $N-L\leq n<N$,
\begin{equation*}
|\alpha^\Delta(n)-\alpha'(n)|\leq\max_{N-L\leq x\leq N}|\alpha''(x)|\ll\max\left(|\alpha''(N)|,\frac{1}{N}\right).
\end{equation*}
Thus, we can change $\alpha'(n)$ to $\alpha^\Delta(n)$ in \Cref{reodsil9oerwifdslk}, with error $O(L|\alpha''(N)|+1)$. Now, as $x\mapsto e(x)$ is $2\pi$-Lipschitz, if $n$ is sufficiently large and $\alpha(n)\leq x\leq\alpha(n+1)$, then 
\begin{equation*}
|e(\alpha(n)\theta)-e(\lfloor x\rfloor\theta)|\leq
2\pi|\theta|\cdot\left|\alpha(n)-\lfloor x\rfloor\right|
\ll
(\alpha'(n)+1)|\theta|.
\end{equation*}
So, for all $N\geq N_0(\alpha)$, we have
\begin{align*}
&\left|\sum_{N-L\leq n<N}\alpha^\Delta(n)e\left(\alpha(n)\theta\right)-\sum_{\alpha(N-L)\leq m<\alpha(N)}e(m\theta)\right|\\
\leq&
\left|\sum_{N-L\leq n<N}\alpha^\Delta(n)e\left(\alpha(n)\theta\right)-\int_{\alpha(N-L)}^{\alpha(N)}e(\lfloor x\rfloor\theta)dx\right|+2\\
=&2+\left|\sum_{N-L\leq n<N}\left(\int_{\alpha(n)}^{\alpha(n+1)}e(\alpha(n)\theta)-e(\lfloor x\rfloor\theta)dx\right)\right|\\
\ll&1+\sum_{N-L\leq n<N}\alpha^\Delta(n)\cdot(\alpha'(n)+1)|\theta|\\
\ll&1+L|\theta|\alpha'(N)(\alpha'(N)+1)
\ll1+L|\theta|(\alpha'(N)^2+1)
\end{align*}
This proves the general estimate. The stated special case follows immediately.
\end{proof}

The remaining results in this section will come into play starting in \Cref{SecErgralphaSmall}.

\begin{lemma}[Weyl inequality]
\label{4refds9oiref9doi}
For any polynomial $p(x)=\a_1x^d+\cdots+\a_dx\in\mathbb{R}[x]$, and for any $X,Y,a,q\in\mathbb{N}$ such that $(a,q)=1$ and $\left|\a_1-\frac{a}{q}\right|\leq\frac{1}{q^2}$, we have
\begin{equation}
\label{43re9fsd9e4rdiooi}
\sum_{X\leq n< X+Y}e(p(n))\ll_{d,\varepsilon}Y^{1+\varepsilon}\left(\frac1q+\frac1Y+\frac{q}{Y^d}\right)^{\frac{1}{2^{d-1}}}.
\end{equation}
\end{lemma}
\begin{proof}
It follows from \cite[Lemma 2.4]{Va97} applied to $\phi(x)=p(x+X-1)$. The implied constant only depends on $d,\varepsilon$, as stated in the `Notation' section of \cite{Va97}.
\end{proof}

\begin{lemma}[Wooley's transference principle, cf. {\cite[Lemma A.1]{W15}}]
\label{WooleyTransference}
Let $\tau,X_0,Y_0,Z_0$ be positive real numbers, and let $\Psi:\mathbb{R}\to\mathbb{C}$. Suppose that whenever $\xi\in\R$, $a\in\mathbb{Z}$ and $q\in\mathbb{N}$ satisfy $(a,q)=1$ and $
\left|\xi-a/q\right|\leq q^{-2}$,
one has
\[
\Psi(\xi)\ll X_0\left(\frac{1}{q}+\frac{1}{Y_0}+\frac{q}{Z_0}\right)^\tau.
\]
Then, whenever $b\in\mathbb{Z}$ and $r\in\mathbb{N}$ satisfy $(b,r)=1$, one has
\[
\Psi(\xi)\ll X_0\left(\frac{1}{r+Z_0|r\xi-b|}+\frac{1}{Y_0}+\frac{r+Z_0|r\xi-b|}{Z_0}\right)^\tau.
\]
The implied constant in the conclusion depends only on $\tau$ and on the implied constant in the hypothesis.
\end{lemma}

\begin{lemma}\label{weyl}
Let $p(x)=\a_1x^d+\cdots+\a_dx\in\mathbb{R}[x]$ be a polynomial. 
Suppose for some $a\in\mathbb{Z}$, $q\in\mathbb{N}$ and $s\in[1,q]$ with $(a,q)=1$ we have
$$
\bigg|\a_1-\frac{a}{q}\bigg|\leq\frac1{q s}.
$$
Then for any $\varepsilon>0$ and $X,Y\in\mathbb{N}$,
$$
\sum_{X\leq n< X+Y}e\left(p(n)\right)\ll_{d,\varepsilon}Y^{1+\varepsilon}\left(\frac1{s}+\frac1Y+\frac{q}{Y^d}\right)^{\frac1{2^{d-1}}}.
$$
\end{lemma}

\begin{proof}
Let
\[
\Phi:\mathbb{R}\to\mathbb{C};\qquad
\Phi(\xi)=\sum_{X\leq n< X+Y}e(\xi n^d+\a_2n^{d-1}+\cdots+\a_dn).
\]
By \Cref{4refds9oiref9doi}, the hypothesis of \Cref{WooleyTransference} holds for $\Phi$ with
\[
X_0=Y^{1+\varepsilon},\qquad Y_0=Y,\qquad Z_0=Y^d,\qquad \tau=\frac{1}{2^{d-1}}.
\]
Applying \Cref{WooleyTransference} with $\xi=\a_1$, $b=a$, and $r=q$, we deduce that
\begin{align*}
\sum_{X\leq n< X+Y}e(p(n))&\ll_{d,\varepsilon} Y^{1+\varepsilon}\left(\frac{1}{q+Y^d|q\a_1-a|}+\frac{1}{Y}+\frac{q+Y^d|q\a_1-a|}{Y^d}\right)^{\frac1{2^{d-1}}}\\
&\ll Y^{1+\varepsilon}\left(\frac{1}{s}+\frac{1}{Y}+\frac{q}{Y^d}\right)^{\frac1{2^{d-1}}},
\end{align*}
where in the last step we used $\left|\a_1-\frac{a}{q}\right|\leq\frac{1}{qs}$ and $s\leq q$.
\end{proof}

\begin{lemma}\label{ptheta} Suppose that $q,N,Y\in\mathbb{N}$, $Y\leq N$ and $p(x)\in\Z[x]$. Let $\psi$ be differentiable on $[N-Y,N]$ with $\Psi=\max_{N-Y\leq x\leq N}|\psi'(x)|$. Then
$$
\sum_{N-Y\leq n< N}e\left(\frac{p(n)}q+\psi(n)\right)=
\frac1q\sum_{r=1}^qe\left(\frac{p(r)}q\right)\sum_{N-Y\leq n< N}e\left(\psi(n)\right)+O\left(q+q\Psi Y\right).$$
\end{lemma}

\begin{proof}
By definition of $\Psi$, for all $\frac{N-Y}{q}\leq m<\frac{N}{q}-1$ and $0\leq r\leq q-1$ we have $|\psi(mq)-\psi(mq+r)|<q\Psi$. This implies that, for all $0\leq r\leq q-1$,
\begin{align*}
\sum_{\frac{N-Y}{q}\leq m<\frac{N}{q}}e(\psi(qm))&=\sum_{\frac{N-Y}{q}\leq m<\frac{N}{q}}e(\psi(qm+r))+O(1+\Psi Y)\\&
=\frac{1}{q}\sum_{N-Y\leq n< N}e(\psi(n))+O(1+\Psi Y),
\end{align*}
where the second equality follows by averaging the first equality over $0\leq r\leq q-1$. Thus,
\begin{align*}
&\sum_{N-Y\leq n< N}e\left(\frac{p(n)}q+\psi(n)\right)=\sum_{r=0}^{q-1}\sum_{\frac{N-Y}{q}\leq m<\frac{N}{q}}e\left(\frac{p(mq+r)}{q}+\psi(mq+r)\right)+O(q)\\
&=\sum_{r=0}^{q-1}e\left(\frac{p(r)}{q}\right)\sum_{\frac{N-Y}{q}\leq m<\frac{N}{q}}e\left(\psi(mq+r)\right)+O(q)\\
&=
\sum_{r=0}^{q-1}e\left(\frac{p(r)}{q}\right)\frac{1}{q}\sum_{N-Y\leq n< N}e(\psi(n))+O(q+q\Psi Y).\qedhere
\end{align*}
\end{proof}

\begin{lemma}\label{QuantitativeUnifDist}
Let $X,Y\in\mathbb{N}$, $\delta\in(0,1/4)$ and $p:\mathbb{R}\to\mathbb{R}$, and suppose that for all natural $1\leq j\leq\delta^{-2}$ we have
\begin{equation*}
\left|\sum_{X\leq n<X+Y}e(j\cdot p(n))\right|<\delta^2 Y.
\end{equation*}
Then, for all $\b\in\mathbb{R}$,
$$
|\{n\in\mathbb{N}:X\leq n<X+Y,\,\|p(n)-\b\|\leq \delta\}|\ll\delta Y.
$$
\end{lemma}

\begin{proof}
By \cite[Chapter I, Lemma 12]{Vi54} with $r=1$, $\alpha=-2\delta$, $\beta=2\delta$, $\Delta=2\delta$, there is a $1$-periodic function $\chi:\mathbb{R}\to[0,1]$ satisfying $\chi(x)=1$ if $\|x\|\leq \delta$, $\chi(x)=0$ if $\|x\|\geq3\delta$ and with a Fourier expansion
$$
\chi(x)=4\delta+\sum_{|j|\geq 1}a_je(jx),\textup{ where }|a_j|\leq\min\left\{\frac{1}{|j|},8\delta,\frac{1}{\delta j^2}\right\}.
$$ 
Therefore, for all $\b\in\mathbb{R}$,
\begin{align*}
&|\{X\leq n< X+Y:\,\|p(n)-\b\|\leq \delta\}|\\
\leq&
\sum_{X\leq n< X+Y}\chi(p(n)-\b)\\
\leq&
4\delta Y+\sum_{|j|\geq1}|a_j|\cdot\left|\sum_{X\leq n<X+Y}e(j\cdot(p(n)-\b))\right|
\\
\leq&4\delta Y+\sum_{|j|>\delta^{-2}}\frac{Y}{\delta j^2}+\sum_{1\leq |j|\leq \delta^{-2}}8\delta\cdot\bigg|\sum_{X\leq n<X+Y}e(j\cdot p(n))\bigg|\\
\leq&
4\delta Y+2\delta Y+2\delta^{-2}\cdot 8\delta\cdot (\delta^2 Y)=22\delta Y.\qedhere
\end{align*}
\end{proof}

\msubsection{Proof of \Cref{texp} when $r_\alpha(x)\succ\log x$}
\label{SecErgralphaBig}

Fix a Hardy function $\alpha(x)=p_\alpha(x)+r_\alpha(x)$, where $p_\alpha,r_\alpha$ are as in \Cref{PolyAndSpecialParts}, and assume that $\alpha(x)\gg_\alpha x$. Let $d=\deg(\alpha)$.

Since $\alpha(x)\gg_\alpha x$, if $\deg(r_\alpha)<1$, then $\deg(p_\alpha)=d\geq1$. If $p_\alpha\neq0$, we write the polynomial part as
$$p_\alpha(x)=\a_1x^{d}+\a_2x^{d-1}+\cdots+\a_{d}x,\textup{ where }\a_1\neq0.$$
Throughout this section we assume that $r_\alpha\succ\log x$. Then, by L'Hôpital's rule, for every $k\geq1$,
$$
\lim_{x\to+\infty}x^k|r_\alpha^{(k)}(x)|=\lim_{x\to+\infty}\frac{|r_\alpha^{(k)}(x)|}{|\log^{(k)}(x)|}=+\infty
$$
Indeed, L'Hôpital's rule applies because, as $r_\alpha$ is the non-polynomial part of $\alpha$, we have $\lim_{x\to+\infty}|r_\alpha^{(k)}(x)|\in\{0,+\infty\}$ for all $k$.
Consider the Hardy function
$$
\lambda_0(x)=\min\left(x^{1/2},x^{d+1}\left|r_\alpha^{(d+1)}(x)\right|\right),
$$
so that $\lim_{x\to+\infty}\lambda_0(x)=+\infty$ and $
\deg(\lambda_0)=\min(1/2,\deg(r_\alpha))$. Also let
\begin{equation*}
\sigma_0=\frac{1}{2^{\deg(\alpha)+5}}.
\end{equation*}

\begin{lemma}
\label{32ew8ido8er9wiodkNoFloors}
There is $N_0=N_0(\alpha)\in\mathbb{N}$ such that, for all $N\geq N_0$, $1\leq Y\leq N/2$ and $\beta\in\mathbb{R}$ such that 
$$
\frac{\lambda_0(N)}{\alpha(N)}\leq\beta\leq\lambda_0(N),
$$
we have
\begin{equation}
\label{re9fdsoic}
\sum_{N-Y\leq n< N}e\left(\alpha(n)\beta\right)\ll
\frac{N}{\lambda_0(N)^{2\sigma_0}}.
\end{equation}
\end{lemma}

\begin{proof}
We first consider the case $\deg(r_\alpha)=0$. Let $N$ be sufficiently large. 

We first consider the range of values of $\beta$ given by
$$
\frac{\lambda_0(N)}{\alpha(N)}\leq\beta\leq\frac1{\lambda_0(N)^{\frac12}}.
$$
Set $\b_1:=d!\a_1$. Since $\deg(r_\alpha)=0$, we have
$\alpha^{(d)}(x)=\b_1+o(1)$.
If $d\geq 2$, 
then by \Cref{DerivEstimates}\ref{dDerivEstimates} with $\lambda=\frac{\beta\b_1}{2},h=3$ (we may replace $Y$ by $N/2$, which only weakens the desired estimate), for big enough $N$ we have
\begin{align*}
\sum_{N-Y\leq n< N}e\left(\alpha(n)\beta\right)&\ll
N\left(\left(\frac{\beta\b_1}{2}\right)^{\frac{1}{{2^d}-2}}+N^{-\frac{2}{2^d}}+\left(N^{d}\frac{\beta\b_1}{2}\right)^{-\frac2{2^d}}\right)
\\
&\ll
N\left(\left(\frac{\b_1}{\lambda_0(N)^{\frac12}}\right)^{\frac{1}{{2^d}-2}}+N^{-\frac{2}{2^d}}+\lambda_0(N)^{-\frac2{2^d}}\right)\\
&\ll\frac{N}{\lambda_0(N)^{\frac1{2\cdot{2^d}}}}.
\end{align*}
If $d=1$, then, for all sufficiently large $N$, $\beta$ is small enough that $\|\beta\alpha'(n)\|\geq\frac{\beta\b_1}{2}$, so by \Cref{DerivEstimates}\ref{1DerivEstimates},
\begin{equation}\label{bkl}
\sum_{N-Y\leq n< N}e\left(\alpha(n)\beta\right)\leq\frac2{\b_1\beta}\leq\frac{2\alpha(N)}{\b_1\lambda_0(N)}\ll\frac{N}{\lambda_0(N)}.
\end{equation}

It remains to consider the range $\frac1{\lambda_0(N)^{\frac12}}<\beta\leq\lambda_0(N)$.
Since $p_\alpha^{(d+1)}=0$, $$\left|\left(\beta \alpha(x)\right)^{(d+1)}\right|=\beta\cdot\lambda_0(x) x^{-d-1}.$$
Applying \Cref{DerivEstimates}\ref{dDerivEstimates} (to the $(d+1)^{\textup{st}}$ derivative),
\begin{align*}
\sum_{N-Y\leq n< N}e\left(\alpha(n)\beta\right)&\ll
N\cdot\left(\left(\beta \lambda_0(N) N^{-d-1}\right)^{\frac1{{2^{d+1}}-2}}+N^{-\frac2{2^{d+1}}}+(\beta\lambda_0(N))^{-\frac2{2^{d+1}}}\right)\\
&\ll
N\cdot\left(\left(N^{-d}\right)^{\frac{1}{2^{d+1}-2}}+N^{-\frac2{2^{d+1}}}+(\lambda_0(N))^{-\frac1{2^{d+1}}}\right)
\\
&\ll
\frac{N}{\lambda_0(N)^{\frac1{2^{d+1}}}}.
\end{align*}
This completes the proof in the case $\deg(r_\alpha)=0$.

We now turn to the case $\deg(r_\alpha)>0$, so $\lambda_0(x)=x^{1/2}$. We argue similarly to \cite[Section 3.2]{CKW10} and add the details for the case $p_\alpha=0$. Let $c=\deg(r_\alpha)$, and set $\eta=\frac{1}{4}\min\left(\frac12,c\right)=\frac{1}{4}\deg(\lambda_0)$. We split the proof into two subcases:
\begin{enumerate}
    \item $p_\alpha\neq0$, so $r_\alpha(x)\prec p_\alpha(x)$. Then $\alpha(x)^{(d)}=\b_1+o(1)$. First, assume
\[
\frac{\lambda_0(N)}{\alpha(N)}\leq\beta\leq N^{-c+\eta}.
\]
If $d=1$, then \Cref{bkl} holds true again, for the same reason. If $d\geq 2$, 
then applying \Cref{DerivEstimates}\ref{dDerivEstimates} with $\lambda=\frac{\beta\b_1}{2},h=3$, we have for all sufficiently large $N$ that 
\begin{align*}
\sum_{N-Y\leq n< N}e(\alpha(n)\beta)
&\ll
N\left(\left(\frac{\beta\b_1}{2}\right)^{\frac{1}{{2^d}-2}}+N^{-\frac{2}{2^d}}+\left(N^{d}\frac{\beta\b_1}{2}\right)^{-\frac2{2^d}}\right)\\
&\ll N\left(\beta^{\frac{1}{2^d}}+N^{\frac{-2}{2^d}}+\lambda_0(N)^{\frac{-2}{2^d}}\right)\\
&\ll N\left(N^{\frac{-c+\eta}{2^d}}+\lambda_0(N)^{\frac{-2}{2^d}}\right)\ll\frac{N}{\lambda_0(N)^{\frac{1}{2^{d+1}}}},
\end{align*}
where in the last inequality we used $\deg(\lambda_0)\leq c$.

It remains to consider the range $N^{-c+\eta}\leq\beta\leq \lambda_0(N)$. Let $\varepsilon=\eta\cdot100^{-d}$. Then $\deg\left(\alpha^{(d+1)}\right)=c-d-1<-1$. We apply \Cref{DerivEstimates}\ref{dDerivEstimates} with $(d+1)^{\textup{th}}$-derivative, $\lambda=\beta N^{c-d-1-\frac{\varepsilon}{4}}$ (so $N^{\eta-d-1-\frac{\varepsilon}{4}}\ll\lambda\ll N^{\frac{-1}{2}}$) and $h=N^{\frac\varepsilon2}$, to obtain
\begin{align*}
\sum_{N-Y\leq n< N}e(\alpha(n)\beta)&\ll N^{1+\frac{\varepsilon}{2}}\left(\left(N^{-1/2}\right)^{\frac{1}{2^{d+1}-2}}+N^{\frac{-2}{2^{d+1}}}+\left(N^{\eta-\frac{\varepsilon}{4}}\right)^{\frac{-2}{2^{d+1}}}\right)
\\
&\ll\frac{N}{N^{\frac{\eta}{2^{d+1}}}}
\ll\frac{N}{\lambda_0(N)^{\frac{1}{2^{d+4}}}}.
\end{align*}

    \item $p_\alpha=0$, so $\alpha(x)=r_\alpha(x)\succ x$, $c\geq1$, and $\lambda_0(x)=x^{1/2}$. 
    Fix $\beta$ with $\frac{\lambda_0(N)}{\alpha(N)}\leq\beta\leq\lambda_0(N)$, and let $k=\lceil c+1\rceil\geq2$ and $\varepsilon=\eta\cdot100^{-k}$. We apply \Cref{DerivEstimates}\ref{dDerivEstimates} to the $k^{\textup{th}}$ derivative. Since $\alpha^{(k)}$ has degree $c-k$, we may take $\lambda=\beta N^{c-k-\frac{\varepsilon}{4}}$ and $h=N^{\frac{\varepsilon}{2}}$, so that
    $
    N^{\frac{1}{2}-k-\frac{\varepsilon}{2}}\ll\lambda\ll N^{-\frac12}.
    $
    Therefore,
\begin{align*}
\sum_{N-Y\leq n< N}e(\alpha(n)\beta)&\ll N^{1+\frac{\varepsilon}{2}}\left(\left(N^{\frac{-1}2}\right)^{\frac{1}{2^k-2}}+N^{\frac{-2}{2^k}}+\left(N^{\frac{1}{2}-\frac\varepsilon2}\right)^{\frac{-2}{2^k}}\right)
\\
&\ll\frac{N}{N^{\frac1{2^{k+1}}}}
\ll 
\frac{N}{\lambda_0(N)^{\frac1{2^{c+3}}}}
\qedhere
\end{align*}
\end{enumerate}
\end{proof}

\begin{lemma}
\label{32ew8ido8er9wiodk}
For all $N\geq N_0(\alpha)$, $1\leq Y\leq N/2$ and $\theta\in\mathbb{R}$ such that 
$
\|\theta\|\geq\frac{\lambda_0(N)}{\alpha(N)}
$, we have
\begin{equation}\label{minore}
\sum_{N-Y\leq n< N}e\left(\lfloor\alpha(n)\rfloor\theta\right)\ll\frac{N}{\lambda_0(N)^{\sigma_0}}.
\end{equation}
\end{lemma}

\begin{proof}
Since the sum is unchanged when $\theta$ is replaced by its fractional part, we may assume that
$
\frac{\lambda_0(N)}{\alpha(N)}\leq \theta\leq1-\frac{\lambda_0(N)}{\alpha(N)}.
$
Applying \Cref{fit} with $H=\lfloor\lambda_0(N)^{\frac12}\rfloor$, we obtain
\begin{align}
&\left|\sum_{N-Y\leq n< N}e\left(\lfloor\alpha(n)\rfloor\theta\right)\right|
\leq
\sum_{|h|\leq H}\frac{\|\theta\|}{|h+\theta|}\left|\sum_{N-Y\leq n< N}e((h+\theta)\alpha(n))\right|\notag\\
&+O\left(
\log^2 H\cdot\left(\frac{Y}{H}+\sup_{1\leq h\leq H^2}
\left|
\sum_{N-Y\leq n< N}e(h\alpha(n))
\right|
\right)\right).
\label{4re9wfdso90opdfs}
\end{align}
Write the right-hand side of \Cref{4re9wfdso90opdfs} as $A_1+O(A_2)$. Since $|h+\theta|\geq\|\theta\|$ and $|h+\theta|\leq H+1\leq\lambda_0(N)$ for all sufficiently large $N$, we can apply \Cref{32ew8ido8er9wiodkNoFloors} with $\beta=|h+\theta|$ for each $h\in\mathbb{Z}$, $|h|\leq H$. Thus,
\begin{equation*}
A_1\ll
\sum_{|h|\leq H}\frac{\|\theta\|}{|h+\theta|}\cdot\frac{N}{\lambda_0(N)^{2\sigma_0}}\leq\frac{2N}{\lambda_0(N)^{2\sigma_0}}\sum_{1\leq h< H+1}\frac{1}{h}\ll\frac{N\log(\lambda_0(N))}{\lambda_0(N)^{2\sigma_0}}.
\end{equation*}
Similarly, applying \Cref{32ew8ido8er9wiodkNoFloors} with $\beta=h$ for all $1\leq h\leq H^2$, we obtain 
$
A_2\ll N\log^2(H)\lambda_0(N)^{-2\sigma_0}.
$
This completes the proof.
\end{proof}

\begin{lemma}\label{minor} For all $N\geq N_0(\alpha)$ and $\theta\in\mathbb{R}$ such that 
$
\|\theta\|\geq\frac{\lambda_0(N)}{\alpha(N)}
$, we have
\begin{equation}\label{minord}
\sum_{N/2\leq n< N}\alpha'(n)e\left(\lfloor\alpha(n)\rfloor\theta\right)\ll\frac{\alpha'(N)N}{\lambda_0(N)^{\sigma_0}}.
\end{equation}
\end{lemma}

\begin{proof}
This follows from \Cref{32ew8ido8er9wiodk} and \Cref{ChangingCoefs}, applied with $\beta(x)=\alpha'(x)$, $z_n=e\left(\lfloor\alpha(n)\rfloor\theta\right)$ and $L=\lfloor N/2\rfloor$.
\end{proof}

\begin{lemma}\label{lemmaefnemth}
For all $N\geq N_0(\alpha)$ and $\theta\in\mathbb{R}$,
\begin{equation}\label{efnemth}
\sum_{N/2\leq n< N}\alpha'(n)e\left(\lfloor\alpha(n)\rfloor\theta\right)=\sum_{\alpha(N/2)\leq m< \alpha(N)}e(m\theta)+O\left(\frac{\alpha(N)}{\lambda_0(N)^{\frac{\sigma_0}2}}\right).
\end{equation}
\end{lemma}
\begin{proof}
By the geometric series estimate, if $\theta\not\in\mathbb{Z}$ then
$$
\sum_{\alpha(N/2)\leq m< \alpha(N)}e(m\theta)=O\left(\|\theta\|^{-1}\right),
$$
so, in view of \Cref{minor}, \Cref{efnemth} holds when $\|\theta\|\geq\frac{\lambda_0(N)}{\alpha(N)}$.
Thus we may assume that $\|\theta\|<\lambda_0(N)/\alpha(N)$. Replacing $\theta$ by its representative in $\left[-\frac12,\frac12\right]$, we may also assume that $|\theta|<\lambda_0(N)/\alpha(N)$. 

Since $|e(\lfloor\alpha(n)\rfloor\theta)-e(\alpha(n)\theta)|\leq|\theta|$ for all $n$ and $\sum_{n=1}^N|\alpha'(n)|\ll\alpha(N)$, we have
$$
\sum_{N/2\leq n<N}\alpha'(n)e\left(\lfloor\alpha(n)\rfloor\theta\right)=\sum_{N/2\leq n<N}\alpha'(n)e\left(\alpha(n)\theta\right)+O\left(\alpha(N)|\theta|\right).
$$

By \Cref{major}, for any $\theta$ with $|\theta|\leq\lambda_0(N)/\alpha(N)$ we have
\begin{multline*}
\sum_{N/2\leq n<N}\alpha'(n)e\left(\alpha(n)\theta\right)-\sum_{\alpha(N/2)\leq m<\alpha(N)}e(m\theta)
\\
\ll
N\frac{\lambda_0(N)}{\alpha(N)}\alpha'(N)^2+N\alpha''(N)+\alpha'(N)
\ll\frac{\alpha(N)}{\lambda_0(N)^{\frac{\sigma_0}2}},
\end{multline*}
where the last inequality follows by comparing degrees and increasing $N_0(\alpha)$ if necessary.
\end{proof}

\begin{proof}[Proof of \Cref{texp} when $r_\alpha(x)\succ\log x$.]

Let $
\beta(x)=\frac{\alpha(x)}{\lambda_0(x)^{\frac{\sigma_0}{2}}}$.
Then $\beta$ belongs to a Hardy field and $\deg(\beta)>\frac{1}{2}$, so by \Cref{LogHardyBehaviour}, for all sufficiently large $N$ we have $\beta(N/2)<0.9\beta(N)$. Let $k$ satisfy 
$$\sqrt{N}\leq 2^k\leq 2\sqrt{N}.$$ Since $\alpha(x)\gg_\alpha x$, we have $\frac{\alpha(x)}{x}\gg_\alpha1$, and hence $\frac{\alpha(x)}{x}\gg\frac{\alpha\left(\sqrt{x}\right)}{\sqrt{x}}$ for all sufficiently large $x$. Thus, for all sufficiently large $N$,
\[
\alpha\left(\frac{N}{2^k}\right)\ll\alpha\left(\sqrt{N}\right)\ll\frac{\alpha(N)}{\sqrt{N}}.
\]
We now conclude the proof by using \Cref{lemmaefnemth}:
\begin{align}
\label{4ref8dsi8e9roikl}&\sum_{1\leq n<N}\alpha'(n)e\left(\lfloor\alpha(n)\rfloor\theta\right)-
\sum_{1\leq n< \alpha(N)}e(n\theta)\\
\notag=&\sum_{j=0}^k\left(\sum_{\frac{N}{2^{j+1}}\leq n<\frac{N}{2^{j}}}\alpha'(n)e\left(\lfloor\alpha(n)\rfloor\theta\right)-
\sum_{\alpha\left(\frac{N}{2^{j+1}}\right)\leq n< \alpha\left(\frac{N}{2^{j}}\right)}e(n\theta)\right)+O\left(\alpha\left(\frac{N}{2^k}\right)\right)\\
\notag\ll&\sum_{j=0}^k\beta\left(\frac{N}{2^j}\right)+\alpha\left(\frac{N}{2^k}\right)\ll\sum_{j=0}^k0.9^j\beta(N)+\frac{\alpha(N)}{\sqrt{N}}\ll\beta(N)=\frac{\alpha(N)}{\lambda_0(N)^{\frac{\sigma_0}{2}}}.\qedhere
\end{align}
\end{proof}

\msubsection{Proof of \Cref{texp} when $r_\alpha(x)\ll_\alpha\log x$}
\label{SecErgralphaSmall}

In this section, we fix a Hardy function $\alpha(x)$ such that $\alpha(x)\gg_\alpha x$ and $\alpha(x)=p_\alpha(x)+r_\alpha(x)$, where $r_\alpha(x)\ll_\alpha\log(x)$, and write the polynomial part $p_\alpha$ as
$$p_\alpha(x)=\a_1x^d+\cdots+\a_dx\in\mathbb{R}[x],\textup{ with }\a_1\neq0.$$
The assumptions of \Cref{texp} imply that $p_\alpha\not\in c\mathbb{Q}[x]$ for all real $c$; in particular, the degree $d$ of $p_\alpha$ is at least $2$. Moreover, throughout the rest of the section we use the notation
\begin{align*}
D&=d\cdot2^{d+2},\\
\varepsilon&=100^{-d^2},\\
j_0&=\min\{j:\,\a_j/\a_1\not\in\Q\}\geq2.
\end{align*}

For any $\theta\in\R$ and $X\in[1,+\infty)$, let 
\begin{align*}
\L_{X}(\theta)&=\{q\in\mathbb{N}:q\leq X\textup{ and }|\theta q-a|\leq X^{-1}\text{ for some }a\in\mathbb{Z}\text{ with }(a,q)=1\}\\
&=
\{q\in\{1,2,\dots,\lfloor X\rfloor\}:\left|\theta-\frac{a}{q}\right|\leq\frac{1}{qX}\text{ for some }a\in\mathbb{Z}\text{ with }(a,q)=1\}.
\end{align*}

Dirichlet's approximation theorem asserts that $\L_{X}(\theta)$ is always non-empty. Note that, if $\theta$ is rational, say $\theta=\frac{a}{q}$ for some coprime $a,q\in\mathbb{Z}$ with $q>0$, then $\mathcal{L}_X(\theta)=\{q\}$ for any $X>|q|$. If, on the contrary, $\theta$ is irrational, then Dirichlet's approximation theorem implies that $\lim_{N\to\infty}\max \L_N(\theta)=\infty$ and that $\max\L_q(\theta)=q$ for infinitely many $q$. Finally, define
$$
\lambda(N)=\min\big\{N^{\varepsilon},\max\L_{N^{d-j_0+1-2\varepsilon D^{j_0}}}(\a_{j_0}/\a_1)\big\},
$$
so that $\lambda(N)\to+\infty$ as $N\to\infty$.

During this section we will routinely use inequalities such as, for $N\geq1$,
$$
N\geq N^{\frac{1}{d}}\geq N^{\frac{1}{D}}\geq N^{\varepsilon D^d}\geq N^{d\varepsilon}\geq N^{\varepsilon}.
$$

The bulk of the proof is formed by two lemmas which bound exponential sums related to the polynomial part $p_\alpha$. The non-polynomial part $r_\alpha$ will not play a role in the proof until \Cref{438erwiofsd}.

\begin{lemma}\label{pcq}
There is $N_0=N_0(d,\varepsilon)$ such that, for all $N\geq N_0$, if there exists $1\leq j_1\leq d$ satisfying
\begin{enumerate}
    \item for each $1\leq j\leq j_1-1$,
$$
\max\L_{N^{d-j+1-\varepsilon D^{j}}}(\a_j)\leq N^{\varepsilon D^{j}},
$$
\item there exist $a_{j_1},q_{j_1}$ such that $(a_{j_1},q_{j_1})=1$, $N^{\frac34\varepsilon D^{j_1}}\leq q_{j_1}\leq N^{d-j_1+1-\varepsilon D^{j_1}}$ and 
$$
\bigg|\a_{j_1}-\frac{a_{j_1}}{q_{j_1}}\bigg|\leq\frac{1}{q_{j_1}N^{\varepsilon D^{j_1}}}.
$$
\end{enumerate}

Then for all $Y\leq N^{1-\varepsilon}$, 
$$
\sum_{N-Y\leq n< N}e\left(p_\alpha(n)\right)\ll_{d,\varepsilon} N^{1-2\varepsilon}.
$$
\end{lemma}

\begin{proof}
The case $j_1=1$ follows from \Cref{weyl}, applied with $X=N-Y$, $q=q_{j_1}$ and $s=\lfloor N^{\frac34\varepsilon D}\rfloor$. Thus, we assume $j_1\geq 2$.
For $1\leq j\leq j_1-1$, let
$$
q_j=\max\L_{N^{d-j+1-\varepsilon D^{j}}}(\a_j)\leq N^{\varepsilon D^j}
$$
and let
$$
\vartheta_j=\a_j-\frac{a_j}{q_j},$$
where $(a_j,q_j)=1$ and $|\vartheta_j|\leq q_j^{-1}N^{\varepsilon D^j-d+j-1}$.
Set
$$
S_m:=\sum_{N-2Y\leq n< m}e\left(\sum_{j=1}^{j_1-1}\frac{a_jn^{d-j+1}}{q_j}+\sum_{j=j_1}^{d}\a_{j}n^{d-j+1}\right),
$$
and apply the summation by parts formula, \Cref{SumParts}, to obtain
\begin{align}
&\notag\sum_{N-Y\leq n< N}e\left(p_\alpha(n)\right)=
\sum_{N-Y\leq n< N}e\left(\sum_{j=1}^{j_1-1}\a_j n^{d-j+1}+\sum_{j=j_1}^{d}\a_j n^{d-j+1}\right)\\\notag
=&\sum_{N-Y\leq n< N}e\left(\sum_{j=1}^{j_1-1}\frac{a_jn^{d-j+1}}{q_j}+\sum_{j=j_1}^{d}\a_{j}n^{d-j+1}\right)\cdot
e\left(\sum_{j=1}^{j_1-1}\vartheta_jn^{d-j+1}\right)\\\label{4refd9oierfdoi}
\ll&|S_N|+|S_{N-Y}|\\&+\sum_{N-Y\leq m<N}\left|\left(e\left(\sum_{j=1}^{j_1-1}\vartheta_j(m+1)^{d-j+1}\right)-e\left(\sum_{j=1}^{j_1-1}\vartheta_j m^{d-j+1}\right)\right) S_{m+1}\right|.\notag
\end{align}
For all $N-Y\leq m\leq N$,
$$
\left|e\left(\sum_{j=1}^{j_1-1}\vartheta_j(m+1)^{d-j+1}\right)-e\left(\sum_{j=1}^{j_1-1}\vartheta_j m^{d-j+1}\right)\right|\ll_d\sum_{j=1}^{j_1-1}|\vartheta_j|N^{d-j}\ll_d N^{\varepsilon D^{j_1-1}-1}.
$$
Thus it remains to prove that
\begin{equation}
\label{holrdsidlx}
\max_{N-Y\leq m\leq N}|S_m|\ll_{d,\varepsilon}N^{1-\varepsilon}\cdot N^{-\varepsilon D^{j_1-1}}
\end{equation}
Indeed, the two endpoint terms in \Cref{4refd9oierfdoi} are then acceptable, and the last term is
$$
\ll_{d,\varepsilon}Y\cdot N^{\varepsilon D^{j_1-1}-1}\cdot \left(N^{1-\varepsilon}\cdot N^{-\varepsilon D^{j_1-1}}\right)
\leq N^{1-2\varepsilon}.
$$

Let $Q=\mathrm{lcm}(q_1,\ldots,q_{j_1-1})$, and write $\frac{a_*}{q_*}$ for the reduced form of $\frac{a_{j_1}Q^{d-j_1+1}}{q_{j_1}}$, so that
$$
q_*=\frac{q_{j_1}}{(q_{j_1},Q^{d-j_1+1})}.
$$
Then we have
\begin{equation}
\label{3ew9ds0op9ewdospl}
\bigg|\a_{j_1} Q^{d-j_1+1}-\frac{a_*}{q_*}\bigg|=Q^{d-j_1+1}\bigg |\a_{j_1} -\frac{a_{j_1}}{q_{j_1}}\bigg|\leq\frac{Q^{d-j_1+1}}{q_{j_1} N^{\varepsilon D^{j_1}}}.
\end{equation}
Moreover, using that $q_j\leq N^{\varepsilon D^j}$, we obtain
\begin{equation}
\label{3ew9ds0op9ewdospl2}
Q\leq\prod_{j=1}^{j_1-1}q_j\leq N^{\frac{\varepsilon D(D^{j_1-1}-1)}{D-1}}\leq N^{\frac{\varepsilon D^{j_1}}{2^{d+2}(d-1)}}.
\end{equation}
In particular, $Q^{d-1}\leq N^{\frac{\varepsilon D^{j_1}}{2^{d+2}}}$.
Since $j_1\geq2$, we have $d-j_1+1\leq d-1$. Combining this with $q_*\leq q_{j_1}$, \Cref{3ew9ds0op9ewdospl,3ew9ds0op9ewdospl2} imply that
\begin{align}
\label{qXe}
\bigg|\a_{j_1} Q^{d-j_1+1}-\frac{a_*}{q_{*}}\bigg|
&\leq \frac{Q^{d-j_1+1}}{q_{j_1}N^{\varepsilon D^{j_1}}}
\leq \frac{Q^{d-1}}{q_*N^{\varepsilon D^{j_1}}}
\leq\frac{1}{q_*N^{\frac{\varepsilon D^{j_1}}{2}}}.
\end{align}

Now, as $q_j$ divides $Q$ for all $1\leq j< j_1$, we have
\begin{align*}
|S_m|\leq&\sum_{r=1}^Q\left|\sum_{\substack{N-2Y\leq n< m\\ n\equiv r\pmod{Q}}}e\left(\sum_{j=1}^{j_1-1}\frac{a_j n^{d-j+1}}{q_j}+\sum_{j=j_1}^{d}\a_{j}n^{d-j+1}\right)\right|\\
\leq&Q\max_{1\leq r\leq Q}
\left|\sum_{(N-2Y-r)/Q\leq n< (m-r)/Q}e\left(\sum_{j=j_1}^{d}\a_{j}(Qn+r)^{d-j+1}\right)\right|
\end{align*}

Therefore, in order to prove \Cref{holrdsidlx} (and thus conclude the proof of \Cref{pcq}), it is enough to show, uniformly for $N-Y\leq m\leq N$ and $1\leq r\leq Q$, that
\begin{align}
&\sum_{(N-2Y-r)/Q\leq n< (m-r)/Q}e\left(\sum_{j=j_1}^{d}\a_{j}(Qn+r)^{d-j+1}\right)\notag
\\
\ll&\left(\frac{N^{1-\varepsilon}}{Q}\right)^{1+\varepsilon}
\left(\frac1{N^{\frac{\varepsilon D^{j_1}}{2}}}+\frac{Q}{N^{1-\varepsilon}}+\frac{Q^{d-j_1+1}q_*}{N^{(1-\varepsilon)(d-j_1+1)}}\right)^{\frac{8d}D}\ll\frac{N^{1-\varepsilon}}{QN^{\varepsilon(2D^{j_1-1}-1)}}.\label{reofidsodfslk}
\end{align}
First note that the polynomial $\sum_{j=j_1}^{d}\a_{j}(Qn+r)^{d-j+1}$ has leading coefficient $\a_{j_1}Q^{d-j_1+1}$. We will also need upper and lower bounds for $q_*$, which follow from the definition of $q_*$ and the bounds on $q_{j_1}$:
\begin{equation*}
N^{\frac{\varepsilon D^{j_1}}{2}}\leq\frac{q_{j_1}}{Q^{d-1}}\leq q_*\leq q_{j_1}\leq N^{d-j_1+1-\varepsilon D^{j_1}}.
\end{equation*}

The first inequality in \Cref{reofidsodfslk} follows from \Cref{weyl}, using \eqref{qXe}, except that we replace the length of the interval by the larger quantity $N^{1-\varepsilon}/Q$ and the exponent $\frac{1}{2^{d-j_1}}$ by the smaller exponent $\frac{8d}{D}=\frac{1}{2^{d-1}}$. 
This change in exponents is harmless because, by the bounds on $Q$ and $q_*$, the expression inside the big parentheses in \eqref{reofidsodfslk} is at most $1+1+1=3$. 
It remains to verify the second inequality in \Cref{reofidsodfslk}. For this, it suffices to check, for each $X\in\left\{\frac1{N^{\frac{\varepsilon D^{j_1}}{2}}},\frac{Q}{N^{1-\varepsilon}},\frac{Q^{d-j_1+1}q_*}{N^{(1-\varepsilon)(d-j_1+1)}}\right\}$, that 
\begin{equation}
\label{43refds9oierodfo9i}
X\ll\left(\frac{Q^\varepsilon}{N^{(1-\varepsilon)\varepsilon} N^{2\varepsilon D^{j_1-1}}}\right)^{\frac D{8d}}=\frac{Q^{\frac{\varepsilon D}{8d}}}{N^{\varepsilon(1-\varepsilon)\frac{D}{8d}+\varepsilon\frac{D^{j_1}}{4d}}}=
\frac{Q^{\frac{\varepsilon D}{8d}}}{N^{\varepsilon\left(\frac{D^{j_1}}{4d}+\frac{D}{8d}-\frac{\varepsilon D}{8d}\right)}}
\end{equation}
The bound $1\leq Q\leq N^{\frac{\varepsilon D^{j_1}}{2^{d+2}(d-1)}}$ implies \Cref{43refds9oierodfo9i} for $X=\frac1{N^{\frac{\varepsilon D^{j_1}}{2}}}$ and $X=\frac{Q}{N^{1-\varepsilon}}$. For the remaining value $X=\frac{Q^{d-j_1+1}q_*}{N^{(1-\varepsilon)(d-j_1+1)}}$, we have
\begin{equation*}
X\leq\frac{N^{\frac{\varepsilon D^{j_1}(d-j_1+1)}{2^{d+2}(d-1)}}\cdot N^{d-j_1+1-\varepsilon D^{j_1}}}{N^{(1-\varepsilon)(d-j_1+1)}}= N^{\varepsilon\left(\frac{ D^{j_1}(d-j_1+1)}{2^{d+2}(d-1)}-D^{j_1}+d-j_1+1\right)}\leq\frac1{N^{\frac{\varepsilon D^{j_1}}{2}}},
\end{equation*}
so \Cref{43refds9oierodfo9i} follows from the first case. This completes the proof.\qedhere
\end{proof}

\begin{lemma}\label{pfnbeta}
There is $N_0=N_0(\alpha)$ such that, for all $N>N_0$, all $Y\leq N^{1-\varepsilon}$, and all $\theta\in\R$ satisfying
\[
N^{\varepsilon(D+1)-d}\leq\theta\leq\lambda(N)^{\frac13},
\]
we have
\begin{equation}\label{pfs0}
\sum_{N-Y\leq n< N}e\left(p_\alpha(n)\theta\right)\ll\frac{N^{1-\varepsilon}}{\lambda(N)^{\frac{1}{4d}}}.
\end{equation}
\end{lemma}

\begin{proof} 
Suppose first that there is some $1\leq j\leq d$ such that $\max\L_{N^{d-j+1-\varepsilon D^j}}(\a_j\theta)\geq N^{\varepsilon D^j}$.
Let $j_1$ be the smallest such index. Then $\max\L_{N^{d-j+1-\varepsilon D^j}}(\a_j\theta)<N^{\varepsilon D^j}$ for each $1\leq j<j_1$, and there are $a_{j_1},q_{j_1}\in\mathbb{Z}$ such that
\[
N^{\varepsilon D^{j_1}}\leq q_{j_1}\leq N^{d-{j_1}+1-\varepsilon D^{j_1}}
\]
and 
$$
\left|\a_{j_1}\theta-\frac{a_{j_1}}{q_{j_1}}\right|\leq\frac{1}{q_{j_1}N^{d-{j_1}+1-\varepsilon D^{j_1}}}\leq\frac{1}{q_{j_1}N^{\varepsilon D^{j_1}}}.
$$
Then \Cref{pfs0} follows by applying \Cref{pcq} to the polynomial $\theta p_\alpha(n)$, since $\lambda(N)\leq N^\varepsilon$.

So we assume that $\max\L_{N^{d-j+1-\varepsilon D^j}}(\a_j\theta)<N^{\varepsilon D^j}$ for every $1\leq j\leq d$.
For $j\neq j_0$, let
\[
q_j=\max\L_{N^{d-j+1-\varepsilon D^j}}(\a_j\theta),
\]
and let
\[
q_*=\max\L_{N^{d-j_0+1-2\varepsilon D^{j_0}}}(\a_{j_0}/\a_1).
\]
Choose $a_j,a_*\in\mathbb{Z}$ so that $\frac{a_j}{q_j}$ and $\frac{a_*}{q_*}$ are the corresponding closest approximants to $\a_j\theta$ and $\a_{j_0}/\a_1$, respectively; in particular, $(a_1,q_1)=(a_*,q_*)=1$.
Write 
\begin{align*}
\a_{j_0}\theta=\a_{1}\theta\cdot\frac{\a_{j_0}}{\a_1}=&\left(\frac{a_1}{q_1}+O\left(\frac1{q_1N^{d-\varepsilon D}}\right)\right)\cdot\left(\frac{a_*}{q_*}+O\left(\frac1{q_*N^{d-j_0+1-2\varepsilon D^{j_0}}}\right)\right)\\
=&\frac{a_1}{q_1}\cdot\frac{a_*}{q_*}+O_\alpha\left(\frac1{q_1N^{d-\varepsilon D}}+\frac{\theta}{q_*N^{d-j_0+1-2\varepsilon D^{j_0}}}\right),
\end{align*}
Note that $a_1\neq0$, because
$$
\left|\frac{a_1}{q_1}\right|\geq|\a_1\theta|-\frac{1}{q_1N^{d-\varepsilon D}}\geq\frac{|\a_1|}{N^{d-\varepsilon(D+1)}}-\frac{1}{q_1N^{d-\varepsilon D}}>0.
$$
Let $\frac{a_{j_0}}{q_{j_0}}$ be the reduced form of $\frac{a_1a_*}{q_1q_*}$, so that
\begin{equation}
\label{t5erdf89oire89dfi}
q_{j_0}=\frac{q_1q_*}{(q_1q_*,a_1a_*)}=\frac{q_1}{(q_1,a_*)}\cdot
\frac{q_*}{(q_*,a_1)}.
\end{equation}
Since $q_*\geq q_{j_0}/q_1$, $q_1\leq N^{\varepsilon D}$, and $0\leq \theta\leq N^{\varepsilon}$, we have $\frac{\theta}{q_*}\leq\frac{N^{\varepsilon(D+1)}}{q_{j_0}}$, so
$$
\a_{j_0}\theta=
\frac{a_{j_0}}{q_{j_0}}+O_\alpha\left(\frac1{N^{d-\varepsilon D}}+\frac{N^{\varepsilon(D+1)}}{q_{j_0}N^{d-j_0+1-2\varepsilon D^{j_0}}}\right).
$$

If $q_*\geq N^{\varepsilon D^{j_0}}$, then
$$
q_{j_0}\leq q_*q_1\leq N^{d-j_0+1-\varepsilon D^{j_0}},
$$
and, using \Cref{t5erdf89oire89dfi},
$$
q_{j_0}\geq\frac{q_*}{|a_1|}\gg_\alpha\frac{q_*}{\theta q_1}\geq N^{\frac34\varepsilon D^{j_0}}.
$$
Thus $q_{j_0}$ satisfies requirement (ii) of \Cref{pcq}, and the desired result follows.

So we assume that $q_*\leq N^{\varepsilon D^{j_0}}$. For $1\leq j\leq d$, let
$
\vartheta_j=\a_j\theta-a_j/q_j
$, so that
$$|\vartheta_j|\leq\frac{1}{q_jN^{d-j+1-\varepsilon D^j}}\textup{ for }j\neq j_0\textup{ and }|\vartheta_{j_0}|\ll_\alpha\frac{N^{\varepsilon(D+1)}}{N^{d-j_0+1-2\varepsilon D^{j_0}}}.$$
Let $\psi(x)=\vartheta_1 x^d+\cdots+\vartheta_d x$.
Then
$$
\Psi:=\max_{N-N^{1-\varepsilon}\leq x\leq N}|\psi'(x)|\leq d
\sum_{j=1}^d|\vartheta_j|N^{d-j}\ll_\alpha N^{3\varepsilon D^d-1}.
$$
Let $Q=\mathrm{lcm}(q_1,\ldots,q_d)$ and let
\[
p(x)=Q\left(\frac{a_1x^d}{q_1}+\cdots+\frac{a_dx}{q_d}\right).
\]
By \Cref{ptheta},
\begin{align}
\label{43rewfsido94eriofs}
\sum_{N-Y\leq n< N}e\left(p_\alpha(n)\theta\right)=&
\sum_{N-Y\leq n< N}e\left(\frac{p(n)}Q+\psi(n)\right)\\
\notag=&
\frac1Q\sum_{r=1}^Q e\left(\frac{p(r)}Q\right)\sum_{N-Y\leq n< N}e\left(\psi(n)\right)+O\left(Q N^{3\varepsilon D^d}\right).
\end{align}
We claim that $(Qa_1/q_1,\ldots,Qa_d/q_d,Q)=1$. 
Suppose not, and let $\rho$ be a prime dividing $(Qa_1/q_1,\ldots,Qa_d/q_d,Q)$.
By definition of $Q$, there exists $1\leq j\leq d$ such that
$\rho$ does not divide $Q/q_j$. Since $\rho$ divides $Qa_j/q_j$, it follows that $\rho$ divides $a_j$. Also, since $\rho$ divides $Q$ but not $Q/q_j$, it divides $q_j$, contradicting
$(a_j,q_j)=1$. Therefore, we may apply \Cref{Va97Thm7.1}:
\begin{equation}
\label{43rewfsido94eriofs2}
\frac1Q\sum_{r=1}^Q e\left(\frac{p(r)}Q\right)=\frac1Q\sum_{r=1}^Q e\left(\frac{\frac{Qa_1}{q_1} r^d+\cdots+\frac{Qa_d}{q_d} r}{Q}\right)\ll_{d,\varepsilon} Q^{-\frac1d+\varepsilon}.
\end{equation}
Note that $Q\geq\max\{q_1,q_{j_0}\}\gg_{\a_1} q_*^{\frac13}$; to check the latter inequality we use that, if
$q_1\leq q_*^{\frac13}$, then as $\lambda(N)=\min(N^\varepsilon,q_*)\leq q_*$ and $\theta\leq\lambda(N)^{\frac13}$, we have
$$
q_{j_0}\geq\frac{q_*}{|a_1|}\geq\frac{q_*}{2|\a_1|\theta q_1}\geq \frac{q_*^{\frac13}}{2|\a_1|}.
$$
Also, $q_{j_0}\leq q_*q_1\leq N^{2\varepsilon D^{j_0}}$, and hence
\[
Q\leq q_1\cdots q_{j_0}\cdots q_d\leq N^{2\varepsilon D^dd}.
\]
It now follows from \Cref{43rewfsido94eriofs,43rewfsido94eriofs2} that
\begin{align*}
\sum_{N-Y\leq n< N}e\left(p_\alpha(n)\theta\right)
&\ll \frac{Y}{Q}\bigg|\sum_{r=1}^Q e\left(\frac{p(r)}Q\right)\bigg|+N^{5\varepsilon D^dd}\\
&\ll_{d,\varepsilon}\frac{N^{1-\varepsilon}}{Q^{\frac{1}{d}-\varepsilon}}
\ll_\alpha\frac{N^{1-\varepsilon}}{q_*^{\frac1{3.5d}}}
\ll\frac{N^{1-\varepsilon}}{\lambda(N)^{\frac1{3.5d}}}.
\end{align*}
Therefore, by increasing $N_0$ if necessary, for all $N\geq N_0$ we have
\begin{equation}
\left|\sum_{N-Y\leq n< N}e\left(p_\alpha(n)\theta\right)\right|\leq\frac{N^{1-\varepsilon}}{\lambda(N)^{\frac{1}{4d}}}.\qedhere
\end{equation}
\end{proof}

\begin{lemma}\label{pff}
There is $N_0=N_0(\alpha)$ such that, for all $N\geq N_0$, all $\b\in\R$, and all $\delta\geq\lambda(N)^{-\frac{1}{8d}}$,
\[
|\{n\in[N-N^{1-\varepsilon},N]:\,\|p_\alpha(n)+\b\|\leq \delta\}|\ll\delta N^{1-\varepsilon}.
\]
\end{lemma}

\begin{proof}
If $\delta\geq1/4$, the result is trivial. Otherwise, apply \Cref{QuantitativeUnifDist} to $p_\alpha(x)+\b$ on the interval $[N-N^{1-\varepsilon},N]$ and \Cref{pfnbeta}.
\end{proof}

We now combine the previous lemmas into the final estimate needed for the proof of the case
$r_\alpha(x)\ll_\alpha\log x$ of \Cref{texp}.

\begin{lemma}
\label{438erwiofsd}
There is $N_0=N_0(\alpha)$ such that, for all $N\geq N_0$, all $\theta\in\mathbb{R}$, and all $Y\leq N^{1-\varepsilon}$,
\[
\sum_{N-Y\leq n< N}\alpha'(n)e\left(\lfloor\alpha(n)\rfloor\theta\right)=
\sum_{\alpha(N-Y)\leq m< 
\alpha(N)}e(m\theta)+O\left(N^{d-\varepsilon}\lambda(N)^{-\frac{1}{10d}}\right).
\]
\end{lemma}

\begin{proof} 
Let $\b_0=r_\alpha(N)$. If $N$ is big enough, then
\[
\b_0-r_\alpha(n)\ll\lambda(N)^{-\frac{1}{8d}}
\]
for all $n\in [N-N^{1-\varepsilon},N]$, since $r_\alpha'(x)=O(x^{-1})$ and $\lambda(N)^{-\frac{1}{8d}}\gg N^{-\varepsilon/2}$.
Note that
\[
\lfloor p_\alpha(n)+r_\alpha(n)\rfloor\not=\lfloor p_\alpha(n)+\b_0\rfloor
\]
only if 
\[
\|p_\alpha(n)+\b_0\|<|\b_0-r_\alpha(n)|.
\]
In view of \Cref{pff}, the number of such $n$ in $[N-N^{1-\varepsilon},N]$ is at most $O(N^{1-\varepsilon}\lambda(N)^{-\frac{1}{8d}})$. Therefore, for any $\theta\in\R$,
\[
\sum_{N-Y\leq n< N}e\left(\lfloor\alpha(n)\rfloor\theta\right)=
\sum_{N-Y\leq n< N}e\left(\lfloor p_\alpha(n)+\b_0\rfloor\theta\right)+O\left(N^{1-\varepsilon}\lambda(N)^{-\frac{1}{8d}}\right).
\]
Suppose first that $\|\theta\|\geq N^{\varepsilon(D+1)-d}$ and let $H=\lambda(N)^{\frac{1}{8d}}$. By periodicity and conjugation, we may assume $\theta\in(0,1/2]$. Applying \Cref{fit,pfnbeta},
we have
\begin{align*}
&\sum_{N-Y\leq n< N}e\left(\lfloor p_\alpha(n)+\b_0\rfloor\theta\right)\\
\ll&\sum_{|h|\leq H}\frac{\|\theta\|}{|h+\theta|}\left|\sum_{N-Y\leq n< N}e((h+\theta)(p_\alpha(n)+\b_0))\right|\\
&+\log^2 H\sup_{\substack{1\leq h\leq H^2\\
h\in\mathbb{N}}}
\bigg|\sum_{N-Y\leq n< N}e\left(h(p_\alpha(n)+\b_0)\right)\bigg|+O\left(\frac{Y\log^2 H}{H}\right)\\
\ll&\log(H)\frac{N^{1-\varepsilon}}{\lambda(N)^{\frac1{4d}}}+\log^2H\frac{N^{1-\varepsilon}}{\lambda(N)^{\frac1{4d}}}+\frac{N^{1-\varepsilon}\log^2H}{H}
\ll N^{1-\varepsilon}\lambda(N)^{-\frac{1}{10d}}.
\end{align*}
Thus, using \Cref{ChangingCoefs} to change the coefficients to $\alpha'(n)$ we obtain 
$$
\sum_{N-Y\leq n< N}\alpha'(n)e\left(\lfloor\alpha(n)\rfloor\theta\right)\ll 
N^{d-\varepsilon}\lambda(N)^{-\frac{1}{10d}}.
$$

On the other hand,
if $\|\theta\|\leq N^{\varepsilon(D+1)-d}$, we may assume $\theta\in(-1/2,1/2]$, so that $\|\theta\|=|\theta|$.
Since $\lfloor\alpha(n)\rfloor=p_\alpha(n)+\b_0+O(1)$ for $N-Y\leq n<N$, $r_\alpha'(N)=O_\alpha(N^{-1})$ and
using \Cref{major} in the last step, for big enough $N$ we have
\begin{align*}
&\sum_{N-Y\leq n< N}\alpha'(n)e\left(\lfloor\alpha(n)\rfloor\theta\right)\\
=&\sum_{N-Y\leq n< N}p_\alpha'(n)e\left(\lfloor\alpha(n)\rfloor\theta\right)+O_\alpha(N^{-\varepsilon})\\
=&\sum_{N-Y\leq n< N}p_\alpha'(n) e\left((p_\alpha(n)+\b_0)\theta\right)+O_\alpha\left(N^{d-\varepsilon}\|\theta\|+N^{-\varepsilon}\right)\\
=&\sum_{p_\alpha(N-Y)+\b_0\leq m< 
p_\alpha(N)+\b_0}e(m\theta)+O\left(p_\alpha'(N)^2N^{1-\varepsilon}\|\theta\|+p_\alpha''(N)N^{1-\varepsilon}+p_\alpha'(N)\right)\\
=&\sum_{\alpha(N-Y)\leq m< 
\alpha(N)}e(m\theta)+O\left(N^{d-\frac12}\right).\qedhere
\end{align*}
\end{proof}

\begin{proof}[Proof of \Cref{texp} when $r_\alpha(x)=O(\log x)$]

Let $a_0=N,a_{j+1}=a_{j}-a_{j}^{1-\varepsilon}$ and $k:=\min\left\{j\in\mathbb{N};a_j\leq\frac{N}{\lambda(N)}\right\}$. Then we have $k\leq N$, and for each $j=0,\dots,k-1$, by \Cref{438erwiofsd},
\begin{equation}
\label{4refd9oil}
\sum_{a_{j+1}\leq n<a_{j}}\alpha'(n)e\left(\lfloor\alpha(n)\rfloor\theta\right)-\sum_{\alpha(a_{j+1})\leq m<\alpha(a_{j})}e(m\theta)\ll\frac{a_{j}^{d-\varepsilon}}{\lambda(a_j)^{\frac{1}{10d}}}.
\end{equation}
Also note that $a_{j}^{d-\varepsilon}\leq N^{d-1}a_{j}^{1-\varepsilon}=N^{d-1}(a_j-a_{j+1})$, and let $\Lambda:\mathbb{N}\to\mathbb{N};\Lambda(N)=\min(\lambda(n);n\geq N)$, so that $\lim_{N\to\infty}\Lambda(N)=+\infty$. We conclude:
\begin{align*}
&\sum_{1\leq n< N}\alpha'(n)e\left(\lfloor\alpha(n)\rfloor\theta\right)-\sum_{1\leq m< \alpha(N)}e(m\theta)\\
=&\sum_{a_k\leq n< N}\alpha'(n)e\left(\lfloor\alpha(n)\rfloor\theta\right)-
\sum_{\alpha(a_k)\leq m< \alpha(N)}e(m\theta)+O\left(\alpha\left(\frac{N}{\lambda(N)}\right)\right)\\
=&
\sum_{j=0}^{k-1}
\left(\sum_{a_{j+1}\leq n<a_{j}}\alpha'(n)e\left(\lfloor\alpha(n)\rfloor\theta\right)-\sum_{\alpha(a_{j+1})\leq m<\alpha(a_{j})}e(m\theta)\right)
+O\left(\frac{N^d}{\lambda(N)}\right)\\
\ll&\sum_{j=0}^{k-1}\frac{a_j^{d-\varepsilon}}{\lambda(a_j)^{\frac{1}{10d}}}+\frac{N^d}{\lambda(N)}
\ll N^{d-1}\sum_{j=0}^{k-1}\frac{a_j-a_{j+1}}{\lambda(a_j)^{\frac{1}{10d}}}+\frac{N^d}{\lambda(N)}\ll\frac{N^d}{\Lambda\left(\frac{N}{\lambda(N)}\right)^{\frac{1}{10d}}}+\frac{N^d}{\lambda(N)}.\qedhere
\end{align*}
\end{proof}

\msection{Proof of \Cref{mainp} for sub-polynomials of type ~\protect\hyperlink{CaseIII}{(III)}}
\label{ProofOfmainp}

We have already proved \Cref{mainp} for the sub-polynomials satisfying the hypotheses of \Cref{maind}, namely sub-polynomials of types ~\hyperlink{CaseI}{(I)} and ~\hyperlink{CaseII}{(II)}. Indeed, for any finite coloring of $\mathbb{N}$, one color class has positive upper density, and \Cref{maind} applies to that color class. Thus it remains to treat sub-polynomials $\alpha(x)$ of type ~\hyperlink{CaseIII}{(III)}, namely those of the form
\begin{equation*}
\alpha(x)=c(p_\alpha(x)+r_\alpha(x)),
\end{equation*}
where $c>0$, $p_\alpha\in\Q[x]$ and $1\prec r_\alpha(x)\ll_\alpha\log(x)$.

The section is organized as follows. In \Cref{SecExpPolyRat}, we prove \Cref{ErgodicThmIntegerPolys}, the ergodic approximation theorem for integer polynomials needed in this remaining case. Then, in \Cref{SecColorralphaSmall}, we use \Cref{ErgodicThmIntegerPolys} to prove \Cref{DistanceColorThmDensityVersionCaseIII}, a finitistic result which implies the remaining case of \Cref{mainp}.

\msubsection{A version of \Cref{tergodic} for a class of integer polynomials}
\label{SecExpPolyRat}
\Cref{tergodic} need not hold when $\alpha$ is an integer polynomial (in fact, even \Cref{maindc}, which follows from \Cref{tergodic}, can fail in this case; see \Cref{SecObstructionsDensityTheorem}). However, we still have the following variant of \Cref{tergodic} along arithmetic progressions.

\begin{theorem}\label{ErgodicThmIntegerPolys}
Let $d\in\mathbb{N}$, $\varepsilon=100^{-d^2}$, and let $p\in\mathbb{Z}[x]$ be given by
\begin{equation*}
p(x)=\b_1x^d+\b_2x^{d-1}+\cdots+\b_{d-1}x^2+ux,\textup{ with }\b_1>0,u\in\{-1,1\}.
\end{equation*}
There exists $w_0(p)\in\mathbb{N}$ such that, for all $w,N\in\N$ and all $N\in\mathbb{N}$ satisfying $w\geq w_0$ and $1\leq w\leq\left\lfloor\varepsilon^3\log(N)\right\rfloor$, the following holds. Define $W=\textup{lcm}(1,2,\dots,w)$ and $P(n)=p(Wn)\in\mathbb{Z}[x]$. Then, for all m.p.s. $(\mathcal{X},\mathcal{B},\mu,T)$ and for all $f\in L^2(\mu)$, we have
\begin{align*}
\Bigg\|\frac{1}{P(N)}
\sum_{1\leq n< N}P'(n)T^{P(n)}f
-\frac{W}{P(N)}\sum_{\substack{1\leq m<\frac{P(N)}{W}}}T^{Wm}f\Bigg\|_2
\ll_d\frac{\|f\|_2}{w^{\frac{1}{d}-\varepsilon}}.
\end{align*}

More generally, for all $1\leq Y\leq N-N^{1-\varepsilon}$ we have
\begin{align*}
\Bigg\|\frac{1}{P(N)-P(Y)}
\sum_{Y\leq n< N}P'(n)T^{P(n)}f
-\frac{W}{P(N)-P(Y)}\sum_{\substack{\frac{P(Y)}{W}\leq m<\frac{P(N)}{W}}}T^{Wm}f\Bigg\|_2
\ll_d\frac{\|f\|_2}{w^{\frac{1}{d}-\varepsilon}}.
\end{align*}
\end{theorem}

\begin{remark}\label{ErgodicThmIntegerPolysRmk}
With the same setup as in \Cref{ErgodicThmIntegerPolys}, for all $A,B\in\mathcal{B}$ we have
\begin{multline*}
\Bigg|\frac{1}{P(N)-P(Y)}
\sum_{Y\leq n< N}P'(n)\mu\left(A\cap T^{P(n)}B\right)
\\-\frac{W}{P(N)-P(Y)}\sum_{\substack{\frac{P(Y)}{W}\leq m<\frac{P(N)}{W}}}\mu\left(A\cap T^{Wm}B\right)\Bigg|
\ll_d\frac{\sqrt{\mu(A)\mu(B)}}{w^{\frac{1}{d}-\varepsilon}}.
\end{multline*}
Indeed, this follows by applying \Cref{ErgodicThmIntegerPolys} to $1_B$, taking the inner product with $1_A$, and using the Cauchy--Schwarz inequality.
\end{remark}

We now state the exponential sum version of \Cref{ErgodicThmIntegerPolys}. The proof of \Cref{ErgodicThmIntegerPolys} from \Cref{WeaktexpForIntegerPolys} is the same as the proof of \Cref{tergodic} from \Cref{texp}, and is omitted.

\begin{theorem}\label{WeaktexpForIntegerPolys}
Let $d\in\mathbb{N}$, $\varepsilon=100^{-d^2}$, and let $p\in\mathbb{Z}[x]$ be given by
\begin{equation*}
p(x)=\b_1x^d+\b_2x^{d-1}+\cdots+\b_{d-1}x^2+ux,\textup{ with }\b_1>0,u\in\{-1,1\}.
\end{equation*}
There exists $w_0(p)\in\mathbb{N}$ such that, for all $w,N\in\N$ and all $N\in\mathbb{N}$ satisfying $w\geq w_0$ and $1\leq w\leq\left\lfloor\varepsilon^3\log(N)\right\rfloor$, the following holds. Define $W=\textup{lcm}(1,2,\dots,w)$ and $P(n)=p(Wn)\in\mathbb{Z}[x]$. Then, for all $\theta\in\mathbb{R}$,
\begin{align*}
\frac{1}{P(N)}
\sum_{1\leq n< N}P'(n)e\left(P(n)\theta\right)=&\frac{W}{P(N)}\sum_{\substack{1\leq m<\frac{P(N)}{W}}}e(mW\theta)
+O_d\left(\frac{1}{w^{\frac{1}{d}-\varepsilon}}\right).
\end{align*}
More generally, for all $1\leq Y\leq N-N^{1-\varepsilon}$ we have
\begin{align}
\label{r4ew9dfsoiklrowiekdls}
\frac{1}{P(N)-P(Y)}
\sum_{Y\leq n< N}P'(n)e\left(P(n)\theta\right)=&\frac{W}{P(N)-P(Y)}\sum_{\substack{\frac{P(Y)}{W}\leq m<\frac{P(N)}{W}}}e(mW\theta)\\
&+O_d\left(\frac{1}{w^{\frac{1}{d}-\varepsilon}}\right).\notag
\end{align}
\end{theorem}
We will presently formulate a more convenient for our purposes version of \Cref{WeaktexpForIntegerPolys}, namely \Cref{UniformSumEstimateRationalPolyGeneral}. We then derive \Cref{WeaktexpForIntegerPolys} from this formulation and devote the remainder of the subsection to proving \Cref{UniformSumEstimateRationalPolyGeneral}.

The following notation will be used throughout the rest of this subsection. Fix $d\in\mathbb{N}$ and set
\begin{align*}
D&=2^{d+2}d,\\
\varepsilon&=100^{-d^2}.
\end{align*}
During the rest of this section, we let $w,N\in\N$ be natural numbers such that 
$$w_0\leq w\leq2\varepsilon^3\log N,$$ 
where $w_0$ is a sufficiently large natural number depending only on $d,\b_1,\dots,\b_{d-1}$. Define
\begin{align}
W=W(w):=\textup{lcm}(1,2,\dots,w)\label{DefFunctsN}
\end{align}
It is proved in \cite[Theorem 12]{RS62} that $W(w)\leq\exp(1.03883w)$, and therefore
\[
W\leq N^{3\varepsilon^3}.
\]

Finally, let $\b_1,\dots,\b_{d-1}\in\mathbb{Z}$ satisfy $\b_1>0$, let $u\in\{-1,1\}$, and define
\begin{equation}
\label{GoodPolynomial}
p(x)=\b_1W^{d-1}x^d+\b_2W^{d-2}x^{d-1}+\cdots+\b_{d-1}Wx^2+ux\in\mathbb{Z}[x].
\end{equation}

With this convention,
\[
p(N)-p(N-N^{1-\varepsilon})\asymp_d \b_1W^{d-1}N^{d-\varepsilon}.
\]

\begin{theorem}\label{UniformSumEstimateRationalPolyGeneral}
For all $\theta\in\mathbb{R}$ and all $Y\in\mathbb{N}$ with $1\leq Y\leq N^{1-\varepsilon}$, we have
\begin{align*}
\sum_{N-Y\leq n< N}p'(n)e\left(p(n)\theta\right)&=
\sum_{p(N-Y)\leq m< p(N)}e(m\theta)
+O_{d}\left(\b_1 W^{d-1}N^{d-\varepsilon}w^{\frac{-1}{d}+\varepsilon}\right)\\
&=
\sum_{p(N-Y)\leq m< p(N)}e(m\theta)
+O_{d}\left(\frac{p(N)-p(N-N^{1-\varepsilon})}{w^{\frac{1}{d}-\varepsilon}}\right).
\end{align*}
\end{theorem}

The deduction of \Cref{WeaktexpForIntegerPolys} from \Cref{UniformSumEstimateRationalPolyGeneral} proceeds through the following intermediate statement.

\begin{corollary}\label{UniformSumEstimateRationalPolyGeneralLongIntervals}
Assume $w\leq\varepsilon^3\log N$. Then, for all $Y\in\mathbb{N}$ with $1\leq Y\leq N-N^{1-\varepsilon}$ and all $\theta\in\mathbb{R}$, we have
\begin{equation}
\label{4re9fsdo9ero}
\sum_{Y\leq n< N}p'(n)e\left(p(n)\theta\right)=
\sum_{p(Y)\leq m< p(N)}e(m\theta)
+O_{d}\left(\frac{p(N)-p(Y)}{w^{\frac{1}{d}-\varepsilon}}\right).
\end{equation}
\end{corollary}

\begin{proof}[Proof of \Cref{UniformSumEstimateRationalPolyGeneralLongIntervals} from \Cref{UniformSumEstimateRationalPolyGeneral}]
We may assume that $Y\geq\sqrt{N}$; the case when $Y$ is smaller follows, for sufficiently large $N$, from the triangle inequality. We may partition the interval $[Y,N]$ as $Y=Y_0<Y_1<\cdots<Y_k=N$ so that
\[
\frac{Y_i^{1-\varepsilon}}{2}\leq Y_i-Y_{i-1}\leq Y_i^{1-\varepsilon}
\]
for every $1\leq i\leq k$. Since $w\leq \varepsilon^3\log N\leq 2\varepsilon^3\log Y_i$, we may apply \Cref{UniformSumEstimateRationalPolyGeneral} to each interval $[Y_{i-1},Y_i)$ and sum the resulting estimates. The total error is
\[
\ll_d \frac{\b_1W^{d-1}}{w^{\frac1d-\varepsilon}}\sum_{i=1}^{k}Y_i^{d-\varepsilon}
\ll_d \frac{\b_1W^{d-1}}{w^{\frac1d-\varepsilon}}\sum_{i=1}^{k}Y_i^{d-1}(Y_i-Y_{i-1})
\ll_d \frac{p(N)-p(Y)}{w^{\frac1d-\varepsilon}},
\]
which gives the claim.
\end{proof}

\begin{proof}[Proof of \Cref{WeaktexpForIntegerPolys} from \Cref{UniformSumEstimateRationalPolyGeneralLongIntervals}]
With the notation of \Cref{WeaktexpForIntegerPolys}, the polynomial $\frac{P(x)}{W}=\frac{p(Wx)}{W}$ has the form given in \Cref{GoodPolynomial}. Applying \Cref{UniformSumEstimateRationalPolyGeneralLongIntervals} to $\frac{P(x)}{W}$, with frequency $W\theta$, gives
\begin{equation*}
\sum_{Y\leq n< N}\frac{P'(n)}{W}e\left(P(n)\theta\right)=
\sum_{\frac{P(Y)}{W}\leq m< \frac{P(N)}{W}}e(mW\theta)
+O_{d}\left(\frac{P(N)-P(Y)}{Ww^{\frac{1}{d}-\varepsilon}}\right).
\end{equation*}
Multiplying both sides by $\frac{W}{P(N)-P(Y)}$, we obtain \Cref{r4ew9dfsoiklrowiekdls}.
\end{proof}

We now turn to the proof \Cref{UniformSumEstimateRationalPolyGeneral}.
For $a\in\mathbb{Z}$ and $q,N\in\mathbb{N}$ with $(a,q)=1$, let $\M_{a,q}=\M_{a,q}(N)$ denote the set of $\theta\in\mathbb{R}$ such that
\begin{equation}\label{h1b1Rational}
\left|\theta-\frac{a}{q}\right|\leq\frac{1}{qN^{d-\varepsilon(d+1)}}.
\end{equation}
By Dirichlet's approximation theorem, for every $\theta\in\mathbb{R}$ there exist coprime $a\in\mathbb{Z}$ and $1\leq q\leq N^{d-\varepsilon(d+1)}$ such that $\theta\in\M_{a,q}$.

The proof of \Cref{UniformSumEstimateRationalPolyGeneral} is split according to the size of the denominator $q$ such that $\theta\in\M_{a,q}$ for some $a\in\Z$. We use different arguments in the cases
\[
q=1,\qquad
2\leq q\leq N^{\frac{\varepsilon}{3d}},\qquad
N^{\frac{\varepsilon}{3d}}\leq q\leq N^{d-\varepsilon(d+1)}.
\]
These three cases are handled, respectively, by the next three lemmas.

\begin{lemma}\label{q=1SumClose}
For all $Y\in\mathbb{N}$ with $Y\leq N^{1-\varepsilon}$, all $a\in\mathbb{Z}$ and all $\theta\in\M_{a,1}$, we have
\begin{align*}
\sum_{N-Y\leq n< N}p'(n)e\left(p(n)\theta\right)&=\sum_{p(N-Y)\leq m< p(N)}e\left(m\theta\right)+O_d\left(N^{d-1+\varepsilon(d+1)}\right).
\end{align*}
\end{lemma}
\begin{proof}
We have $\|\theta\|\leq\frac{1}{N^{d-\varepsilon(d+1)}}$, and as $p\in\mathbb{Z}[x]$, we may assume $|\theta|=\|\theta\|$. Using \Cref{major} (whose proof gives a choice of $N_0$ which only depends on $d,\b_1,\dots,\b_{d-1}$),
we have
\begin{align*}
&\left|\sum_{N-Y\leq n< N}p'(n)e\left(p(n)\theta\right)-
\sum_{p(N-Y)\leq m< p(N)}
e(m\theta)\right|\\
&\ll p'(N)^2\cdot N^{1-\varepsilon}\cdot\frac{1}{N^{d-\varepsilon(d+1)}}+N^{1-\varepsilon}|p''(N)|+|p'(N)|+1
\ll_d N^{d-1+\varepsilon(d+1)}.\qedhere
\end{align*}
\end{proof}

\begin{lemma}\label{RatEstimateBigSumqsmall}
For all $Y\in\mathbb{N}$ with $1\leq Y\leq N^{1-\varepsilon}$, all $a\in\mathbb{Z}$ and $q\in\mathbb{N}$ with $2\leq q\leq N^{\frac\varepsilon{3d}}$, and all $\theta\in\M_{a,q}$,
\begin{equation}
\label{4erfds9ofdoRat}
\sum_{N-Y\leq n< N}e\left(p(n)\theta\right)\ll_{d}N^{1-\varepsilon}w^{-\frac1d+\varepsilon}.
\end{equation}
\end{lemma}

\begin{proof}
Let $\vartheta=\theta-\frac{a}{q}$, so that $|\vartheta|\leq\frac{1}{q N^{d-\varepsilon(d+1)}}$, and let $\psi(x)=\vartheta p(x)$. If $N$ is big enough (depending only on $d,\b_i$),
$$\Psi:=\max_{N-N^{1-\varepsilon}\leq x\leq N}|\psi'(x)|\ll_d\frac{\b_1W^{d-1}N^{d-1}}{qN^{d-\varepsilon(d+1)}}\leq N^{\varepsilon(d+2)-1}.$$
Thus $q\Psi Y\leq N^{\varepsilon(d+3/2)}$, which by \Cref{ptheta}, applied to the polynomial $ap$, implies
\begin{align}
\sum_{N-Y\leq n< N}e\left(p(n)\theta\right)
=\frac1q\sum_{r=1}^qe\left(p(r)\cdot \frac aq\right)\sum_{N-Y\leq n< N}e\left(p(n)\vartheta\right)+O\left(N^{\varepsilon(d+3/2)}\right).\label{ewiodssidfkl}
\end{align}
Now, for each $2\leq j\leq d$, the coefficient of $x^j$ in $p(x)$ is divisible by $W$. Suppose first that $q$ does not divide $W$. Then $q\geq w$ by definition of $W$, so by \Cref{Va97Thm7.1},
\begin{equation}
\label{ewrsdfoiewsdoRat}
\frac1q\sum_{r=1}^qe\left(\frac aq\cdot p(r)\right)\ll_d q^{-\frac1d+\varepsilon}\leq w^{-\frac1d+\varepsilon}.
\end{equation}

The same estimate also holds if $q$ divides $W$, since then, using $q>1$,
$$
\frac1q\sum_{r=1}^qe\left(\frac aq\cdot p(r)\right)=\frac1q\sum_{r=1}^qe\left(\frac aq\cdot ur\right)=0.
$$
We conclude using \Cref{ewiodssidfkl,ewrsdfoiewsdoRat}:
\begin{align*}
\sum_{N-Y\leq n< N}e\left(p(n)\theta\right)&\ll_d\left|\frac1q\sum_{r=1}^qe\left(\frac aq\cdot p(r)\right)\right|\cdot\left|\sum_{N-Y\leq n< N}e\left(p(n)\vartheta\right)\right|+N^{\varepsilon(d+3/2)}\\
&\ll w^{\frac{-1}{d}+\varepsilon}\cdot N^{1-\varepsilon}.\qedhere
\end{align*}
\end{proof}

\begin{lemma}\label{betaaqRat}
For all $Y\in\N$ with $1\leq Y\leq N^{1-\varepsilon}$, all $\theta\in\mathbb{R}$, and all $a,q\in\mathbb{Z}$ satisfying $1\leq q\leq N^{d-\varepsilon (d+1)}$, $(a,q)=1$ and
$$
\left|\theta-\frac aq\right|\leq\frac1{qN^{d-2\varepsilon D}},
$$
we have
$$
\sum_{\substack{N-Y\leq n< N}}e\left(p(n)\theta\right)
\ll_d
\frac{N^{1-\varepsilon}}{q^{\varepsilon^2}}.
$$
\end{lemma}

\begin{proof}
Suppose first that $q\geq N^\varepsilon$, and let $\frac{a_1}{q_1}$ be the reduced form of $W^{d-1}\b_1\frac{a}{q}$. For sufficiently large $N$, we have 
\begin{equation*}
N^{\varepsilon-3d\varepsilon^3}\leq \frac{q}{\b_1W^{d-1}}\leq q_1 \leq q\leq N^{d-\varepsilon(d+1)}.
\end{equation*}
Moreover, $\frac{a_1}{q_1}$ is very close to the leading coefficient of $\theta p(n)$:
$$\left|W^{d-1}\b_1\theta-\frac{a_1}{q_1}\right|=W^{d-1}\b_1\left|\theta-\frac{a}{q}\right|
\leq
\frac{\b_1W^{d-1}}{qN^{d-2\varepsilon D}}
\leq
\frac{1}{q_1N^{d-3\varepsilon D}}.$$
Applying \Cref{weyl} (with the variables $p(x),a,q,s,\varepsilon$ from \Cref{weyl} taking the values $p(x)\theta,a_1,q_1,\min(q_1,N^{d-3\varepsilon D}),\varepsilon^2$ respectively) gives
\begin{align*}
\sum_{\substack{N-Y\leq n< N}}e\left(p(n)\theta\right)
&
\ll_d
Y^{1+\varepsilon^2}\left(\frac{1}{\min(q_1,N^{d-3\varepsilon D})}+\frac{1}{Y}+\frac{q_1}{Y^d}\right)^{\frac{1}{2^{d-1}}}\\
&\ll
\left(N^{1-\varepsilon}\right)^{1+\varepsilon^2}\left(\frac{1}{\min(q_1,N^{d-3\varepsilon D})}+\frac{1}{N^{1-\varepsilon}}+\frac{q_1}{\left(N^{1-\varepsilon}\right)^d}\right)^{\frac{1}{2^{d-1}}}\\
&\ll
N^{1-\varepsilon+\varepsilon^2}
\left(\frac{1}{N^{\varepsilon/2}}\right)^{\frac{1}{2^{d-1}}}\ll\frac{N^{1-\varepsilon}}{N^{d\varepsilon^2}}
\ll
\frac{N^{1-\varepsilon}}{q^{\varepsilon^2}}.
\end{align*}

We may therefore suppose $q<N^\varepsilon$. Then the coefficients of the polynomial $\psi(x):=\left(\theta-\frac{a}{q}\right)p(x)$ are $\ll_p \frac{W^{d-1}}{qN^{d-2\varepsilon D}}$, and hence
$$\Psi:=\max_{N-N^{1-\varepsilon}\leq x\leq N}|\psi'(x)|
\ll
\frac{W^{d-1}}{qN^{d-1.5\varepsilon D}}\cdot N^{d-1}\leq N^{2\varepsilon D-1}.$$
Moreover, the linear coefficient of $ap(x)$ is $\pm a$, which is coprime with $q$. Thus, using \Cref{ptheta} and \Cref{Va97Thm7.1}, we obtain
\begin{align*}
\sum_{\substack{N-Y\leq n< N}}e\left(p(n)\theta\right)&=\sum_{\substack{N-Y\leq n< N}}e\left(\frac{a}{q}p(n)+\psi(n)\right)\\
&=\frac{1}{q}\sum_{r=1}^qe\left(\frac{a}{q}p(r)\right)\sum_{N-Y\leq n<N}e(\psi(n))+O(q+qN^{2\varepsilon D-1}Y)\\
&\ll_d q^{\frac{-1}{d}+\varepsilon}\cdot Y+q+qN^{2\varepsilon D-1}Y\ll q^{\frac{-1}{d}+\varepsilon}N^{1-\varepsilon}.\qedhere
\end{align*}
\end{proof}

\begin{proof}[Proof of \Cref{UniformSumEstimateRationalPolyGeneral}]
If $\theta\in\M_{a,1}$ for some $a\in\mathbb{Z}$, that is, if $\|\theta\|\leq\frac{1}{N^{d-\varepsilon(d+1)}}$, then \Cref{UniformSumEstimateRationalPolyGeneral} follows from \Cref{q=1SumClose}. We may therefore assume $\|\theta\|>\frac{1}{N^{d-\varepsilon(d+1)}}$. First,
\begin{align}
\label{EZBound}
\left|\sum_{p(N-Y)\leq m< p(N)}e(m\theta)\right|\leq\frac{2}{\|\theta\|}\leq 2N^{d-\varepsilon(d+1)}.
\end{align}
Also, in this case we have $\theta\in\M_{a,q}$ for some $a\in\mathbb{Z}$ and some $2\leq q\leq N^{d-\varepsilon(d+1)}$. If $q\leq N^{\frac{\varepsilon}{3d}}$, then  \Cref{4erfds9ofdoRat} implies that, for all $Y\leq N^{1-\varepsilon}$,
\begin{equation}
\label{4erfds9ofdoRat2}
\sum_{N-Y\leq n< N}e\left(p(n)\theta\right)\ll_{d}N^{1-\varepsilon}w^{-\frac1d+\varepsilon}.
\end{equation}
If $q\geq N^{\frac{\varepsilon}{3d}}$, then, since $\frac{1}{N^{d-\varepsilon(d+1)}}\leq\frac{1}{N^{d-2\varepsilon D}}$, we may apply \Cref{betaaqRat}. Using $q^{\varepsilon^2}\geq N^{\frac{\varepsilon^3}{3d}}\geq w$, we see that \Cref{4erfds9ofdoRat2} holds in this case as well. Applying \Cref{ChangingCoefs}, we conclude that for all $Y\leq N^{1-\varepsilon}$,
\begin{equation}
\sum_{N-Y\leq n< N}p'(n)e\left(p(n)\theta\right)\ll_{d}|p'(N)|N^{1-\varepsilon}w^{-\frac1d+\varepsilon}
\ll\b_1 W^{d-1}N^{d-\varepsilon}w^{\frac{-1}{d}+\varepsilon},
\end{equation}
which, together with \Cref{EZBound}, concludes the proof.
\end{proof}

\msubsection{Proof of \Cref{mainp}}
\label{SecColorralphaSmall}

In this section we prove the following result, \Cref{DistanceColorThmDensityVersionCaseIII}, which implies \Cref{mainp} for sub-polynomials $\alpha(x)$ of
type ~\protect\hyperlink{CaseIII}{(III)}:

\begin{theorem}
\label{DistanceColorThmDensityVersionCaseIII}
Let $\rho\in(0,0.1)$ and let $\alpha(x)=c(p_\alpha(x)+r_\alpha(x))$ be a positive Hardy function, where $c>0$, $p_\alpha\in\mathbb{Q}[x]$ has degree $d\geq1$ and $1\prec r_\alpha(x)\ll_\alpha\log(x)$. Then there is a finite set $F=F(\rho,\alpha)\subseteq\mathbb{N}$ such that, for all sets $A,B\subseteq F$ with $|A|,|B|\geq\rho|F|$, there are distinct $x\in A,y\in B$ such that $|x+y-\alpha(n)|<\rho$ for some $n\in\mathbb{N}$.
\end{theorem}

\begin{corollary}
\label{ColoringThmForClosestInteger}
Let $\rho>0$ and let $\alpha(x)=c(p_\alpha(x)+r_\alpha(x))$ be a positive Hardy function, where $c>0$, $p_\alpha\in\mathbb{Q}[x]$ has degree $d\geq1$, $p_\alpha(0)=0$ and $1\prec r_\alpha(x)\ll_\alpha\log(x)$. Then there is a finite set $F=F(\rho,\alpha)\subseteq\mathbb{N}$ such that, for all sets $A,B\subseteq F$ with $|A|,|B|\geq\rho|F|$, there are distinct $x\in A,y\in B$ such that $x+y=\lfloor\alpha(n)\rfloor$ for some $n\in\mathbb{N}$. The same holds if we replace the floor function by the ceiling or closest integer functions.
\end{corollary}

\Cref{ColoringThmForClosestInteger} for the floor, ceiling and closest integer functions follows by applying \Cref{DistanceColorThmDensityVersionCaseIII} with $\rho=0.05$ to the functions $\alpha+\frac{1}{2},\alpha-\frac{1}{2}$ and $\alpha$, respectively.

In turn, \Cref{mainp} follows immediately from \Cref{ColoringThmForClosestInteger}. Indeed, by the pigeonhole principle, for every $k$-coloring of $\mathbb{N}$ and every finite set $F\subseteq\mathbb{N}$, some color class $A$ satisfies
\[
\frac{|A\cap F|}{|F|}\geq\frac{1}{k}.
\]

We now move to proving \Cref{DistanceColorThmDensityVersionCaseIII}.
\begin{lemma}
\label{DistanceColorThmDensityVersionSimpleCaseIII}
\Cref{DistanceColorThmDensityVersionCaseIII} is true in the case when $c=1$ and $p_\alpha$ is of the form
$$
p_\alpha(x)=\b_1x^d+\b_2x^{d-1}+\cdots+\b_{d-1}x^2+ux,\textup{ with }\b_i\in\mathbb{Z},\b_1>0,u\in\{1,-1\}.
$$
\end{lemma}

\begin{proof}
We use the notation of \Cref{ErgodicThmIntegerPolys}, and let $C_d$ denote the implied constant from \Cref{ErgodicThmIntegerPolys}. First choose $w\in\mathbb{N}$, with $w\geq w_0$, such that
\[
\frac{C_d}{w^{\frac{1}{d}-\varepsilon}}<\frac{\rho^2}{400}.
\]
As $1\prec r_\alpha(x)\ll_\alpha\log(x)$, there exists $\eta>0$ such that, for sufficiently large $N$, 
\begin{equation*}
|r_\alpha(N)-r_\alpha((1-\eta)N)|\leq\frac{\rho}{2}.
\end{equation*}
Now choose $N\in\mathbb{N}$ sufficiently large so that $w\leq\left\lfloor\varepsilon^3\log(N)\right\rfloor$ and $|r_\alpha(WN)-\b W|<\frac{\rho}{2}$ for some integer $\b$, and let $Y=(1-\eta)N$. Then \Cref{ErgodicThmIntegerPolys} can be applied, and for every $n\in\mathbb{N}\cap[Y,N]$ we have 
\begin{equation}
\label{RhoLittleInInterval}
|\alpha(Wn)-p_\alpha(Wn)-\b W|<\rho.
\end{equation}
As in \Cref{ErgodicThmIntegerPolys}, put $P(x)=p_\alpha(Wx)$, and define
\begin{equation}
\label{DefFiniteSetWhereWeWantBigDensity}
I=\left[\frac{P(Y)}{2},\frac{P(N)}{2}\right)
\textup{ and }
F=I\cap W\mathbb{Z}.
\end{equation}
\begin{claim}\label{Claim3eowi}
For all sets $C,D\subseteq F$ with $|C|,|D|>\frac{\rho}{4}|F|$, there are $c\in C,d\in D$ such that $c+d=P(n)$ for some $n\in[Y,N]$.
\end{claim}
We first check that \Cref{DistanceColorThmDensityVersionSimpleCaseIII} follows from \Cref{Claim3eowi}. Let $A,B\subseteq F$ satisfy $|A|,|B|>\rho|F|$. After removing some elements from $A$ and $B$ if necessary, we obtain sets $A_0\subseteq A$ and $B_0\subseteq B$ such that $|A_0|,|B_0|>\frac{\rho}{3}|F|$ and $A_0\cap B_0=\varnothing$.

Set $C=A_0$ and $D=F\cap(B_0-\b W)$. Clearly $|C|>\frac{\rho}{3}|F|$, and, since $\b\leq r_\alpha(WN)\leq\log(N)^2\leq\frac{\rho|F|}{12}$ for sufficiently large $N$, we obtain
\begin{equation*}
|D|\geq|B_0-\b W|-|(B_0-\b W)\setminus F|
\geq
|B_0|-\b
\geq
\frac{\rho}{3}|F|
-\log(N)^2>\frac{\rho}{4}|F|.
\end{equation*}
We can therefore apply \Cref{Claim3eowi} to obtain numbers $c\in C,d\in D$ such that $c+d=P(n)$. Hence, by \Cref{RhoLittleInInterval},
\[
|c+(d+\b W)-\alpha(Wn)|<\rho.
\]
Since $x:=c\in A_0$ and $y:=d+\b W\in B_0$, the numbers $x,y$ are distinct elements of $A,B$ with $|x+y-\alpha(Wn)|<\rho$, so we are done.

It remains to prove \Cref{Claim3eowi}. Let $L\in\mathbb{N}$ satisfy $3.1|I|< L\leq3.9|I|$. 
Consider the m.p.s. $(\mathcal{X},\mathcal{B},\mu,T)$, where $\mathcal{X}=\frac{\mathbb{Z}}{L\mathbb{Z}}$, so that
\begin{equation*}
3W|F|\leq |\mathcal{X}|\leq 4W|F|,
\end{equation*}
$\mu$ is normalized counting measure on $\mathcal{X}$ and $Tx=x+1$.
Let $\phi:W\mathbb{Z}\to \mathcal{X}$, $x\mapsto\overline{x}$, be the quotient map, and set $\overline{C}=\phi(C),\overline{D}=\phi(D)$. Since $\phi$ is injective on $F$, the assumption $|C|,|D|>\frac{\rho}{4}|F|$ implies that
\[
\mu\left(\overline C\right),\mu\left(\overline D\right)\geq\frac{\rho}{20W}.
\]
To prove \Cref{Claim3eowi}, it suffices to show that
\begin{equation}
\label{459rtefdore9odf}
\sum_{Y\leq n< N}P'(n)\mu\left(\overline{C}\cap\left(P(n)-\overline{D}\right)\right)>0.
\end{equation}
For all $x\in C$ and $y\in D$, the sum $x+y$ is an integer multiple of $W$ contained in the interval $[P(Y),P(N))$. Therefore, by a double-counting argument,
\begin{equation}
\label{5t9eirojir4}
\sum_{\frac{P(Y)}{W}\leq m<\frac{P(N)}{W}}|C\cap(Wm-D)|=|C|\cdot|D|.
\end{equation}
Moreover, for all $\frac{P(Y)}{W}\leq m<\frac{P(N)}{W}$, both sets $C$ and $Wm-D$ are contained in an interval of length $L$, so
\begin{equation}
\label{490refeioopp0}
\mu\left(\overline{C}\cap\left(Wm-\overline{D}\right)\right)=\frac{|C\cap(Wm-D)|}{L}.
\end{equation}
Combining \Cref{5t9eirojir4} and \Cref{490refeioopp0}, we obtain
\begin{multline}
\frac{W}{P(N)-P(Y)}\sum_{\frac{P(Y)}{W}\leq m<\frac{P(N)}{W}}
\mu\left(\overline{C}\cap\left(Wm-\overline{D}\right)\right)\geq\frac{W\cdot|C|\cdot|D|}{(P(N)-P(Y))\cdot L}\\
\geq
\frac{W\cdot\frac{\rho L}{20W}\cdot\frac{\rho L}{20W}}{L\cdot L}\geq\frac{\rho^2}{400W}.
\end{multline}
On the other hand, \Cref{ErgodicThmIntegerPolysRmk} implies that
\begin{multline*}
\Bigg|\frac{1}{P(N)-P(Y)}
\sum_{Y\leq n< N}P'(n)\mu\left(\overline{C}\cap\left(P(n)-\overline{D}\right)\right)
\\-\frac{W}{P(N)-P(Y)}\sum_{\substack{\frac{P(Y)}{W}\leq m<\frac{P(N)}{W}}}\mu\left(\overline{C}\cap\left(Wm-\overline{D}\right)\right)\Bigg|
\leq C_d
\frac{\sqrt{\mu\left(\overline{C}\right)\mu\left(\overline{D}\right)}}{w^{\frac{1}{d}-\varepsilon}}\leq\frac{C_d}{Ww^{\frac{1}{d}-\varepsilon}}.
\end{multline*}
By our choice of $w$, we have $\frac{C_d}{Ww^{\frac{1}{d}-\varepsilon}}<\frac{\rho^2}{400W}$, so \Cref{459rtefdore9odf} follows by the triangle inequality. This proves \Cref{Claim3eowi}.
\end{proof}

\begin{proof}[Proof of \Cref{DistanceColorThmDensityVersionCaseIII} from \Cref{DistanceColorThmDensityVersionSimpleCaseIII}]
After replacing $\alpha(x)$ by $\alpha(x+k)$ for a suitable $k\in\Z$ (which does not affect the validity of \Cref{DistanceColorThmDensityVersionCaseIII}), we may assume that the linear coefficient of $p_\alpha$ is nonzero. Thus we may write, for some $d\in\mathbb{N}$, $c_1\in(0,+\infty)$, $u\in\{-1,1\}$, and $\a_i\in\mathbb{Z}$ satisfying $\a_1,\a_d>0$,
\begin{equation*}
p_\alpha(x)=c_1(\a_1x^d+\a_2x^{d-1}+\cdots+\a_{d-1}x^2+\a_dux).
\end{equation*}
In particular, when $d=1$, we have $\a_d=\a_1>0$ and $u=1$. It suffices to prove \Cref{DistanceColorThmDensityVersionCaseIII} for the function $\alpha(\a_dx)$, whose polynomial part is
\begin{align*}
p_\alpha(\a_dx)&=c_1(\a_1\a_d^{d}x^d+\cdots+\a_{d-1}\a_d^{2}x^2+\a_d^2ux)\\
&=c(\b_1x^d+\cdots+\b_{d-1}x^2+ux),
\end{align*}
where $c=c_1\a_d^2\in(0,+\infty)$ and $\b_i=\a_i\cdot\a_d^{d-i-1}$.

We finish the proof by, essentially, scaling the construction from the proof of \Cref{DistanceColorThmDensityVersionSimpleCaseIII} by a factor of $c$. By replacing $\rho$ with a smaller positive number if necessary, we may assume that $\rho<c$. Set
\begin{align*}
\alpha_2(x)&=\frac{\alpha(\a_dx)}{c}=\b_1x^d+\cdots+\b_{d-1}x^2+ux+\frac{r_\alpha(\a_dx)}{c},\\
\rho_2&=\min\left(\frac{\rho}{3c},\frac{\rho^2}{40}\right)\in\left(0,\frac13\right)
.\end{align*}
By \Cref{DistanceColorThmDensityVersionSimpleCaseIII}, there exists a finite set $F_2\subseteq\mathbb{N}$ such that, whenever $A,B\subseteq F_2$ satisfy $|A|,|B|\geq\rho_2|F_2|$, there are distinct $x\in A,y\in B$ such that $|x+y-\alpha_2(n)|<\rho_2$ for some $n\in\mathbb{N}$. In the proof of \Cref{DistanceColorThmDensityVersionSimpleCaseIII}, this set has the form
\begin{equation*}
F_2=I\cap W\mathbb{Z}.
\end{equation*} 
We regard $W$ as fixed; by increasing the number $w$ from that proof if necessary, we may assume that $cW\geq2$. The interval $I$, however, can be chosen with arbitrarily large length, independently of $W$; denote this length by $L$. In particular, $|F_2|\leq\frac{2L}{W}$. For $x\in\mathbb{R}$ and $A\subseteq\R$, write $d(x,A)=\min_{a\in A}|x-a|$, and define
\begin{align*}
F&=\left\{x\in\mathbb{Z}\cap c I;d(x,c F_2)<\frac{\rho}{3}\right\}.
\end{align*}
For sufficiently large $L$, \Cref{IntegersInDifferentIntervals} gives
\begin{equation*}
|F|\geq\frac{c L\rho}{20c W}=\frac{L\rho}{20W}\geq \frac{\rho}{40}|F_2|,
\end{equation*}
since $F$ agrees, up to at most two endpoint errors, with the set
\begin{align*}
\left\{x\in\mathbb{Z}\cap c I;d(x,c W\mathbb{Z})<\frac{\rho}{3}\right\}.
\end{align*}

For each $x\in F$, let $x_2$ be the unique element of $F_2$ such that $|x-c x_2|<\frac{\rho}{3}$; uniqueness follows from $\rho\leq c$. Similarly, for each subset $X\subseteq F$, set $X_2=\{x_2:x\in X\}$, so that $|X_2|=|X|$.

To conclude the proof, let $A,B\subseteq F$ satisfy $|A|,|B|\geq\rho|F|$. Then
\[
|A_2|=|A|\geq\rho|F|\geq\rho\frac{\rho}{40}|F_2|\geq\rho_2|F_2|,
\]
and similarly $|B_2|\geq\rho_2|F_2|$. By \Cref{DistanceColorThmDensityVersionSimpleCaseIII}, there exist distinct integers $x_2\in A_2$ and $y_2\in B_2$, corresponding to distinct $x\in A$ and $y\in B$, such that $|x_2+y_2-\alpha_2(n)|\leq\rho_2$. The triangle inequality now gives
\begin{align*}
|x+y-\alpha(\a_dn)|&\leq|x-c x_2|+|y-c y_2|+|c x_2+c y_2-c\alpha_2(n)|\\
&< \frac{\rho}{3}+\frac{\rho}{3}+c\rho_2\leq\rho.\qedhere
\end{align*}
\end{proof}

\begin{lemma}
\label{IntegersInDifferentIntervals}
Let $0<\delta<\frac{1}{2}$ and $\lambda\geq2$. There is $L_0(\delta,\lambda)>0$ such that, for any interval $J\subseteq\mathbb{R}$ with length $L\geq L_0$, we have
\begin{equation*}
\left|\left\{x\in\Z\cap J;d(x,\lambda\mathbb{Z})<\delta\right\}\right|\geq\frac{L\delta}{4\lambda}.
\end{equation*}
\end{lemma}

\begin{proof}
Let $\mathcal{X}:=\overline{\left\{\frac{x}{\lambda}\textup{ mod }1:x\in\mathbb{Z}\right\}}\subseteq\mathbb{T}$, with Haar measure $m_{\mathcal{X}}$. The rotation $T:\mathcal{X}\to \mathcal{X}$, $x\mapsto x+\frac{1}{\lambda}$, is uniquely ergodic, so for all $\varepsilon>0$ we have
\begin{equation}
\label{4refdscxoilkoi}
\lim_{M-N\to\infty}\frac{1}{M-N}\left|N\leq n<M;\left\|\frac{n}{\lambda}\right\|<\varepsilon\right|=m_{\mathcal{X}}((-\varepsilon,\varepsilon)_{\mathbb{T}})\geq\varepsilon.
\end{equation}
Taking $\varepsilon=\frac{\delta}{\lambda}$, we have
\begin{equation}
\label{refdvcoierfd9oi}
\left\{x\in\Z\cap J;d(x,\lambda\mathbb{Z})<\delta\right\}=\left\{x\in\Z\cap J;\left\|\frac{x}{\lambda}\right\|<\varepsilon\right\}.
\end{equation}
Therefore, by \Cref{4refdscxoilkoi}, if $J$ is long enough, the right-hand side of \Cref{refdvcoierfd9oi} has cardinality at least $|\mathbb{Z}\cap J|\cdot\frac{\varepsilon}{2}\geq\frac{L\delta}{4\lambda}$.
\end{proof}

The proof of \Cref{DistanceColorThmDensityVersionSimpleCaseIII} also gives the following version for integer polynomials.
\begin{PolynomialDensityIntro}
Let $\delta>0$ and let $p\in\mathbb{Z}[x]$ satisfy $\lim_{x\to+\infty}p(x)=+\infty$. Suppose that $p(n)$ is even for some $n\in\N$. Then there exist $b,W\in\N$ such that, for every $A\subseteq\mathbb{N}$ satisfying
\begin{equation*}
\overline{d}_{b+W\N}(A):=\limsup_N\frac{\left|A\cap\{b+W,b+2W,\dots,b+NW\}\right|}{N}>\delta,
\end{equation*}
there exist distinct $x,y\in A$ such that $x+y=p(n)$ for some $n\in\mathbb{N}$.
\end{PolynomialDensityIntro}

\begin{proof}
We may assume that $0<\delta<1$, since otherwise the hypothesis on the relative upper density is never satisfied.

We first prove the result when $p$ has degree $d\geq1$ and is of the form
\begin{equation}\label{SpecialFormPolynomialDensity}
p(x)=\b_1x^d+\b_2x^{d-1}+\cdots+\b_{d-1}x^2+ux,
\end{equation}
where $u\in\{1,-1\}$ and, when $d\geq2$, $\b_i\in\mathbb{Z}$ and $\b_1>0$. 
We follow the proof of \Cref{DistanceColorThmDensityVersionSimpleCaseIII}, for $\rho=\frac{\delta}{3^d}$ and $\alpha(x)=p(x)$, making the following minor changes to account for the fact that $r_\alpha(x)=0$. We take $\eta=\frac{1}{2}$, and choose any sufficiently large $N$, because \Cref{RhoLittleInInterval} is trivially satisfied. We define, letting $N=2^k$ be a power of $2$,
\begin{equation*}
I_k=\left[\frac{P\left(2^{k-1}\right)}{2},\frac{P\left(2^{k}\right)}{2}\right)\textup{ and }F_k=I_k\cap W\Z.
\end{equation*}
It follows that there is $k_0\in\mathbb{N}$ such that, for all $k\geq k_0$ and every set $A\subseteq\mathbb{N}$ satisfying $
\frac{|A\cap F_k|}{|F_k|}>\rho$,
there exist distinct $x,y\in A$ such that $x+y=p(n)$. This proves the special case, since $\overline{d}_{W\N}(A)>\delta$ implies $\frac{|A\cap F_k|}{|F_k|}>\rho$ for arbitrarily large $k$.

We now reduce the general case to the special case just proved. First, observe that replacing $p(x)$ by $p(x+a)$, where $a\in\N$ is a natural number such that $p(a)$ is even and $p'(a)>0$, does not affect the conclusion of \Cref{eroidfklsjoifks}. 
Therefore, we may assume that $p(0)$ is even and $p'(0)>0$.

Next, translating the arithmetic progression $W\N$ allows us to remove the constant term of $p$. Indeed, put $b=p(0)/2$. If the result holds for $p(x)-p(0)$ along $W\N$, then it holds for $p$ along $b+W\N$. Thus, we may assume that $p(0)=0$ and $h:=p'(0)>0$.

Finally, $q(x):=\frac{p(hx)}{h^2}$ has the form in \eqref{SpecialFormPolynomialDensity}. Therefore, by the special case applied to $q(x)$, there is $W$ such that for any $A\subseteq\mathbb{N}$ satisfying $\overline{d}_{h^2W\N}(A)>\delta$, there are distinct $x,y\in A$ and $n\in\N$ such that $x+y=p(hn)$.
\end{proof}

As explained in \Cref{SecIntro}, \Cref{KSt} follows immediately from \Cref{eroidfklsjoifks}.

\begin{proof}[Proof of \Cref{mainp1} from \Cref{mainp} and
\Cref{KSt}]
Since \(\alpha\) is a sub-polynomial, for every \(p\in\Q[x]\), either
\[
|\alpha(x)-p(x)|\longrightarrow+\infty
\]or \(\alpha(x)-p(x)\) converges to a finite limit as \(x\to+\infty\). If the former holds for every \(p\in\Q[x]\), then the conclusion follows directly from \Cref{mainp}. 
Otherwise, there exist \(p\in\Q[x]\), $C\in\R$ and a Hardy field function $r(x)\to0$ such that
\begin{equation*}
\alpha(x)=p(x)+C+r(x).
\end{equation*}
Therefore, there is a constant $C_0$ close to $C$ such that $\lfloor\alpha(n)\rfloor=\lfloor p(n)+C_0\rfloor$ for all big enough $n$.
By assumption, \(\lfloor\alpha(n)\rfloor\) is even for infinitely many \(n\). We may therefore choose some \(B\in\N\) such that $\lfloor p(B)+C_0\rfloor$ is even. We then choose \(A\in\N\) so that $p(Ax+B)-p(B)\in\Z[x]$. Consequently, the polynomial
\[
q(x):=p(Ax+B)-p(B)+\lfloor p(B)+C_0\rfloor
\]
has integer coefficients and takes an even value. Moreover,
\[
q(n)=\lfloor p(An+B)+C_0\rfloor
     =\lfloor\alpha(An+B)\rfloor
\]for every \(n\in\N\). The desired conclusion now follows by applying the Khalfalah--Szemerédi theorem to \(q\).
\end{proof}

\msection{Some counterexamples and remarks}
\label{SecRemarksmaind}

In this section, we first make explicit the Rohlin-tower construction underlying the non-uniformity observation made in \Cref{SecIntro} after \Cref{ergodicconverg}. We then check why \Cref{maind} fails if some of its hypotheses are weakened. Finally, in \Cref{LinearDifferenceCounterexample}, we show that \Cref{LinearDifferenceEquationCorollary}, which deals with equations of the form $
kx-\ell y=\lfloor\alpha(n)\rfloor$,
may fail when $k<\ell$.
\begin{lemma}
\label{RohlinTowersNonUniformity}
For every $N_0\in\N$ and every standard aperiodic m.p.s. $(X,\mathcal{B},\mu,T)$, there exists $A\in\mathcal{B}$ such that
\[
\left\|
\frac{1}{N}\sum_{n=1}^N1_{T^{-n}A}
-P_T(1_A)
\right\|_{L^2(X)}
\geq 0.1
\]
for every $1\leq N\leq N_0$.
\end{lemma}

\begin{proof}
By the Rohlin lemma (see for example \cite[Theorem 10.5.1]{CFS82}), there exists $B\in\mathcal{B}$ such that the sets $B,TB,\dots,T^{100N_0-1}B$ are pairwise disjoint and the complement of their union has measure less than $0.01$:
\[
\mu\left(X\setminus\bigcup_{j=0}^{100N_0-1}T^jB\right)<0.01.
\]
Set
\[
A:=\bigcup_{j=0}^{50N_0-1}T^jB
\qquad\text{and}\qquad
C:=T^{50N_0}A.
\]
We then have $\mu(A)\in(0.49,0.5]$
and $A\cap C=\varnothing$. Moreover, $1_A$ and $1_C$ have the same projection onto the $U_T$-invariant functions, and their distances from this common projection are equal. Hence, by the triangle inequality,
\begin{align*}
\|1_A-P_T1_A\|_{L^2(X)}
\geq\frac{1}{2}\|1_A-1_C\|_{L^2(X)}
>0.4.
\end{align*}
On the other hand, for every $1\leq n\leq N_0$, $
\mu(T^nA\mathbin{\Delta}A)=2n\mu(B)<0.02$.
Consequently, for every $1\leq N\leq N_0$,
\begin{equation*}
\left\|1_A-\frac{1}{N}\sum_{n=1}^N1_{T^{-n}A}\right\|_{L^2(X)}
\leq\frac{1}{N}\sum_{n=1}^N\|1_A-1_{T^{-n}A}\|_{L^2(X)}<\sqrt{0.02}<0.15.
\end{equation*}
The conclusion now follows from the triangle inequality.
\end{proof}

We now show that the condition $\overline{d}(A,B)>0$ in \Cref{maind} cannot, in general, be replaced by the weaker condition $\overline{d}(A),\overline{d}(B)>0$.
\begin{lemma}
\label{RemarkMaind1}
Let $\alpha(x)\succ x$ be a Hardy field function. Then there exist sets $A,B\subseteq\mathbb{N}$ such that $
\overline{d}(A)=\overline{d}(B)=1$
and there are no $x\in A$, $y\in B$, and $n\in\mathbb{N}$ such that $
x+y=\lfloor\alpha(n)\rfloor$.
\end{lemma}

\begin{proof}
Set $
S:=\{\lfloor\alpha(n)\rfloor:n\in\mathbb{N}\}\cap\mathbb{N}$.
Since $\alpha(x)\succ x$, the set $S$ has density zero.
Since $\underline{d}(S)=0$, we can recursively choose a sufficiently fast-growing sequence $(N_k)_{k\in\mathbb{N}}$ such that $N_1=1$ and, for every $k\geq2$, the set
\[
A_k:=\{1,\dots,N_k\}\setminus\bigcup_{j=1}^{N_{k-1}}(S-j)
\]
satisfies
\[
\frac{|A_k|}{N_k}\geq1-\frac{1}{k}.
\]
If $l<k$, then $A_l\subseteq\{1,\dots,N_{k-1}\}$. So by definition of $A_k$, we have $A_k\cap(S-A_l)=\varnothing$, or equivalently, $
(A_k+A_l)\cap S=\varnothing$.
Now set
\[
A:=\textstyle\bigcup_{k\in\mathbb{N}}A_{2k}
\qquad\text{and}\qquad
B:=\textstyle\bigcup_{k\in\mathbb{N}}A_{2k+1}.
\]
The preceding observation shows that $(A+B)\cap S=\varnothing$. Moreover,
\[
\frac{|A\cap[1,N_{2k}]|}{N_{2k}}
\geq\frac{|A_{2k}|}{N_{2k}}
\geq1-\frac{1}{2k},
\]
and the analogous estimate along $(N_{2k+1})_{k\in\mathbb{N}}$ holds for $B$. Hence $\overline{d}(A)=\overline{d}(B)=1$.
\end{proof}

We now show that both \Cref{maind} and \Cref{mainp} fail for any function $\alpha(x)$ of exponential growth. 

\begin{lemma}
\label{4re9fdosil}
Assume that $(a_n)$ is a \emph{lacunary} sequence of natural numbers, meaning that for some $\delta>0$, we have $
a_{n+1}>(1+\delta)a_n$
for every $n\in\mathbb{N}$. Then there exists a finite coloring of $\mathbb{N}$ with no distinct monochromatic $x,y$ satisfying $x+y=a_n$ for any $n\in\mathbb{N}$.
\end{lemma}

\begin{proof}
We first consider the case in which $a_{n+1}>2a_n$ for every $n\in\mathbb{N}$. We construct a $2$-coloring $c\colon\mathbb{N}\to\{0,1\}$ recursively on the intervals $[a_{n-1},a_n)$. First, for $1\leq x<a_1$, set
\[
c(x):=
\begin{cases}
0,&1\leq x<\dfrac{a_1}{2},\\
1,&\dfrac{a_1}{2}\leq x<a_1.
\end{cases}
\]
For every $n\geq2$ and $a_{n-1}\leq x<a_n$, set
\[
c(x):=
\begin{cases}
0,&a_{n-1}\leq x<\dfrac{a_n}{2},\\
1,&\dfrac{a_n}{2}\leq x<a_n-a_{n-1},\\
1-c(a_n-x),&a_n-a_{n-1}\leq x<a_n.
\end{cases}
\]

We claim that this coloring has the required property. Suppose that $x+y=a_n$ for some $n\in\mathbb{N}$ and some distinct $x,y\in\mathbb{N}$, and assume that $x<y$. Then $x<a_n/2$, and the definition of $c$ immediately gives $c(x)\neq c(a_n-x)=c(y)$.

Now let $\delta>0$ be arbitrary, and choose $l\in\mathbb{N}$ such that $(1+\delta)^l\geq2$. It follows that $
a_{n+l}>2a_n$
for every $n\in\mathbb{N}$. For each $i\in\{1,\dots,l\}$, apply the first part of the proof to the subsequence $(a_{ln+i})_{n\in\mathbb{N}\cup\{0\}}$. This gives a coloring $c_i\colon\mathbb{N}\to\{0,1\}$ with no distinct $c_i$-monochromatic $x,y$ satisfying $x+y=a_{ln+i}$ for any $n\in\mathbb{N}\cup\{0\}$. Define
\[
c\colon\mathbb{N}\to\{0,1\}^l,
\qquad
c(x):=(c_1(x),\dots,c_l(x)).
\]
If $x$ and $y$ have the same $c$-color, then they have the same $c_i$-color for every $i\in\{1,\dots,l\}$. Since every term of $(a_n)$ belongs to one of the $l$ subsequences above, we cannot have $x+y=a_n$ for any $n\in\mathbb{N}$.
\end{proof}

The number of colors in \Cref{4re9fdosil} cannot, in general, be reduced to two. For example, suppose that $
a_{n+1}<1.1a_n$
for every $n\in\mathbb{N}$, and that $a_n+a_{n+1}+a_{n+2}$ is even for arbitrarily large $n$. Fix such an $n$ and set $b_1=a_n$, $b_2=a_{n+1}$, and $b_3=a_{n+2}$. Since $(a_n)$ is increasing and $b_3<1.1^2b_1$, we have
\[
b_1+b_2+b_3>2\max\{b_1,b_2,b_3\}.
\]
Consequently, the three numbers
\[
\frac{b_1+b_2-b_3}{2},
\qquad
\frac{b_1-b_2+b_3}{2},
\qquad
\frac{-b_1+b_2+b_3}{2}
\]
are distinct natural numbers, and their pairwise sums are $b_1$, $b_2$, and $b_3$. In every $2$-coloring of $\mathbb{N}$, two of these three numbers have the same color; hence there are arbitrarily large distinct monochromatic $x,y$ such that $x+y=a_m$ for some $m\in\mathbb{N}$.

More generally, for every \emph{fixed} $k\in\mathbb{N}$, there exist $\delta>0$ and a sequence of natural numbers $(a_n)$ satisfying
\[
a_{n+1}>(1+\delta)a_n
\]
for every $n\in\mathbb{N}$, such that every $k$-coloring of $\mathbb{N}$ contains distinct monochromatic $x,y$ with $x+y=a_n$ for some $n\in\mathbb{N}$. Indeed, it is enough to construct $(a_n)$ so that, for arbitrarily large $N$, it contains the entire set
\[
S_N:=\{N,2N,\dots,(2k+1)N\}.
\]
For any $k$-coloring, two of the $k+1$ numbers $N,2N,\dots,(k+1)N$ have the same color, and their sum belongs to $S_N$. Although for $k\geq3$ this particular construction cannot arise from an exponentially growing Hardy function, we do not know whether, for every fixed $k$, there exists an exponentially growing Hardy function $\alpha$ such that every $k$-coloring of $\mathbb{N}$ contains distinct monochromatic $x,y$ satisfying $x+y=\lfloor\alpha(n)\rfloor$ for some $n\in\mathbb{N}$.

For a set $A\subseteq\N$, define its \emph{upper Banach density} by
\[
d^*(A):=
\limsup_{\substack{1\leq M<N\\N-M\to\infty}}
\frac{|A\cap[M,N-1]|}{N-M}.
\]
The Furstenberg--S\'ark\"ozy theorem for equations of the form $x-y=p(n)$, stated at the beginning of \Cref{SecIntro}, and its extension to sub-polynomials, \Cref{FW09ThmC}, remain valid when the condition $\overline{d}(A)>0$ is replaced by $d^*(A)>0$.

By contrast, the following result, applied to $a_n=\lfloor\alpha(n)\rfloor$, shows that the upper density in \Cref{maindc} cannot be replaced by upper Banach density if $\alpha(x)\succ x$.
\begin{lemma}
\label{UpperBanachDensityCounterexample}
Let $(a_n)$ be a strictly increasing sequence of natural numbers such that $
a_{n+1}-a_n\longrightarrow\infty$. Then there exists a set $A\subseteq\N$ such that
\[
d^*(A)=1
\qquad\text{and}\qquad
(A+A)\cap\{a_n:n\in\N\}=\varnothing.
\]
\end{lemma}

\begin{proof}
Set $S:=\{a_n:n\in\N\}$. We will recursively choose natural numbers $k_N$ so that, for every $N\in\N$, the set
\[
A_N:=\bigcup_{j=1}^{N}\{k_j,k_j+1,\dots,k_j+j\}
\]
satisfies $(A_N+A_N)\cap S=\varnothing$. Then, letting $A=\cup_NA_N$, we are done.

Set $A_0:=\varnothing$, and suppose that $k_1,\dots,k_{N-1}$ have already been chosen. As $S$ has density $0$, the set of $k\in\N$ for which
\[
\bigl(A_{N-1}\cup\{k,k+1,\dots,k+N\}\bigr)
+\bigl(A_{N-1}\cup\{k,k+1,\dots,k+N\}\bigr)
\]
intersects $S$ has density zero. We may therefore choose $k_N$, completing the recursive step.
\end{proof}

We conclude this section by showing that \Cref{LinearDifferenceEquationCorollary}, which deals with equations of the form $
kx-\ell y=\lfloor\alpha(n)\rfloor$,
may fail when $k<\ell$.

\begin{lemma}
\label{LinearDifferenceCounterexample}
Let $k,\ell\in\N$ satisfy $k<\ell$, and let $\alpha(x)$ be an eventually positive Hardy function such that $\alpha(x)\succ x$. Then there exists $A\subseteq\N$ with $\overline d(A)>0$ such that there are no $x,y\in A$ and $n\in\N$ satisfying
\[
kx-\ell y=\lfloor\alpha(n)\rfloor.
\]
\end{lemma}

\begin{proof}
The proof is similar to that of \Cref{RemarkMaind1}.
Set $
S:=\{\lfloor\alpha(n)\rfloor:n\in\N\}\cap\N$.
Since $\alpha(x)\succ x$, the set $S$ has density zero. Choose $\eta>0$ so small that $k<(1-\eta)\ell$.
Since $S$ has density zero, we can recursively choose a sufficiently fast-growing sequence $(M_j)_{j\in\N}$ such that, writing
\[
\begin{aligned}
I_j&:=\N\cap[(1-\eta)M_j,M_j],
\qquad B_{j-1}:=\bigcup_{i<j}A_i,\\
A_j&:=I_j\setminus\{x\in I_j:kx-\ell y\in S\textup{ for some }y\in B_{j-1}\},
\end{aligned}
\]
we have $|A_j|\geq \eta M_j/2$ for every $j$, and $kM_{j-1}<\ell(1-\eta)M_j$ for every $j\geq2$. Indeed, once $A_1,\ldots,A_{j-1}$ have been chosen, the number of excluded elements of $I_j$ is at most $|B_{j-1}|\cdot |S\cap[1,kM_j]|=o(M_j)$.

Set $A:=\bigcup_{j\in\N}A_j$. Then $\overline d(A)>0$, since $|A_j|\geq\eta M_j/2$ for every $j$. We claim that $(kA-\ell A)\cap S=\varnothing$. Indeed, take $x\in A_i$ and $y\in A_j$. If $i\leq j$, then
\[
kx-\ell y\leq kM_j-\ell(1-\eta)M_j<0.
\]
Finally, if $i>j$, then $kx-\ell y\notin S$ by the definition of $A_i$. This proves the claim and completes the proof.
\end{proof}

\msection{Subpolynomials $\alpha(x)$ for which \Cref{maindc} fails}
\label{SecObstructionsDensityTheorem}

In this section, we prove \Cref{HardysCloseToPolysDensityResult} and \Cref{Alphaclosetolinear}, which together with \Cref{maindc} give an exact characterization of the sub-polynomials $\alpha(x)$ that do not satisfy the conclusion of \Cref{maindc}

The eventually positive sub-polynomials $\alpha(x)$ such that $\lim_{x\to+\infty}\alpha(x)=+\infty$ and which do not satisfy the hypotheses of \Cref{maindc} are those of the form
\begin{equation}
\label{HardyFunctionsCloseToRationalPolys}
\alpha(x)=c p(x)+r(x),
\end{equation}
where $c>0$, $r(x)\ll_r\log(x)$, and either $p=0$ or, for some $d\in\N$,
\[
p(x)=\a_dx^d+\a_{d-1}x^{d-1}+\cdots+\a_1x\in\mathbb{Z}[x],
\qquad \a_d>0.
\]
When $p=0$, the assumption that $\alpha(x)\to+\infty$ amounts to $r(x)\to+\infty$. By \Cref{HardysCloseToPolysDensityResult}, the conclusion of \Cref{maindc} fails whenever $\deg(p)\geq2$. The remaining case, in which $p=0$ or $\deg(p)=1$, is studied in \Cref{Alphaclosetolinear}.

\begin{HardysCloseIntro}
Let $\alpha(x)$ be an eventually positive sub-polynomial of the form 
\begin{equation*}
\alpha(x)=c p(x)+r(x),
\end{equation*}
where $c>0$, $p\in\mathbb{Z}[x]$, $\deg(p)\geq2$, and $r(x)\ll_r\log(x)$. Then there exists $A\subseteq\mathbb{N}$ with $\overline{d}(A)>0$ such that such that there are no \(x,y\in A\) and \(n\in\mathbb N\) satisfying \(x+y=\lfloor\alpha(n)\rfloor\).
\end{HardysCloseIntro}

The main ingredient in the proof of \Cref{HardysCloseToPolysDensityResult} is the following lemma, which we prove afterward.

\begin{lemma}
\label{PolysHaveSparseResidues}
Let $p\in\mathbb{Z}[x]$ have degree at least $2$. Then for every $\varepsilon>0$ there is $N\in\mathbb{N}$ such that $\{p(n);n\in\mathbb{Z}\}$ intersects at most $\varepsilon N$ residue classes modulo $N$.
\end{lemma}

\begin{proof}[Proof of \Cref{HardysCloseToPolysDensityResult} from \Cref{PolysHaveSparseResidues}]
For each real number $C>0$, let $\mathbb{T}_C:=\mathbb{R}/C\mathbb{Z}$, and write $\overline{x}:=x+C\mathbb{Z}$ for the coset of $x\in\mathbb{R}$. We use the notation
$$d_{\mathbb{T}_C}(x,y)=
d_{\mathbb{T}_C}(\overline{x},\overline{y})=\min\{|x-y-nC|;n\in\mathbb{Z}\}$$
for either $x,y\in\mathbb{R}$ or $\overline{x},\overline{y}\in\mathbb{T}_C$. By \Cref{PolysHaveSparseResidues}, we can find a real number $C>100$, of the form $C=cN$ for some $N\in\mathbb{N}$, such that the set of residues
$$R(C,p):=\{c p(n)+C\mathbb{Z};n\in\mathbb{Z}\}\subseteq\mathbb{T}_C$$ has at most $\frac{C}{100}$ elements. We now consider two cases.
\begin{enumerate}[leftmargin=*]
    \item Suppose that $r(x)$ converges to a constant $\lambda$. By adding a constant term to $p$ if necessary, we may assume that $r(x)\prec1$; we can still apply \Cref{PolysHaveSparseResidues} to $p$ even if its constant term is not an integer. Choose $C>100$ such that $|R(C,p)|\leq\frac{C}{100}$. Then there exists $n_0\in\mathbb{N}$ such that
$$d_{\mathbb{T}_C}\left(\overline{2n_0},R(C,p)\right)>5.$$ 
Indeed, there exists $z\in\mathbb{T}_C$ such that $d_{\mathbb{T}_C}(z,R(C,p))>6$, because, for each $y\in R(C,p)$, the set
$$\{z\in\mathbb{T}_C;d_{\mathbb{T}_C}(z,y)\leq6\}$$
is an interval of length $12$, and at most $\frac{C}{100}$ such intervals cannot cover $\mathbb{T}_C$. Choosing an even integer $2n_0$ such that $|z-2n_0|\leq1$ gives the claim.

    Now define $A=\{n\in\mathbb{N};d_{\mathbb{T}_C}(n,n_0)<1\}$. For every $x,y\in A$, we have $d_{\mathbb{T}_C}(x+y,2n_0)<2$, and hence $d_{\mathbb{T}_C}(x+y,R(C,p))>3$. Consequently, $|x+y-cp(n)|>3$ for every $n\in\mathbb{Z}$. Since $\alpha(x)=cp(x)+o(1)$, we cannot have $x+y=\lfloor\alpha(n)\rfloor$ for any sufficiently large $n$.

    Moreover, $A$ is syndetic\footnote{We say that $E\subseteq\mathbb{N}$ is syndetic if there exists $L>0$ such that $E\cap[n,n+L]\neq\varnothing$ for every $n\in\mathbb{N}$.}, because $A-n_0$ is a Bohr neighborhood of $0$. Therefore, $\underline{d}(A)>0$.

\item Suppose that $1\prec |r(x)|$. Let $L:=\lim_{x\to+\infty}\frac{r(x)}{\log(x)}\in\mathbb{R}$, so that $r(x)=L\log(x)+r_1(x)$ for some $r_1(x)\prec\log(x)$. Since $r(n+1)-r(n)\to0$, the sequence $(r(n))$ comes arbitrarily close to multiples of $C$. Moreover,
\[
r(2x)-r(x)\longrightarrow L\log2.
\]
It follows that there exists an increasing sequence of natural numbers $k_j\to\infty$, with $k_{j+1}\geq2k_j$, such that $d_{\mathbb{T}_C}(r(n),0)<1+|L|$ for every $n\in[k_j,2k_j]$. Choose $C>100(1+|L|)$ such that $|R(C,p)|<\frac{C}{100(1+|L|)}$. As above, we can find $n_0$ such that $$d_{\mathbb{T}_C}\left(\overline{2n_0},R(C,p)\right)>4+|L|.$$
We then define
$$A=\{n\in\mathbb{N};d_{\mathbb{T}_C}(n,n_0)<1\}\cap\bigcup_{j\in\mathbb{N}}[\alpha(k_j),1.1\alpha(k_j)].
$$
Suppose that, for some sufficiently large $x,y\in A$ and some $n\in\mathbb{N}$, we have $x+y=\lfloor\alpha(n)\rfloor$. By the definition of $A$, and since $\alpha$ grows like a polynomial of degree at least $2$, we must have $n\in[k_j,2k_j]$ for some $j\in\mathbb{N}$, namely, for the value of $j$ such that $\max(x,y)\in[\alpha(k_j),1.1\alpha(k_j)]$. Hence, $d_{\mathbb{T}_C}(r(n),0)\leq1+|L|$. By the definition of $A$ and the triangle inequality,
\begin{equation*}
d_{\mathbb{T}_C}(2n_0,cp(n))
\leq 
d_{\mathbb{T}_C}(x+y,cp(n))+2
\leq 
d_{\mathbb{T}_C}(x+y,\alpha(n))+3+|L|
\leq 4+|L|.
\end{equation*}
This contradicts the definition of $n_0$. Finally, since $\{n\in\mathbb{N}:d_{\mathbb{T}_C}(n,n_0)<1\}$ is syndetic, the set $A$ has positive upper density. This completes the proof.\qedhere

\end{enumerate}
\end{proof}

We deduce \Cref{PolysHaveSparseResidues} from the following result. Note that any polynomial $p\in\mathbb{Z}[x]$ define for each $q\in\mathbb{N}$ a map 
\begin{align*}
p_q:\mathbb{Z}/q\mathbb{Z}&\to\mathbb{Z}
/q\mathbb{Z};\\
n+q\mathbb{Z}&\mapsto p(n)+q\mathbb{Z}.
\end{align*}

\begin{lemma}
\label{PolysNotSurjectiveModPrimes}
Let $p\in\mathbb{Z}[x]$ have degree at least $2$. Then the set $A_p$ of primes $q$ such that $p_q:\mathbb{Z}/q\mathbb{Z}\to\mathbb{Z}/q\mathbb{Z}$ is not 
surjective has positive relative lower density in the set of primes; that is,
\[
\underline{d}_{\mathcal{P}}(A_p)
:=
\liminf_{N\to\infty}
\frac{|A_p\cap[1,N]|}{|\mathcal{P}\cap[1,N]|}
>0.
\]
\end{lemma}

\begin{proof}
By \cite[Theorem 2]{Tu},\footnote{ MathOverflow user `so-called friend Don' pointed out that one can also prove \Cref{PolysNotSurjectiveModPrimes} using the better-known Chebotarev density theorem; see \url{https://mathoverflow.net/q/512856}.} all polynomials $p\in\mathbb{Z}[x]$ such that $\mathcal{P}\setminus A_p$ is infinite must be of the form $p(x)=aD(p_1(x))+b$, for some $a,b\in\mathbb{Q}$, $p_1\in\mathbb{Q}[x]$ and $D\in\mathbb{Z}[x]$ is a Dickson polynomial of degree $d\geq2$. Therefore, there exists $N\in\mathbb{N}$ such that all values of $p(n), n\in\mathbb{N}$ are of the form $aD\left(\frac{k}{N}\right)+b$, for some $k\in\mathbb{N}$. So, for any sufficiently large prime $q$ (which does not appear as a prime factor in $a,b,N$ or the coefficients of $D$, so that the formula $k\mapsto aD\left(\frac{k}{N}\right)+b$ is well defined in $\mathbb{Z}/q\mathbb{Z}$), the value of $p(n)$ modulo $q$ is of the form $aD\left(\frac{k}{N}\right)+b$ for some $k\in\mathbb{Z}/q\mathbb{Z}$. But if $q\equiv1$ mod $d$, then by \cite[Lemma 1.4]{Tu} the map $k\mapsto D(k)$ is not surjective as a map $\mathbb{Z}/q\mathbb{Z}\to\mathbb{Z}/q\mathbb{Z}$, so $p_q$ cannot be surjective either, and we are done by Dirichlet's theorem on arithmetic progressions.
\end{proof}

\begin{proof}[Proof of \Cref{PolysHaveSparseResidues} from \Cref{PolysNotSurjectiveModPrimes}]
Let $q_1<q_2<\dots$ be the increasing enumeration of all the primes $q$ such that $p_q:\mathbb{Z}/q\mathbb{Z}\to\mathbb{Z}/q\mathbb{Z}$ is not 
surjective. By \Cref{PolysNotSurjectiveModPrimes} and the prime number theorem, there is a constant $C>0$ such that $q_n<Cn\log(n)$ for big enough $n$. By the Chinese remainder theorem, the number of residues of $p$ modulo $N_n:=q_1\cdots q_n$ is at most
\begin{equation}
\label{tre9fdop9ofpdl}
N_n\cdot
\left(1-\frac{1}{q_1}\right)\cdots\left(1-\frac{1}{q_n}\right).
\end{equation}
Since $\sum_{n=1}^{\infty}1/q_n=+\infty$, for every $\varepsilon>0$ the expression in \eqref{tre9fdop9ofpdl} is less than $\varepsilon N_n$ for all sufficiently large $n$.
\end{proof}

\begin{theorem}
\label{Alphaclosetolinear}
Let $\alpha(x)\to+\infty$ be a sub-polynomial of the form
\begin{equation}
\label{re9fodslroepfl}
\alpha(x)=\gamma x+\lambda\log(x)+r_1(x),
\end{equation}
for some $\gamma\geq0$, $\lambda\in\mathbb{R}$, and $r_1(x)\prec\log(x)$. Then the following are equivalent:
\begin{enumerate}
    \item\label{Alphaclosetolinear1} For all $A\subseteq\mathbb{N}$ such that $\overline{d}(A)>0$, there exist $x,y\in A$ and $n\in\mathbb{N}$ such that $x+y=\ci{\alpha(n)}$.
    \item\label{Alphaclosetolinear2} One of the following is true:
    \begin{enumerate}[label=\alph*)]
        \item $\gamma\leq1$.
        \item For some odd $a\in\N$ and some $c\in\mathbb{Z}$, we have, for all sufficiently large $n\in\N$,
\begin{equation*}
\ci{\alpha(n)}=n+\ci{\frac{n+c}{a}},
\end{equation*}
and the set $\mathbb{N}\setminus\{\ci{\alpha(n)}:n\in\N\}$ contains only finitely many even integers.
\item We have
\[
\lim_{x\to+\infty}\left|\lambda\log(x)+r_1(x)\right|=+\infty,
\qquad
1<\gamma\leq2,
\]
and
\[
\left\|\frac{\lambda\log(2)}{\gamma}\right\|
\geq 1-\frac{1}{\gamma}.
\]
    \end{enumerate}
\end{enumerate}
\end{theorem}

\begin{remark}
\label{AlphaclosetolinearTwoSets}
For sub-polynomials $\alpha(x)\to+\infty$ of the form $\alpha(x)=\gamma x+\lambda\log(x)+r_1(x)$ we have the following: for every $A_1,A_2\subseteq\N$ with $\overline d(A_1,A_2)>0$, there exist $x\in A_1$, $y\in A_2$, and $n\in\N$ such that
\[
x+y=\ci{\alpha(n)}
\]
if and only if $\gamma\leq1$. The proof is similar to that of \Cref{Alphaclosetolinear}, and we omit it.
\end{remark}

\begin{proof}[Proof of \Cref{Alphaclosetolinear}]
Set $S:=\{\ci{\alpha(n)}:n\in\N\}$.  If $\gamma\leq1$, then the set $S$ has density $1$, so condition (\ref{Alphaclosetolinear1}) holds. Therefore, we may assume $\gamma>1$. We consider two cases.
\begin{enumerate}[label=\roman*)]
\item Suppose that $
\lim_{x\to+\infty}\bigl(\lambda\log(x)+r_1(x)\bigr)$ is a constant $C\in\mathbb{R}$.
Then $\lambda=0$, and we may write
\[
\alpha(x)=\gamma x+C+r_2(x),\qquad r_2(x)\to0.
\]
We then have 
\begin{equation*}
\limsup_n\left|\ci{\alpha(n)}-\gamma n-C\right|
\leq\limsup_n\left(\left|\ci{\alpha(n)}-\alpha(n)\right|+|r_2(n)|\right)
\leq\frac{1}{2},
\end{equation*}
so multiplying by $\kappa:=1/\gamma$ and taking distance in the torus, we obtain
\begin{align}
\label{ApproxAlpha5teorid}
\limsup_{n\to\infty}\left\|\kappa\ci{\alpha(n)}-\kappa C\right\|
\leq\frac{\kappa}{2}.
\end{align}
Suppose that there is $m_0\in\N$ such that
\begin{equation}
\label{er9fdoloreitdfkl}
u:=\left\|2\kappa m_0-\kappa C\right\|>\frac{\kappa}{2}.
\end{equation}
Choose $\delta>0$ so that $u-3\delta>\kappa/2$. For all sufficiently large $M$, the set
\[
A:=\{m\in\N:m\geq M\textup{ and }\|\kappa(m-m_0)\|<\delta\}
\]
has positive density. Moreover, if $x,y\in A$, then
\[
\left\|\kappa(x+y)-\kappa C\right\|>u-2\delta>\frac{\kappa}{2}+\delta.
\]
By \Cref{ApproxAlpha5teorid}, after increasing $M$ if necessary, we may ensure that $A+A$ is disjoint from $S$. Thus, whenever \eqref{er9fdoloreitdfkl} holds for some $m_0$, condition (\ref{Alphaclosetolinear1}) fails.

If $\kappa$ is irrational, the sequence $(\{2\kappa m\})_{m\in\N}$ is dense in $[0,1)$, and hence such an $m_0$ always exists. Now suppose that $\kappa=a/b$, where $a<b$ are coprime natural numbers. Condition \eqref{er9fdoloreitdfkl} is equivalent to
\begin{equation}
\label{er9fdoloreitdfkl2}
\left\|\frac{2a}{b}m_0-\kappa C-\frac12\right\|<\frac{b-a}{2b}.
\end{equation}
That is, there exists $m_0\in\N$ satisfying \Cref{er9fdoloreitdfkl2} if and only if $U\cap V\neq\varnothing$, where 
\[
U:=\left\{\frac{2a}{b}m\textup{ mod }1;m\in\N\right\}=\left\{\frac{2}{b}m\textup{ mod }1;m\in\N\right\}\subseteq\mathbb{R}/\mathbb{Z}
\]
and $V\subseteq\mathbb{R}/\mathbb{Z}$ is the interval of length $\frac{b-a}{b}$ determined by the inequality in \eqref{er9fdoloreitdfkl2}. Note that if $b$ is odd, the set $U$ has gaps of length $1/b$; if $b$ is even, then $b-a$ is odd and $U$ has gaps of length $2/b$. It follows that \eqref{er9fdoloreitdfkl2} has a solution whenever $b-a\geq2$. Therefore, the only remaining possibility is
\[
b=a+1,
\qquad
\gamma=\frac{a+1}{a}.
\]
In this case, since $r_2(n)\to0$, there exists $c\in\mathbb{Z}$ such that, for every sufficiently large $n$,
\[
\ci{\alpha(n)}
=n+\ci{\frac{n}{a}+C+r_2(n)}
=n+\ci{\frac{n+c}{a}}.
\]
Writing $s_n:=n+\ci{(n+c)/a}$, we have $s_{n+a}=s_n+a+1$. Hence, up to finitely many elements, $S$ consists of all residue classes modulo $a+1$ except one. If $\mathbb{N}\setminus S$ contains infinitely many even integers, then the omitted residue class contains $A+A$ for some arithmetic progression $A$, and hence condition (\ref{Alphaclosetolinear1}) fails. Otherwise, $a$ is odd and $S$ contains every sufficiently large even integer, and hence condition (\ref{Alphaclosetolinear1}) holds.

\item Suppose that $
\lim_{x\to+\infty}|r(x)|=+\infty$, where $
r(x):=\lambda\log(x)+r_1(x)$.
Put
\[
\kappa:=\frac1\gamma\in(0,1),
\qquad
h(t):=\kappa r(\kappa t).
\]
As $h(x)$ does not grow faster than $\log(x)$, the following limit exists:
\[
\Delta:=\lim_{t\to+\infty} h(2t)-h(t)
=\kappa\lambda\log(2).
\]

Let $A\subseteq\N$ have positive upper density; we now find conditions on $\alpha$ that make $(A+A)\cap S=\varnothing$ impossible. Fix a sufficiently small $\delta>0$, and choose $\eta>0$ sufficiently small in terms of $\alpha$ and $\delta$, with $\eta<\delta/4$. By taking a positive density subset of $A$ if necessary, we may assume that there exist $u,v\in\mathbb{R}/\mathbb{Z}$ such that
\[
\|\kappa x-u\|<\eta,
\qquad
\|h(x)-v\|<\eta
\]
for every $x\in A$.

We may also assume that, whenever $x>y$ belong to $A$, 
\begin{equation}
\label{refd9cxoilkkl}
y<\eta x\textup{ or }y>(1-\eta)x;
\end{equation}
otherwise, take a positive density subset of $A$ by intersecting $A$ with an adequate sequence or exponentially growing intervals.

We next translate the condition $x+y=\ci{\alpha(n)}$, $x>y$, into conditions on $u$ and $v$. If $x+y=\ci{\alpha(n)}$, then $
|x+y-\alpha(n)|\leq\frac12$,
and hence
\[
\bigl|\kappa(x+y)-n-\kappa r(n)\bigr|
\leq\frac{\kappa}{2}.
\]
Therefore, 
\begin{equation}
\label{rwe9iodspdsk}
\big\|\kappa(x+y)-\kappa r(n)\bigr\|
\leq\frac{\kappa}{2},
\qquad
\big\|2u-\kappa r(n)\bigr\|
\leq\frac{\kappa}{2}+2\eta.
\end{equation}
If $y<\eta x$, then $x+y=\ci{\alpha(n)}$ implies, for all sufficiently large $x$, that
\[
\frac{x}{n}\in\bigl((1-2\eta)\gamma,(1+2\eta)\gamma\bigr).
\]
By choosing $\eta$ sufficiently small and then taking $x$ sufficiently large, we obtain
\[
\bigl|r(n)-r(x)+\lambda\log(\gamma)\bigr|<\delta
\]
and, using $h(x)=\kappa r(\kappa x)$ and $\|h(x)-v\|<\eta$,
\[
\|\kappa r(n)-v\|<\frac{\delta}{2}.
\]
Therefore, \Cref{rwe9iodspdsk} becomes
\[
\big\|2u-v\bigr\|
\leq\frac{\kappa}{2}+\delta.
\]
If on the other hand, $y>(1-\eta)x$, then similarly we obtain the equation
\[
\big\|2u-v-\Delta\bigr\|
\leq\frac{\kappa}{2}+\delta.
\]
Thus, writing $w:=2u-v$, the necessary conditions for sums of the two respective types are
\begin{equation}
\label{NecessaryConditionsForLinearSums}
\big\|w\bigr\|
\leq\frac{\kappa}{2}+\delta,
\qquad
\big\|w-\Delta\bigr\|
\leq\frac{\kappa}{2}+\delta.
\end{equation}

So, as $\delta$ may be taken arbitrarily small, if we want to find $A$ such that $(A+A)\cap S=\varnothing$ we are led to the following question: for which values of $\Delta$ and $\kappa$ does there exist $w\in\mathbb{T}$ such that
\[
\big\|w\bigr\|
>\frac{\kappa}{2}\textup{ and }\big\|w-\Delta\bigr\|
>\frac{\kappa}{2}?
\]
Such a $w$ exists if and only if
\[
\|\Delta\|<1-\kappa.
\]
Indeed, the sets $\{w\in\mathbb{T}:\|w\|>\kappa/2\}$ and $\{w\in\mathbb{T}:\|w-\Delta\|>\kappa/2\}$ are open intervals of length $1-\kappa$ whose centers are at distance $\|\Delta\|$, and they intersect precisely under the displayed condition. Notice that this is exactly the opposite of the inequality in part~(c) of condition (\ref{Alphaclosetolinear2}).

Therefore, we consider three cases:
\begin{enumerate}
\item Suppose that $\|\Delta\|>1-\kappa$, and assume for contradiction that there is some $A\subseteq\N$ such that $\overline{d}(A)>0$ and $(A+A)\cap S=\varnothing$. The arguments above show that, for every $\delta>0$, there exists $w_\delta\in\mathbb{T}$ such that
\begin{equation}
\label{NecessaryConditionsForAvoidingLinearSums}
\|w_\delta\|\geq\frac{\kappa}{2}-\delta,
\qquad
\|w_\delta-\Delta\|\geq\frac{\kappa}{2}-\delta.
\end{equation}
Therefore, by compactness, there must exist $w\in\mathbb{T}$ such that
\[
\|w\|\geq\frac{\kappa}{2},
\qquad
\|w-\Delta\|\geq\frac{\kappa}{2},
\]
which contradicts $\|\Delta\|>1-\kappa$.

\item Suppose that $\|\Delta\|<1-\kappa$. Choose $w\in\mathbb{T}$ such that
\[
\|w\|>\frac{\kappa}{2},
\qquad
\|w-\Delta\|>\frac{\kappa}{2}.
\]
Then, for every $u\in\mathbb{T}$, setting $v:=2u-w$ gives
\[
\|2u-v\|>\frac{\kappa}{2},
\qquad
\|2u-v-\Delta\|>\frac{\kappa}{2}.
\]
Choose $\delta$, and hence $\eta$, sufficiently small that
\[
\|w\|>\frac{\kappa}{2}+\delta,
\qquad
\|w-\Delta\|>\frac{\kappa}{2}+\delta,
\]
and choose $u$ in the subgroup $G_\kappa:=\overline{\{\kappa n:n\in\N\}}$ of the torus. The set
\[
B:=\left\{x\in\N:\|\kappa x-u\|<\eta,
\ \|h(x)-v\|<\eta\right\}
\]
has positive upper density. Indeed, since $h(x)\to\pm\infty$ and $h'(x)=O_\alpha(1/x)$, the condition $\|h(x)-v\|<\eta$ holds for all $x$ in some set of the form
\begin{equation*}
\N\cap\bigcup_{j}[(1-\delta_2)k_j,k_j],
\end{equation*}
for some fixed $\delta_2>0$ and $k_j\to+\infty$. And then, using the fact that the rotation $t\mapsto t+\kappa$ is uniquely ergodic on $G_\kappa$, we deduce that $\overline{d}(B)>0$. Then take a positive upper density subset $A$ of $B$ which satisfies \Cref{refd9cxoilkkl}. It then follows from the discussion following \Cref{refd9cxoilkkl} that $(A+A)\cap S=\varnothing$, so we are done.

\item Suppose that $\|\Delta\|=1-\kappa$, and assume for contradiction that there exists $A\subseteq\N$ of positive upper density such that $(A+A)\cap S=\varnothing$. Since $\kappa<1$, we have $\|\Delta\|>0$, and hence $\lambda\neq0$. Define
\[
F(n):=2\kappa n-h(n)\pmod1.
\]
Since
\[
h'(x)=\frac{\kappa\lambda}{x}+o\left(\frac1x\right),
\]
Now, by \Cref{NonAtomicLinearPhase}, there is some $\delta>0$ such that for all intervals $I\subseteq\mathbb{T}$ of length $<2\delta$, the set $\{n\in\N;F(n)\in I\}$ has upper density $<\frac{\overline{d}(A)}{3}$. Therefore, there must be
three pairwise disjoint compact intervals $I_1,I_2,I_3\subseteq\mathbb{T}$ of length $\delta$ such that, for each $i\in\{1,2,3\}$, the set
\[
 A_i:=A\cap\{n\in\N:F(n)\in I_i\}
\]
has positive upper density. Applying the arguments above to $A_1,A_2,A_3$, we obtain three distinct values $w_i\in I_i$ satisfying
\begin{equation}
\label{NecessaryConditionsForAvoidingLinearSumsEquality}
\|w_i\|\geq\frac{\kappa}{2},
\qquad
\|w_i-\Delta\|\geq\frac{\kappa}{2}.
\end{equation}
However, since $\|\Delta\|=1-\kappa$, the two inequalities in \eqref{NecessaryConditionsForAvoidingLinearSumsEquality} have at most two simultaneous solutions in $\mathbb{T}$. This contradiction completes the equality case.\qedhere
\end{enumerate}
\end{enumerate}
\end{proof}

\begin{lemma}
\label{NonAtomicLinearPhase}
Let $\kappa\in(0,1)$, let $h$ be an eventually differentiable real-valued function satisfying
\[
h'(x)=\frac{c}{x}+o\left(\frac1x\right)
\]
for some $c\neq0$, and define
\[
F(n):=2\kappa n-h(n)\pmod1,
\qquad n\in\N.
\]
For every $\varepsilon>0$, there exists $\delta>0$ such that, for every interval $J\subseteq\mathbb{T}$ of length at most $\delta$,
\[
\overline d\bigl(\{n\in\N:F(n)\in J\}\bigr)<\varepsilon.
\]
\end{lemma}

\begin{proof}[Proof of \Cref{NonAtomicLinearPhase}]
We prove the stronger estimate
\begin{equation}
\label{SmallIntervalsForLinearPhase}
\overline d\bigl(\{n\in\N:F(n)\in J\}\bigr)\leq C|J|
\end{equation}
for some $C>0$ and every interval $J\subseteq\mathbb{T}$.
If $2\kappa$ is irrational, then, for every nonzero $k\in\mathbb{Z}$, partial summation gives
\[
\sum_{n\leq N}e\bigl(kF(n)\bigr)
=\sum_{n\leq N}e(2k\kappa n)e\bigl(-kh(n)\bigr)
=O_{k,\kappa}(\log N)=o(N).
\]
By Weyl's criterion, $(F(n))_{n\in\N}$ is uniformly distributed in $\mathbb{T}$, and hence \eqref{SmallIntervalsForLinearPhase} holds with $C=1$.

If $2\kappa$ is rational, then $\kappa N=0$ for some $N\in\N$, so we can assume $\kappa=0$ by considering the intersection of $\{n\in\N:F(n)\in J\}$ with the arithmetic progressions of gap $N$. Therefore, we have
\[
F'(x)=-\frac{c}{x}+o\left(\frac1x\right)
\qquad\text{and}\qquad
F(x)=-c\log x+o(\log x)\textup{ mod }1.
\]
Therefore, the preimage $F^{-1}(J)\subseteq\mathbb{N}$ is the union of a sequence of exponentially growing intervals, with upper density 
$$\overline{d}\left(F^{-1}(J)\right)=\frac{e^{1/c}-e^{(1-|J|)/c}}{e^{1/c}-1}\ll_c|J|,$$ 
concluding the proof.
\end{proof}

\begin{remark}
In some cases, when a sub-polynomial $\alpha(x)$ does not satisfy the hypotheses of \Cref{maind}, and even if $\alpha$ is not a polynomial, we can still obtain a set $A$ with positive asymptotic density $d(A)>0$ and there are no $x,y\in A,n\in\mathbb{N}$ such that $x+y=\lfloor\alpha(n)\rfloor$. We present such an example with special part $r_\alpha(x)\succ1$: define 
\begin{align*}
\alpha(n)&:=6n^2+\log\log\left(n\right),\\
S&:=\bigcup_{N\geq1000}[6N^2+N,6N^2+2N],\\
T_{i,j}&:=\{n\in\mathbb{N};\lfloor\log\log n\rfloor\equiv i\textup{ mod }6\textup{ or }\lfloor\log\log n\rfloor\equiv j\textup{ mod }6\}\\
A&:=S\cap\left(
\left(T_{0,1}\cap(6\mathbb{Z}+4)\right)
\cup
\left(T_{2,3}\cap(6\mathbb{Z}+2)\right)
\cup
\left(T_{4,5}\cap6\mathbb{Z}\right)
\right).
\end{align*}
Then one can check that $A$ has density $d(A)=\frac{1}{6}\cdot\frac{1}{12}$, but as we now prove there are no $x,y\in A,n\in\mathbb{N}$ such that $x+y=\lfloor\alpha(n)\rfloor$. 
Indeed, assume $x<y$, and suppose for example $y\in T_{0,1}$ (the cases $y\in T_{2,3},y\in T_{4,5}$ are similar). As $y\in S$, we must have $y\in[6N^2+N,6N^2+N]$, where $N=\left\lfloor\sqrt{\frac{y}{6}}\right\rfloor$. 
If $x<N$, then $\lfloor\alpha(N)\rfloor<x+y<\lfloor\alpha(N+1)\rfloor$, a contradiction. So we must have $x\geq N$. 
But then $$\lfloor\log\log(x)\rfloor
\geq
\left\lfloor\log\log\left(N\right)\right\rfloor
\geq
\lfloor\log\log(y)\rfloor-1,$$ 
so $x$ is either in $T_{0,1}$ or $T_{4,5}$. So by definition of $A$, $y\equiv 4$ and $x\in\{4,0\}$ mod $6$, so $x+y\in\{2,4\}$ mod $6$. 
But the number $n$ such that $x+y=\lfloor\alpha(n)\rfloor$ satisfies $N\leq n\leq 2N$, so modulo $6$ we have $\lfloor\alpha(n)\rfloor\equiv\lfloor\log\log n\rfloor\in\{5,0,1\}$, a contradiction.
\end{remark}

\begin{acknowledgment}
We thank Professor Trevor D. Wooley for bringing to our attention the transference principle used in the proof of \Cref{weyl}. The large language models GPT-5.5 and GPT-5.6 were used to help with phrasing and the identification of typographical errors, as well as a slip in the proof of \Cref{Claim3eowi}. They also helped shorten the proof of \Cref{NonAtomicLinearPhase} and suggested the reference \cite{Tu}.
\end{acknowledgment}

\bibliographystyle{alpha}
\bibliography{bibliography}

\end{document}